\documentclass{article}

\usepackage[margin=1in]{geometry}
\usepackage[numbers]{natbib}

\usepackage[utf8]{inputenc} 
\usepackage[T1]{fontenc}    
\usepackage[colorlinks,linkcolor=blue,filecolor=blue,citecolor=black,urlcolor=blue]{hyperref}       
\usepackage{url}    

\usepackage{booktabs}      
\usepackage{amsfonts}       
\usepackage{nicefrac}       
\usepackage{microtype}      
\usepackage[table]{xcolor}         
\usepackage{tocvsec2}
\usepackage{titletoc}

\usepackage{xspace}

\usepackage{amsmath}
\usepackage{amsthm}
\usepackage{amssymb}

\usepackage{multicol}
\usepackage{multirow}
\usepackage{makecell}
\usepackage{enumitem}
\usepackage{wrapfig}
\usepackage{dsfont}

\usepackage{graphicx,epstopdf}
\usepackage{caption}
\usepackage{subcaption} 

\usepackage{algorithm}
\usepackage{algpseudocode}

\usepackage{threeparttable, booktabs}
\usepackage{mathtools}
\usepackage{float}

\usepackage{hyperref}[pdfusetitle]
\hypersetup{colorlinks,urlcolor=blue}
\usepackage{pifont}       % \ding{xx}
\usepackage{bbding}       % \Checkmark,\XSolid,... (需要和pifont宏包共同使用)

\usepackage{tabularray}

\usepackage[skins,listings]{tcolorbox}

\newtcblisting{tikzexample}{
    sidebyside,
    center lower,
    bicolor,
    colbacklower=white,
    sharp corners,
    frame engine=empty
}

\usepackage{quiver}

\restylefloat{algorithm} 
\newtheorem{assumption}{Assumption}

\newtheorem{remark}{Remark}
\newtheorem{lemma}{Lemma}
\newtheorem{example}{Example}
\newtheorem{proposition}{Proposition}
\newtheorem{corollary}{Corollary}

\usepackage{diagbox}

\allowdisplaybreaks

\newcommand{\E}{\mathbb{E}}

\newcommand{\norm}[1]{\left\|#1\right\|}

\newcommand{\sudamuon}{SUDA--Muon}
\newcommand{\G}{\mathbf{G}}

\title{SUDA-Muon: Structural Design Principles and Boundaries for Fully Decentralized Muon}

\author{
  Hengrui Zhang$^*$ \\
  \texttt{2022141210007@stu.scu.edu.cn} \\
  Sichuan University \\
  \and
  Boao Kong$^*$$^\dagger$\\
  \texttt{kongboao@stu.pku.edu.cn} \\
  Peking University \\
  \and
  Jiahe Geng \\
  \texttt{gengcai02@gmail.com} \\
  Peking University \\
  \and
  Zhengyang Huang\\
  \texttt{huangzy20040420@163.com} \\
  Beihang University \\
}

\begin{document}
\setlength{\parindent}{0pt}
\setlength{\parskip}{0.5em}
\maketitle

\def\thefootnote{$^*$}\footnotetext{Equal contribution.}
\def\thefootnote{$^\dagger$}\footnotetext{Corresponding author.}

\begin{abstract}
Fully decentralized Muon is difficult because its nonlinear matrix-sign operator does not commute with linear gossip averaging. This makes decentralized Muon a structural design problem: in designing the algorithm, one must distinguish modular components from non-modular ones. We propose \sudamuon{}, which realizes this separation through a unified primal--dual communication template called SUDA; within this template, ED/D$^2$, EXTRA, and gradient tracking become modular backbone choices. We prove a topology-separated non-asymptotic convergence guarantee in the nuclear-norm geometry: the dominant term scales as $\mathcal{O}((1+\sigma/\sqrt{N})K^{-1/4})$ and does not explicitly involve graph quantities, identifying the communication backbone as the modular axis in the structure design. We then establish two complementary non-modular boundaries. Internally, tracking-before-polarization is necessary for this natural no-tracking variant to avoid non-stationary fixed points under heterogeneous objectives. Externally, in the absence of a central server, a fully decentralized method cannot perform the federated average-then-polarize update; we show that this non-modular local-polarize-then-average design is the essential reason why can fail to exhibit linear speedup. Experiments on CIFAR-100 and GPT-2 fine-tuning support the same picture: the unified template makes different communication algorithms directly comparable. In mild near-IID regimes, the resulting variants perform similarly, while in the more difficult long-horizon non-IID CIFAR-100 setting, \sudamuon{} achieves higher accuracy and lower loss than \textsc{DeMuon}.
\end{abstract}



\newpage

{\small
\tableofcontents}

\allowdisplaybreaks

\newpage

\section{Introduction}
\label{sec:introduction}

Decentralized optimization studies how a network of agents can cooperatively minimize a global objective using only local computation and neighbor-to-neighbor communication, without relying on a parameter server or any other central coordinator \citep{yang2019survey,nedic2018multi,lian2017can}. This setup is attractive when data and compute are naturally spread across devices, organizations, or edge nodes, and when privacy, robustness, or communication bottlenecks make centralized aggregation undesirable \citep{blot2016gossip,daily2018gossipgrad,yuan2022decentralized}. A recurring theme in this literature is that statistical and network effects are inseparable: local objectives may be heterogeneous, stochastic gradients are noisy, and information propagates through repeated mixing over the communication graph, so the behavior of an algorithm depends not only on smoothness and noise but also on topology and on how consensus errors are corrected \citep{shi2015extra,tang2018d,pu2021distributed,alghunaim2022unified,ying2021exponential}. In this paper, we focus on the matrix-valued stochastic version of this problem, which forms the modeling backdrop for decentralized Muon.

Concretely, we consider the decentralized matrix optimization problem
\begin{equation}
    \label{eq:intro-decentralized-problem}
    \min_{X\in\mathbb{R}^{m\times n}} f(X) \;\triangleq\; \frac{1}{N}\sum_{i=1}^N f_i(X),
\end{equation}
where agent $i\in[N]\triangleq\{1,\dots,N\}$ only has access to its local objective $f_i$ and can query a stochastic first-order oracle. At iteration $k$, given a local iterate $X_i^k$, agent $i$ samples $\xi_i^k$ and obtains a stochastic gradient
\begin{equation}
    \label{eq:intro-local-oracle}
    G_i^k \;\triangleq\; \nabla F_i(X_i^k;\xi_i^k)\in\mathbb{R}^{m\times n},
\end{equation}
which satisfies standard unbiasedness and bounded-variance conditions. The agents are connected by an undirected communication graph encoded by a symmetric, doubly stochastic mixing matrix $W\in\mathbb{R}^{N\times N}$, and they can exchange messages only with their neighbors according to $W$.

Most existing decentralized optimization methods are built around classical vector-valued first-order updates, including decentralized SGD, correction-based variants such as EXTRA and D$^2$, gradient-tracking methods, and more recent decentralized adaptive gradient methods inspired by AdaGrad/Adam \citep{yang2019survey,lian2017can,shi2015extra,tang2018d,pu2021distributed,nazari2022dadam,chen2023convergence}. By contrast, decentralized algorithms for the recently emerging matrix optimizer \emph{Muon} \citep{jordan2024muon} remain much less explored. Muon replaces a coordinatewise step by a matrix polarization direction based on the matrix sign operator $\operatorname{msgn}$ and has shown strong empirical performance in centralized and federated settings \citep{jordan2024muon,liu2025muon,liu2025fedmuon,takezawa2025fedmuon}. The decentralized case is structurally more subtle because $\operatorname{msgn}$ is nonlinear, whereas decentralized communication is built from linear averaging. Once these two operations are combined, averaging and polarization need not agree, so the usual intuition from decentralized SGD no longer applies. As a result, decentralized Muon becomes a structural design problem: in designing the algorithm, one must distinguish modular components from non-modular ones.

Current decentralized Muon results do not yet provide a unified picture. DeMuon \citep{he2025demuon} combines Muon with gradient tracking and works over general graphs, but it is organized around one specific communication template. This leaves four questions open:
\begin{itemize}
    \item \textbf{Q1:} Can decentralized Muon be organized so that the communication backbone can vary while the Muon geometry stays fixed?
    \item \textbf{Q2:} Under such a decomposition, can one obtain a topology-separated rate in the Muon-induced nuclear-norm geometry, rather than an analysis in which network effects and optimizer geometry are entangled from the start?
    \item \textbf{Q3:} Is tracking before polarization genuinely necessary, or merely one convenient design choice?
    \item \textbf{Q4:} Compared with server-based/federated Muon, why can a fully decentralized Muon update not exploit the usual linear-speedup effect from averaging stochastic gradients?
\end{itemize}

To answer these questions, we introduce SUDA as a primal--dual communication template for decentralized Muon design. SUDA naturally contains existing communication frameworks such as ED/D$^2$, EXTRA, and gradient tracking, and can instantiate several standard heterogeneity-correction mechanisms while keeping the Muon update rule fixed \citep{alghunaim2022unified}. Based on this template, we propose the \sudamuon{} analysis framework. Our contributions are as follows.

\begin{itemize}
    \item 
    \textbf{C1:} We formulate decentralized Muon through a separation between matrix polarization and the decentralized backbone that mixes or tracks information before polarization. Concretely, \sudamuon{} realizes this separation through a primal--dual template parameterized by $(A,B,C)$, following the SUDA formulation, so that ED/D$^2$, EXTRA, and ATC-style gradient tracking appear as modular backbone choices while the Muon direction is kept fixed. This answers Q1.
    
    \item 
    \textbf{C2:} We prove a topology-separated non-asymptotic convergence guarantee in the Muon-induced nuclear-norm geometry. The full five-term bound is stated in Proposition~\ref{prop:sudamuon-final-convergence-K-1-4}. In particular, with suitable stepsize and EMA schedules, its dominant term scales as $\mathcal{O}((1+\sigma/\sqrt{N})K^{-1/4})$, while topology enters only through faster-decaying corrections involving the SUDA contraction rate $\gamma$ and the mixing rate $\lambda$. This identifies the communication backbone as the modular axis of the design: changing $(A,B,C)$ changes only the lower-order network corrections through $\gamma$. This answers Q2.
    
    \item 
    \textbf{C3:} We show that not every design choice is modular. As an internal structural boundary, tracking before polarization is necessary: via a smooth matrix-logistic counterexample, we prove that if each node applies $\operatorname{msgn}$ directly to its local EMA, the averaged iterate can remain fixed at a non-stationary point. This answers Q3.
    
    \item 
    \textbf{C4:} We give a second counterexample that isolates an external structural boundary between federated and fully decentralized Muon. In a simple noisy quadratic, server-side averaging before polarization yields the usual $\sigma/\sqrt{N}$ variance reduction, whereas local polarization before averaging produces an update whose contraction rate is independent of $N$, even under exact averaging on a complete graph. This answers Q4.
    
    \item 
    \textbf{C5:} We evaluate the framework on CIFAR-100 classification and GPT-2 fine-tuning. The experiments support a more specific qualitative picture: short-horizon ring benchmarks do not strongly separate decentralized Muon variants, but in the $20$-node, $100$-epoch non-IID CIFAR-100 setting both SUDA-based variants outperform \textsc{DeMuon}, with \textsc{SUDA--Muon-ED} performing best; under near-IID partitions, the gap between different SUDA--Muon instances becomes much smaller and the preferred topology changes.
    
\end{itemize}

\section{Related Work}
\label{sec:related-work}

\subsection{Muon and Other Matrix-Aware Optimizers in Centralized Training}

Recent centralized work on Muon has expanded quickly from the original optimizer proposal into a small but growing family of matrix-aware or LMO-based methods for deep learning. Unlike classical coordinatewise optimizers, and unlike matrix preconditioners such as Adafactor or Shampoo that still end with a vector-like first-order step, Muon-family methods operate directly at the level of each weight-matrix block and construct update directions through polarization, orthogonalization, or related norm-constrained LMO operations. Follow-up empirical papers study Muon's behavior in large-scale language-model training, while related variants such as Gluon and other norm-constrained LMO formulations aim to improve its practical stability and clarify its relation to a broader optimizer family \citep{shazeer2018adafactor,gupta2018shampoo,jordan2024muon,liu2025muon,riabinin2025gluon,pethick2025training}. In parallel, recent theory interprets Muon-type updates through spectral- or operator-norm geometry and analyzes their convergence behavior \citep{wang2025muon,shen2025convergence}. 

From an analysis viewpoint, however, Muon-type methods are substantially harder than standard SGD- or Adam-style optimizers for two related reasons. First, the search direction is generally not an unbiased estimator of the Euclidean gradient itself: after forming a gradient or momentum matrix, the algorithm applies a matrix polarization or orthogonalization map, so the final step is a biased transformed direction that is better understood through matrix geometry than through classical unbiased-gradient recursions. Second, this transformation is intrinsically nonlinear at the matrix level, which makes it difficult to transfer standard arguments based on linear averaging, coordinatewise decomposition, or diagonal preconditioning. Existing theory therefore has to work with operator- or spectral-norm geometry, LMO-style formulations, and more delicate progress measures than those used for conventional vector optimizers \citep{jordan2024muon,pethick2025training,riabinin2025gluon,shen2025convergence}.

\subsection{Decentralized Optimization and Heterogeneity Correction}

Classical decentralized optimization provides a large toolbox for coordinating first-order updates over graphs. Early diffusion and consensus-correction methods already showed that network mixing, local objective mismatch, and stochastic noise must be analyzed jointly rather than as separable effects \citep{chen2012diffusion,shi2015extra}. In modern decentralized learning, DSGD is the basic stochastic baseline, while Exact Diffusion/D$^2$ and gradient-tracking methods introduce correction terms to mitigate the bias caused by heterogeneity, local drift, or imperfect consensus \citep{lian2017can,tang2018d,pu2021distributed,yuan2023removing}. More recent unified analyses further characterize how topology, local updates, and communication schedules shape transient and asymptotic behavior across a common family of methods \citep{koloskova2020unified,alghunaim2022unified,zhu2024sparkle,kong2025decentralized}. SUDA is especially relevant because it places several of these methods inside one primal--dual template parameterized by $(A,B,C)$, making it possible to compare communication mechanisms without redesigning the algorithm each time \citep{alghunaim2022unified,zhu2024sparkle,kong2025decentralized}.

Most of this literature, however, works in a Euclidean setting: gradients are communicated linearly, and the optimizer itself is vector-based or at most diagonally preconditioned. The resulting theory says a great deal about topology and bias correction, but it does not address what happens when the optimizer applies a nonlinear matrix transformation before the update. For the present paper, these methods should therefore be viewed as candidate communication backbones rather than as competing optimizer geometries.

\subsection{Distributed Muon: Decentralized and Federated Extensions}

Distributed extensions of Muon are only beginning to emerge, and most of the current literature is on the federated side rather than on fully peer-to-peer graphs. In federated learning, the standard star-topology setup already separates local client optimization from server aggregation, and the interaction between data heterogeneity and averaging is known to be a central difficulty \citep{mcmahan2017communication,kairouz2021advances}. Recent FedMuon-style methods bring matrix-aware or LMO-based updates into this setting and show that naively combining local Muon with FedAvg can be unstable or biased under heterogeneity, which motivates additional server-side correction and bias control \citep{takezawa2025fedmuon,liu2025fedmuon}. These results are directly relevant because they show that Muon does not simply drop into a distributed pipeline as a black-box replacement for SGD or Adam.

DeMuon is, to our knowledge, the clearest fully decentralized Muon method on general graphs \citep{he2025demuon}. It demonstrates that combining Muon with gradient tracking is viable, but it is organized around one specific communication template. Our paper differs in emphasis along three axes. First, we separate Muon geometry from the communication backbone, rather than analyzing one fixed coupling from the outset. Second, we ask which backbone substitutions are modular inside a unified template, rather than studying only one admissible design. Third, we identify two complementary non-modular boundaries: an internal one, namely whether tracking must precede polarization, and an external one, namely whether federated Muon intuitions transfer to the fully decentralized setting.

\section{Problem Formulation and Algorithm}\label{sec:method}

\subsection{Problem Setup}

We consider the decentralized matrix optimization problem
\[
    \min_{X\in\mathbb{R}^{m\times n}}\; f(X)\;\triangleq\;\frac{1}{N}\sum_{i=1}^N f_i(X),
\]
where agent $i\in[N]\triangleq\{1,\dots,N\}$ can only access its local objective $f_i$ and query a stochastic first-order oracle.
At iteration $k$, given $X_i^k$, agent $i$ samples $\xi_i^k$ and obtains the stochastic gradient
\[
    G_i^k \;\triangleq\; \nabla F_i(X_i^k;\xi_i^k)\in\mathbb{R}^{m\times n}.
\]

\subsection{Algorithm Description}

Our design principle is to separate two roles that are often entangled in decentralized Muon variants. The first role is the \emph{decentralized backbone}: the algorithm should use communication and, when needed, tracking to construct the signal to be polarized. The second role is \emph{matrix geometry}: Muon's nonlinear operator $\operatorname{msgn}$ should act on that backbone-processed signal rather than on a purely local one. A unified primal--dual communication template is a natural way to realize this separation, because it lets the communication backbone vary while leaving the polarization rule unchanged.

In this paper, we instantiate that template using the stochastic unified decentralized algorithm (SUDA) formulation. SUDA is used only as a unified communication template. It is not the source of Muon's geometry. The role of the matrices $(A,B,C)$ is simply to parameterize how decentralized mixing and tracking are performed, while the Muon direction is always the matrix polarization $\operatorname{msgn}(\cdot)$. In the language of this paper, the unified template is used to study the modular axis of the design space; the two non-modular boundaries are established separately by structural counterexamples.

At a high level, each node maintains four coupled quantities: the local primal iterate $X_i^k$, an EMA state $M_i^k$ for stochastic gradients, a tracked signal $H_i^k$ that aggregates network-corrected gradient information, and a dual variable $Y_i^k$ that accumulates consensus violations in the SUDA backbone. Each iteration first updates the EMA, then updates the tracked signal through communication, next applies the matrix-sign operator to obtain the Muon direction, and finally performs the SUDA primal--dual correction step.

\begin{algorithm}[t]
\caption{\sudamuon{} (stacked form)}
\label{alg:suda-muon}
\begin{algorithmic}[1]
\State \textbf{Input:} stepsize $\alpha>0$, EMA parameter $\beta\in[0,1)$, mixing / SUDA matrices $(W,A,B,C)$.
\State \textbf{Initialize:} choose $\mathbf{X}^0=\mathrm{col}\{X_1^0,\ldots,X_N^0\}$, set $\mathbf{Y}^0=0$.
\State Query stochastic gradients $\mathbf{G}^0=\mathrm{col}\{G_1^0,\ldots,G_N^0\}$ at $\mathbf{X}^0$.
\State Set $\mathbf{M}^0:=\mathbf{G}^0$ and $\mathbf{H}^0:=\mathbf{M}^0$.
\For{$k=0,1,2,\ldots$}
    \If{$k\ge 1$}
        \State Query stochastic gradients $\mathbf{G}^k=\mathrm{col}\{G_1^k,\ldots,G_N^k\}$ at $\mathbf{X}^k$.
    \EndIf
    \State $\mathbf{M}^{k+1} \gets \beta\mathbf{M}^{k} + (1-\beta)\,\mathbf{G}^{k}$ \hfill (EMA)
    \State $\mathbf{H}^{k+1} \gets W\bigl(\mathbf{H}^k + \mathbf{M}^{k+1}-\mathbf{M}^{k}\bigr)$ \hfill (gradient tracking)
    \State $\mathbf{X}^{k+1} \gets A\bigl(C\mathbf{X}^k - \alpha\,\operatorname{msgn}(\mathbf{H}^{k+1})\bigr) - B\mathbf{Y}^k$ \hfill (SUDA backbone)
    \State $\mathbf{Y}^{k+1} \gets \mathbf{Y}^k + B\mathbf{X}^{k+1}$
\EndFor
\end{algorithmic}
\end{algorithm}

The framework is parameterized by the SUDA communication matrices $(A,B,C)$, which are low-degree polynomials of the mixing matrix $W$.
Different choices recover well-known decentralized schemes as special cases. In the language of this paper, these are \emph{modular backbone choices}: they change how communication correction is carried out, but not the Muon geometry itself.
\begin{itemize}
\item \textbf{ED/D$^2$:} $A=C=W$ and $B=(I-W^2)^{\frac{1}{2}}$.
\item \textbf{EXTRA:} $A=C=\frac{I+W}{2}$ and $B^2=(\frac{I-W}{2})^{\frac{1}{2}}$.
\item \textbf{Gradient Tracking (ATC-GT):} $A=C=W$ and $B=I-W$.
\end{itemize}
The point is not that these communication mechanisms are new. Rather, \sudamuon{} makes them comparable within a single decentralized-Muon architecture while keeping the polarization rule fixed.
\begin{remark}
DeMuon \citep{he2025demuon} uses a closely related gradient-tracking mechanism combined with Muon directions, but applies mixing directly to the primal variables rather than through the SUDA primal--dual backbone. Specifically, DeMuon does not maintain the dual variable $\mathbf{Y}^k$ and instead performs a gossip-style averaging step on $\mathbf{X}^k$ after the Muon update. This subtle structural difference means DeMuon is not recovered as a special case of Algorithm~\ref{alg:suda-muon}, but its gradient-tracking component shares the same motivations analyzed in our convergence theory.
\end{remark}
In all cases, the Muon polarization direction $\operatorname{msgn}(\mathbf{H}^{k+1})$ is applied \emph{after} gradient tracking, so that the $\operatorname{msgn}$ operator acts on a network-aggregated quantity rather than a purely local signal.
The stacked update equations are:
\begin{equation}
\label{eq:sudamuon-framework}
\begin{aligned}
\mathbf{X}^{k+1} &= A\bigl(C\mathbf{X}^k - \alpha\,\operatorname{msgn}(\mathbf{H}^{k+1})\bigr) - B\mathbf{Y}^k,\\
\mathbf{Y}^{k+1} &= \mathbf{Y}^k + B\mathbf{X}^{k+1},\\
\mathbf{M}^{k+1} &= \beta\mathbf{M}^{k} + (1-\beta)\,\mathbf{G}^{k},\\
\mathbf{H}^{k+1} &= W\bigl(\mathbf{H}^k + \mathbf{M}^{k+1}-\mathbf{M}^{k}\bigr),
\end{aligned}
\end{equation}

\subsection{Assumptions}

We adopt standard assumptions from refined SUDA analyses and recent Muon/DeMuon convergence analyses, adapted to matrix-valued parameters.

\begin{assumption}[Network and SUDA communication structure]
    \label{assump:network-suda}
    Let $W\in\mathbb{R}^{N\times N}$ be the (undirected) mixing matrix associated with the communication graph.
    We assume:
    \begin{enumerate}
        \vspace{-1mm}
        \item $W$ is symmetric, doubly stochastic, and primitive. Denote the mixing rate
        \[
            \lambda \;\triangleq\; \rho\Big(W-\tfrac{1}{N}\mathbf{1}_N\mathbf{1}_N^\top\Big)\in(0,1).
        \]
        \vspace{-1mm}
        \item The SUDA matrices $A,B,C\in\mathbb{R}^{N\times N}$ satisfy:
        (i) $A$ and $C$ are doubly stochastic; (ii) $BZ=0$ iff $Z_1=\cdots=Z_N$ for stacked variables;
        (iii) $A$, $B^2$, and $C$ are (low-degree) polynomial functions of $W$.
    \end{enumerate}
\end{assumption}

\begin{assumption}[Matrix smoothness and lower boundedness]
    \label{assump:smooth}
    Each local objective $f_i:\mathbb{R}^{m\times n}\to\mathbb{R}$ is $L_*$-smooth in the sense that for all $X,Y\in\mathbb{R}^{m\times n}$,
    \[
        \|\nabla f_i(X)-\nabla f_i(Y)\|_* \;\le\; L_* \|X-Y\|,
        \qquad \forall i\in[N],
    \]
    where $\|\cdot\|_*$ and $\|\cdot\|$ denote the nuclear norm and operator norm, respectively.
    Moreover, the global objective $f(X)\triangleq\frac{1}{N}\sum_{i=1}^N f_i(X)$ is bounded below: $f(X)\ge f_{\inf}>-\infty$.
\end{assumption}

\begin{assumption}[Stochastic first-order oracle]
    \label{assump:stoch}
    At each iteration $k$ and node $i$, the oracle returns
    $G_i^k=\nabla F_i(X_i^k;\xi_i^k)$ satisfying
    \[
        \mathbb{E}\big[G_i^k\mid\mathcal{F}^k\big]=\nabla f_i(X_i^k),\qquad
        \mathbb{E}\big[\|G_i^k-\nabla f_i(X_i^k)\|_F^2\mid\mathcal{F}^k\big]\le \sigma^2,
    \]
    where $\mathcal{F}^k$ is the natural filtration of the algorithm.
    Stochastic gradients are conditionally independent across agents and time given $\mathcal{F}^k$.
\end{assumption}

\begin{assumption}[Muon polarization operator]
    \label{assump:muon}
    For any $H\in\mathbb{R}^{m\times n}$ with a reduced SVD $H=U\Sigma V^\top$, define
    \begin{equation}
    \label{eq:msgn}
        \operatorname{msgn}(H)\triangleq UV^\top,\qquad \operatorname{msgn}(0)\triangleq 0.
    \end{equation}
    Each agent uses $S_i^{k+1}=\operatorname{msgn}(H_i^{k+1})$ as the Muon direction, computed locally
    (e.g., via an exact SVD or approximate Newton--Schulz iterations).
\end{assumption}

\section{Theoretical Analysis and Structural Consequences}\label{sec:theory}

We present the theoretical analysis of \sudamuon{}.
For the reader's convenience, we first state the main convergence result and its big-$\mathcal{O}$ simplification in Subsection~\ref{subsec:main-statement}, then establish the proof ingredients and final assembly in Subsection~\ref{subsec:technical-analysis}. Subsections~\ref{sec:discussion}--\ref{sec:linear-speedup} then draw three structural consequences of this theory: one modular axis, namely the choice of communication backbone, and two non-modular boundaries, namely the internal order between tracking and polarization and the external difference between decentralized and federated Muon.

%% ================================================================
\subsection{Statement of Main Results}
\label{subsec:main-statement}
%% ================================================================

The following proposition provides the explicit convergence rate for \sudamuon{}. Its role in the paper is not only to certify convergence, but also to expose the structural decomposition behind decentralized Muon: the leading stochastic term reflects the Muon geometry itself, whereas the network-dependent corrections isolate the effect of the communication backbone.

\begin{proposition}[Main convergence rate --- informal statement]
\label{prop:main-informal}
Under Assumptions~\ref{assump:network-suda}--\ref{assump:muon} and standard SUDA stability conditions, choose the stepsize $\alpha=\alpha_0 K^{-3/4}$ and EMA parameter $1-\beta=b_0 K^{-1/2}$.
Then the average nuclear-norm stationarity measure of \sudamuon{} satisfies
\begin{equation}
\label{eq:main-rate-preview}
\frac{1}{K}\sum_{k=0}^{K-1}\mathbb{E}\big[\|\nabla f(\bar X^k)\|_*\big]
\;=\;\mathcal{O}\!\left(
\underbrace{\Bigl(1+\frac{\sigma}{\sqrt{N}}\Bigr)K^{-1/4}}_{\text{topology-free}}
+\underbrace{\frac{K^{-3/4}}{(1-\gamma)(1-\lambda)}}_{\text{network-dependent}}
+\text{lower-order terms}
\right),
\end{equation}
where $\gamma\in(0,1)$ is the SUDA contraction rate and $\lambda=\rho(W-J)\in(0,1)$ is the mixing rate.
\end{proposition}

The precise five-term bound, together with all constants, is given in Proposition~\ref{prop:sudamuon-final-convergence-K-1-4} and its big-$\mathcal{O}$ simplification in Corollary~\ref{cor:sudamuon-final-convergence-bigO}.
Three structural features of this result deserve emphasis:
\begin{enumerate}
\item \textbf{Topology separation.}
The leading term $(1+\sigma/\sqrt{N})K^{-1/4}$ is entirely free of any network quantity; the communication topology affects only the faster-decaying correction terms.
\item \textbf{Modular network dependence.}
Network effects enter only through $(1-\gamma)$ and $(1-\lambda)$, and different choices of $(A,B,C)$ alter only $\gamma$, so one can compare communication strategies by simply substituting the corresponding SUDA contraction rate (see Subsection~\ref{sec:discussion}).
\item \textbf{Muon-induced geometry.}
The stationarity measure is the nuclear norm $\|\nabla f\|_*$, which is the natural dual norm for Muon's operator-norm constraint set. This is fundamentally different from the Euclidean norm used in standard decentralized SGD analyses.
\end{enumerate}

The proof proceeds in three stages: we first derive a one-step descent inequality in the nuclear-norm geometry, then establish the auxiliary bounds controlling EMA error, gradient tracking, SUDA disagreement, and forcing terms, and finally assemble these ingredients via telescoping and the chosen stepsize/EMA schedules.

%% ================================================================
\subsection{Basic Properties of the Matrix Sign Operator}
\label{subsec:msgn-basic}
%% ================================================================

\begin{lemma}[Basic properties of $\operatorname{msgn}$ and a robust inner-product lower bound]
\label{lem:sudamuon-msgn-basic}
Let $H\in\mathbb{R}^{m\times n}$ and define $\operatorname{msgn}(H)$ as in~\eqref{eq:msgn}: if $H\neq 0$ has reduced SVD $H=U\Sigma V^{\top}$ with $U\in\mathbb{R}^{m\times r}$, $V\in\mathbb{R}^{n\times r}$, $\Sigma=\mathrm{diag}(\sigma_1,\dots,\sigma_r)$ and $\sigma_i>0$, then $\operatorname{msgn}(H):=UV^{\top}$, and set $\operatorname{msgn}(0):=0$.
Then the following hold.
\begin{enumerate}
\item \textbf{Alignment with the nuclear norm.} One has
\begin{equation}
\langle H,\operatorname{msgn}(H)\rangle = \|H\|_*.
\end{equation}
\item \textbf{Operator- and Frobenius-norm bounds.} One has
\begin{equation}
\|\operatorname{msgn}(H)\|\le 1,\qquad \|\operatorname{msgn}(H)\|_F = \sqrt{\mathrm{rank}(H)}\le \sqrt{r_{\max}},\qquad r_{\max}:=\min\{m,n\}.
\end{equation}
Moreover, if $H\neq 0$ then $\|\operatorname{msgn}(H)\|=1$.
\item \textbf{A key inequality (no Lipschitzness required).} For any $G,H\in\mathbb{R}^{m\times n}$,
\begin{equation}
\label{eq:msgn-robust-ip-lb}
\langle G,\operatorname{msgn}(H)\rangle \ge \|G\|_* - 2\|G-H\|_*.
\end{equation}
\end{enumerate}
The proof uses only the definition~\eqref{eq:msgn}, basic linear algebra, and the standard nuclear/operator norm duality inequality
\begin{equation}
\label{eq:nuc-op-duality}
|\langle A,B\rangle| \le \|A\|_*\,\|B\|,\qquad \forall A,B\in\mathbb{R}^{m\times n},
\end{equation}
where $\langle A,B\rangle:=\mathrm{tr}(A^{\top}B)$.
\end{lemma}

\begin{proof}
If $H=0$, then $\operatorname{msgn}(H)=0$ by definition and all claims are immediate. Hence assume $H\neq 0$ and write its reduced SVD as $H=U\Sigma V^{\top}$ with $U^{\top}U=I_r$ and $V^{\top}V=I_r$.

\medskip\noindent\textbf{Part (1): $\langle H,\operatorname{msgn}(H)\rangle=\|H\|_*$.}
Using cyclicity of the trace and $U^{\top}U=I_r$, $V^{\top}V=I_r$,
\[
\langle H,\operatorname{msgn}(H)\rangle
=\mathrm{tr}\bigl((U\Sigma V^{\top})^{\top}(UV^{\top})\bigr)
=\mathrm{tr}\bigl(V\Sigma U^{\top}U V^{\top}\bigr)
=\mathrm{tr}(V\Sigma V^{\top})
=\mathrm{tr}(\Sigma)
=\sum_{i=1}^r \sigma_i
=\|H\|_*.
\]

\medskip\noindent\textbf{Part (2): operator- and Frobenius-norm bounds.}
Set $M:=\operatorname{msgn}(H)=UV^{\top}$. Then
\[
M^{\top}M = VU^{\top}UV^{\top} = VV^{\top},
\]
which is the orthogonal projector onto $\mathrm{col}(V)$ and hence has eigenvalues in $\{0,1\}$. Therefore
\[
\|M\|^2 = \lambda_{\max}(M^{\top}M) = \lambda_{\max}(VV^{\top}) = 1,
\]
so $\|\operatorname{msgn}(H)\|=1$ for $H\neq 0$, and the bound $\|\operatorname{msgn}(H)\|\le 1$ holds in general (including $H=0$).
Moreover,
\[
\|M\|_F^2 = \mathrm{tr}(M^{\top}M) = \mathrm{tr}(VV^{\top}) = \mathrm{rank}(VV^{\top}) = r = \mathrm{rank}(H),
\]
so $\|\operatorname{msgn}(H)\|_F=\sqrt{\mathrm{rank}(H)}\le \sqrt{\min\{m,n\}}=: \sqrt{r_{\max}}$.

\medskip\noindent\textbf{Part (3): robust inner-product lower bound.}
Write
\[
\langle G,\operatorname{msgn}(H)\rangle
= \langle H,\operatorname{msgn}(H)\rangle + \langle G-H,\operatorname{msgn}(H)\rangle.
\]
By Part~(1), $\langle H,\operatorname{msgn}(H)\rangle=\|H\|_*$. For the second term, apply the duality inequality~\eqref{eq:nuc-op-duality} with $A=G-H$ and $B=\operatorname{msgn}(H)$, together with Part~(2) ($\|\operatorname{msgn}(H)\|\le 1$), to obtain
\[
\langle G-H,\operatorname{msgn}(H)\rangle \ge -|\langle G-H,\operatorname{msgn}(H)\rangle|
\ge -\|G-H\|_*\,\|\operatorname{msgn}(H)\|
\ge -\|G-H\|_*.
\]
Consequently,
\begin{equation}
\label{eq:ip-lb-step1}
\langle G,\operatorname{msgn}(H)\rangle \ge \|H\|_* - \|G-H\|_*.
\end{equation}
Finally, by the triangle inequality for the nuclear norm,
\[
\|G\|_* = \|H + (G-H)\|_* \le \|H\|_* + \|G-H\|_*,
\]
which rearranges to $\|H\|_* \ge \|G\|_* - \|G-H\|_*$. Plugging this into~\eqref{eq:ip-lb-step1} yields
\[
\langle G,\operatorname{msgn}(H)\rangle \ge (\|G\|_* - \|G-H\|_*) - \|G-H\|_* = \|G\|_* - 2\|G-H\|_*,
\]
which is~\eqref{eq:msgn-robust-ip-lb}.
\end{proof}

%% ================================================================
\subsection{Proof Ingredients and Main Proofs}
\label{subsec:technical-analysis}
%% ================================================================

This subsection develops the technical ingredients behind the master rate and then assembles them into the final convergence bound.

\subsubsection*{One-Step Descent Analysis}

\begin{lemma}[Averaged SUDA--Muon dynamics]
\label{lem:sudamuon-average-dynamics}
Suppose Assumption~\ref{assump:network-suda} holds and the iterates are generated by~\eqref{eq:sudamuon-framework} with initialization $Y_i^0=0$ for all $i\in[N]$.
Define the averages
\[
    \bar X^k \;\triangleq\; \frac{1}{N}\sum_{i=1}^N X_i^k,
    \qquad
    \bar S^{k+1} \;\triangleq\; \frac{1}{N}\sum_{i=1}^N S_i^{k+1},\quad S_i^{k+1}=\operatorname{msgn}(H_i^{k+1}).
\]
Then, for all $k\ge 0$,
\begin{equation}
\label{eq:avg-dynamics-barX}
    \bar X^{k+1} \,=\, \bar X^k - \alpha\,\bar S^{k+1}.
\end{equation}
\end{lemma}

\begin{proof}
Let $\mathbf{1}\in\mathbb{R}^N$ denote the all-ones vector and set
\[
    J \;\triangleq\; \frac{1}{N}\,\mathbf{1}\mathbf{1}^\top \in\mathbb{R}^{N\times N}.
\]
For a stacked variable $\mathbf{Z}=\mathrm{col}\{Z_1,\dots,Z_N\}$ (with matrix blocks $Z_i\in\mathbb{R}^{m\times n}$),
$J\mathbf{Z}$ has all blocks equal to the average $\bar Z\triangleq\frac{1}{N}\sum_i Z_i$.

From~\eqref{eq:sudamuon-framework},
\[
    \mathbf{X}^{k+1} \,=\, A\Bigl(C\mathbf{X}^{k} - \alpha\,\operatorname{msgn}(\mathbf{H}^{k+1})\Bigr) - B\mathbf{Y}^{k}.
\]
Left-multiplying by $J$ and using that $A$ and $C$ are doubly stochastic (Assumption~\ref{assump:network-suda}(2)(i)), hence $JA=J$ and $JC=J$, yields
\begin{equation}
\label{eq:avg-dynamics-JX}
    J\mathbf{X}^{k+1} \,=\, J\mathbf{X}^{k} - \alpha\,J\operatorname{msgn}(\mathbf{H}^{k+1}) - JB\mathbf{Y}^{k}.
\end{equation}
We claim that $JB\mathbf{Y}^k=0$ for all $k\ge 0$ under $\mathbf Y^0=0$.
Indeed, define the row vector $b^\top\triangleq \tfrac{1}{N}\mathbf 1^\top B$ and the stacked linear functional
$u^k\triangleq b^\top \mathbf Y^k\in\mathbb{R}^{m\times n}$.
From the dual update $\mathbf Y^{k+1}=\mathbf Y^k + B\mathbf X^{k+1}$,
\[
    u^{k+1}=u^k + b^\top B\mathbf X^{k+1} = u^k + \tfrac{1}{N}\mathbf 1^\top B^2\mathbf X^{k+1}.
\]
By Assumption~\ref{assump:network-suda}(2)(ii), for every \emph{consensus} stacked variable
$\mathbf Z=\mathrm{col}\{Z,\dots,Z\}$ (with an arbitrary block $Z\in\mathbb R^{m\times n}$), we have $B\mathbf Z=0$.
Writing the $i$-th block of $B\mathbf Z$ as $(B\mathbf Z)_i=\sum_{j=1}^N b_{ij}Z$, we obtain
$(\sum_{j=1}^N b_{ij})Z=0$ for every $Z$.
Choosing any nonzero $Z$ (e.g., a basis matrix) implies $\sum_{j=1}^N b_{ij}=0$ for all $i$, i.e.,
\begin{equation}
\label{eq:avg-dynamics-B1}
B\mathbf 1=0.
\end{equation}
Consequently, $B^2\mathbf 1=0$.

Next, by Assumption~\ref{assump:network-suda}(2)(iii), there exists a polynomial $p$ such that $B^2=p(W)$.
Since $W$ is doubly stochastic, $W\mathbf 1=\mathbf 1$ and $\mathbf 1^\top W=\mathbf 1^\top$, hence
$p(W)\mathbf 1=p(1)\mathbf 1$ and $\mathbf 1^\top p(W)=p(1)\mathbf 1^\top$.
From $B^2\mathbf 1=p(W)\mathbf 1=0$ we conclude $p(1)=0$, and therefore
\begin{equation}
\label{eq:avg-dynamics-1TB2}
\mathbf 1^\top B^2=\mathbf 1^\top p(W)=p(1)\mathbf 1^\top=0.
\end{equation}
Therefore $b^\top B=\tfrac{1}{N}\mathbf 1^\top B^2=0$, and thus $u^{k+1}=u^k$.
Since $\mathbf Y^0=0$, we get $u^0=0$ and hence $u^k\equiv 0$ for all $k$.
Finally,
\[
JB\mathbf Y^k=\frac{1}{N}\mathbf 1\,(\mathbf 1^\top B\mathbf Y^k)=\mathbf 1\,u^k=0.
\]
Substituting $JB\mathbf Y^k=0$ into~\eqref{eq:avg-dynamics-JX} and taking the common block yields
$\bar X^{k+1}=\bar X^k-\alpha\bar S^{k+1}$, which is~\eqref{eq:avg-dynamics-barX}.
\end{proof}

\begin{lemma}[One-step descent of $f(\bar X^k)$ with Muon direction and tracking error]
\label{lem:sudamuon-one-step-descent}
Suppose Assumptions~\ref{assump:network-suda} and~\ref{assump:smooth} hold, the iterates are generated by~\eqref{eq:sudamuon-framework}, and $Y_i^0=0$ for all $i\in[N]$ (as in Algorithm~\ref{alg:suda-muon}).
Let $(\bar X^k)_{k\ge 0}$ be the averaged iterates defined in Lemma~\ref{lem:sudamuon-average-dynamics}.
Then, for all $k\ge 0$,
\begin{equation}
\label{eq:one-step-descent}
 f(\bar X^{k+1}) \,\le\, f(\bar X^k)
 - \alpha\,\|\nabla f(\bar X^k)\|_* + \frac{L_*\alpha^2}{2}
 + \frac{2\alpha}{N}\sum_{i=1}^N \|\nabla f(\bar X^k) - H_i^{k+1}\|_*.
\end{equation}
Here and throughout, $\langle A,B\rangle\!:=\!\mathrm{tr}(A^{\top}B)$ denotes the matrix inner product.
\end{lemma}

\begin{proof}
\medskip\noindent\textbf{Step 1: a descent lemma under Assumption~\ref{assump:smooth}.}
Since $f(X)=\frac{1}{N}\sum_{i=1}^N f_i(X)$, Assumption~\ref{assump:smooth} implies that for all $X,Y$,
\begin{equation}
\label{eq:f-smoothness}
\|\nabla f(X)-\nabla f(Y)\|_* 
\;=\;\Big\|\frac{1}{N}\sum_{i=1}^N\bigl(\nabla f_i(X)-\nabla f_i(Y)\bigr)\Big\|_*
\;\le\;\frac{1}{N}\sum_{i=1}^N \|\nabla f_i(X)-\nabla f_i(Y)\|_*
\;\le\; L_*\|X-Y\|.
\end{equation}
Fix $X,Y\in\mathbb{R}^{m\times n}$ and let $D:=Y-X$.
Define $g(t):=f(X+tD)$ for $t\in[0,1]$. Then $g$ is differentiable and
\[ 
 g'(t)=\langle \nabla f(X+tD), D\rangle.
\]
Using the fundamental theorem of calculus,
\begin{equation}
\label{eq:FTC}
 f(Y)-f(X) 
 = g(1)-g(0)
 = \int_0^1 \langle \nabla f(X+tD), D\rangle\,dt.
\end{equation}
Add and subtract $\nabla f(X)$ inside the integrand to obtain
\[
\begin{aligned}
 f(Y)-f(X)
 &= \langle \nabla f(X),D\rangle + \int_0^1 \langle \nabla f(X+tD)-\nabla f(X), D\rangle\,dt.
\end{aligned}
\]
By the nuclear/operator duality inequality $|\langle A,B\rangle|\le \|A\|_*\,\|B\|$ and~\eqref{eq:f-smoothness}, for each $t\in[0,1]$,
\[
\begin{aligned}
\langle \nabla f(X+tD)-\nabla f(X), D\rangle
&\le |\langle \nabla f(X+tD)-\nabla f(X), D\rangle| \\
&\le \|\nabla f(X+tD)-\nabla f(X)\|_*\,\|D\| \\
&\le L_*\,\|tD\|\,\|D\| 
= L_*\,t\,\|D\|^2.
\end{aligned}
\]
Plugging this into~\eqref{eq:FTC} gives the descent inequality
\begin{equation}
\label{eq:descent-lemma-op}
 f(Y)\le f(X)+\langle \nabla f(X),Y-X\rangle + \frac{L_*}{2}\|Y-X\|^2.
\end{equation}

\medskip\noindent\textbf{Step 2: apply~\eqref{eq:descent-lemma-op} to the averaged recursion.}
By Lemma~\ref{lem:sudamuon-average-dynamics},
$\bar X^{k+1}=\bar X^k-\alpha\bar S^{k+1}$ with $\bar S^{k+1}=\frac{1}{N}\sum_{i=1}^N S_i^{k+1}$ and $S_i^{k+1}=\operatorname{msgn}(H_i^{k+1})$.
Applying~\eqref{eq:descent-lemma-op} with $X=\bar X^k$ and $Y=\bar X^{k+1}$ yields
\begin{equation}
\label{eq:descent-step1}
 f(\bar X^{k+1})
 \le f(\bar X^k) - \alpha\,\langle \nabla f(\bar X^k),\bar S^{k+1}\rangle
 + \frac{L_*\alpha^2}{2}\,\|\bar S^{k+1}\|^2.
\end{equation}
Moreover, by Lemma~\ref{lem:sudamuon-msgn-basic}(2), $\|S_i^{k+1}\|\le 1$ for all $i$, and by convexity of the operator norm,
\begin{equation}
\label{eq:barS-op-bound}
 \|\bar S^{k+1}\|
 = \Big\|\frac{1}{N}\sum_{i=1}^N S_i^{k+1}\Big\|
 \le \frac{1}{N}\sum_{i=1}^N \|S_i^{k+1}\|
 \le 1.
\end{equation}
Combining~\eqref{eq:descent-step1} and~\eqref{eq:barS-op-bound} gives
\begin{equation}
\label{eq:descent-step2}
 f(\bar X^{k+1})
 \le f(\bar X^k) - \alpha\,\langle \nabla f(\bar X^k),\bar S^{k+1}\rangle
 + \frac{L_*\alpha^2}{2}.
\end{equation}

\medskip\noindent\textbf{Step 3: lower bound $\langle \nabla f(\bar X^k),\bar S^{k+1}\rangle$ via Lemma~\ref{lem:sudamuon-msgn-basic}.}
Let $G:=\nabla f(\bar X^k)$. For each agent $i$, Lemma~\ref{lem:sudamuon-msgn-basic}(3) with $H=H_i^{k+1}$ gives
\[
 \langle G, S_i^{k+1}\rangle
 = \langle G,\operatorname{msgn}(H_i^{k+1})\rangle
 \ge \|G\|_* - 2\,\|G-H_i^{k+1}\|_*.
\]
Averaging over $i$ and using $\langle G,\bar S^{k+1}\rangle=\frac{1}{N}\sum_{i=1}^N \langle G,S_i^{k+1}\rangle$ yields
\begin{equation}
\label{eq:ip-lb-average}
 \langle \nabla f(\bar X^k),\bar S^{k+1}\rangle
 \ge \|\nabla f(\bar X^k)\|_* - \frac{2}{N}\sum_{i=1}^N \|\nabla f(\bar X^k)-H_i^{k+1}\|_*.
\end{equation}
Substituting~\eqref{eq:ip-lb-average} into~\eqref{eq:descent-step2} gives~\eqref{eq:one-step-descent}.
\end{proof}

\begin{proposition}[Descent inequality with an explicit ``tracking-error'' remainder]
\label{prop:telescoping-descent}
Suppose the assumptions of Lemma~\ref{lem:sudamuon-one-step-descent} hold.
Then for any integer $K\ge 1$,
\begin{equation}
\label{eq:telescoping-descent}
\frac{1}{K}\sum_{k=0}^{K-1}\mathbb{E}\big[\|\nabla f(\bar X^k)\|_*\big]
\;\le\;
\frac{\Delta_0}{\alpha K}
+\frac{L_*\alpha}{2}
+\frac{2}{NK}\sum_{k=0}^{K-1}\sum_{i=1}^N\mathbb{E}\big[\|\nabla f(\bar X^k)-H_i^{k+1}\|_*\big],
\end{equation}
where $\Delta_0:=\mathbb{E}[f(\bar X^0)]-f_{\inf}$.
\end{proposition}

\begin{proof}
Sum the one-step inequality \eqref{eq:one-step-descent} from Lemma~\ref{lem:sudamuon-one-step-descent} over $k=0,\dots,K-1$, take expectations, and use $\mathbb{E}[f(\bar X^K)]\ge f_{\inf}$.
Divide by $\alpha K$.
\end{proof}

%% ================================================================
\subsubsection*{EMA and Gradient Tracking Foundations}
%% ================================================================

\begin{lemma}[EMA stochastic error (per-agent and averaged)]
\label{lem:sudamuon-ema-stoch-error}
Suppose Assumption~\ref{assump:stoch} holds and the iterates are generated by~\eqref{eq:sudamuon-framework} (equivalently, Algorithm~\ref{alg:suda-muon}).
Fix any agent $i\in[N]$.
Define $C_i^0:=\nabla f_i(X_i^0)$ and for $k\ge 0$,
\[
C_i^{k+1}:=\beta C_i^k+(1-\beta)\nabla f_i(X_i^k),
\]
and recall that $M_i^0=G_i^0$ and $M_i^{k+1}=\beta M_i^k+(1-\beta)G_i^k$.
Then, for all $k\ge 0$,
\begin{equation}
\label{eq:ema-stoch-error-agent}
\mathbb{E}\,\|C_i^k-M_i^k\|_F
\;\le\;
\sqrt{\tfrac{1-\beta}{1+\beta}}\,\sigma \, +\, \beta^k\sigma.
\end{equation}
Moreover, defining the averages $\bar M^k:=\tfrac{1}{N}\sum_{i=1}^N M_i^k$ and $\bar C^k:=\tfrac{1}{N}\sum_{i=1}^N C_i^k$, we have for all $k\ge 0$,
\begin{equation}
\label{eq:ema-stoch-error-average}
\mathbb{E}\,\|\bar C^k-\bar M^k\|_F
\;\le\;
\sqrt{\tfrac{1-\beta}{1+\beta}}\,\tfrac{\sigma}{\sqrt{N}} \, +\, \beta^k\tfrac{\sigma}{\sqrt{N}}.
\end{equation}
Finally, letting $r_{\max}:=\min\{m,n\}$, for all $k\ge 0$,
\begin{equation}
\label{eq:ema-stoch-error-nuc}
\mathbb{E}\,\|C_i^k-M_i^k\|_* \;\le\; \sqrt{r_{\max}}\,\mathbb{E}\,\|C_i^k-M_i^k\|_F.
\end{equation}
\end{lemma}

\begin{proof}
Throughout, define the stochastic gradient noise
\[
    \Delta_i^k\;:=\;G_i^k-\nabla f_i(X_i^k).
\]
Then Assumption~\ref{assump:stoch} implies
\begin{equation}
\label{eq:delta-mds}
\mathbb{E}[\Delta_i^k\mid \mathcal{F}^k]=0,\qquad
\mathbb{E}[\|\Delta_i^k\|_F^2\mid \mathcal{F}^k]\le \sigma^2.
\end{equation}

\medskip\noindent\textbf{Step 1: per-agent EMA error.}
Let $E_i^k:=C_i^k-M_i^k$.
By the definitions of $C_i^{k+1}$ and $M_i^{k+1}$,
\begin{equation}
\label{eq:E-rec}
E_i^{k+1}=\beta E_i^k + (1-\beta)\bigl(\nabla f_i(X_i^k)-G_i^k\bigr)
=\beta E_i^k-(1-\beta)\Delta_i^k.
\end{equation}
Unrolling~\eqref{eq:E-rec} gives, for any $k\ge 0$,
\begin{equation}
\label{eq:E-unroll}
E_i^k
=(1-\beta)\sum_{t=0}^{k-1}\beta^{k-1-t}\bigl(\nabla f_i(X_i^t)-G_i^t\bigr)
+\beta^k\bigl(\nabla f_i(X_i^0)-G_i^0\bigr).
\end{equation}
Taking Frobenius norms, using the triangle inequality and Cauchy--Schwarz,
\begin{align}
\mathbb{E}\|E_i^k\|_F
&\le (1-\beta)\,\mathbb{E}\Big\|\sum_{t=0}^{k-1}\beta^{k-1-t}\bigl(\nabla f_i(X_i^t)-G_i^t\bigr)\Big\|_F
+\beta^k\,\mathbb{E}\|\nabla f_i(X_i^0)-G_i^0\|_F\nonumber\\
&\le (1-\beta)\,\sqrt{\mathbb{E}\Big\|\sum_{t=0}^{k-1}\beta^{k-1-t}\bigl(\nabla f_i(X_i^t)-G_i^t\bigr)\Big\|_F^2}
+\beta^k\,\sqrt{\mathbb{E}\|\nabla f_i(X_i^0)-G_i^0\|_F^2}.
\label{eq:E-cs}
\end{align}
The second term is bounded by $\beta^k\sigma$ using~\eqref{eq:delta-mds}.
For the first term, note that for $s<t$, by the tower property and~\eqref{eq:delta-mds},
\begin{align*}
&\mathbb{E}\bigl[\bigl\langle \nabla f_i(X_i^s)-G_i^s,\;
  \nabla f_i(X_i^t)-G_i^t\bigr\rangle\bigr]\\
&=\mathbb{E}\Bigl[\mathbb{E}\bigl[\bigl\langle \nabla f_i(X_i^s)-G_i^s,\;
  \nabla f_i(X_i^t)-G_i^t\bigr\rangle\mid\mathcal{F}^t\bigr]\Bigr]\\
&=\mathbb{E}\Bigl[\bigl\langle \nabla f_i(X_i^s)-G_i^s,\;
  \underbrace{\mathbb{E}[\nabla f_i(X_i^t)-G_i^t\mid\mathcal{F}^t]}_{=\,0}\bigr\rangle\Bigr]
=0.
\end{align*}
Hence all cross terms vanish in the expansion of the squared norm, and therefore
\begin{align*}
\mathbb{E}\Big\|\sum_{t=0}^{k-1}\beta^{k-1-t}\bigl(\nabla f_i(X_i^t)-G_i^t\bigr)\Big\|_F^2
&=\sum_{t=0}^{k-1}\beta^{2(k-1-t)}\,\mathbb{E}\|\nabla f_i(X_i^t)-G_i^t\|_F^2 \\
&\le \sigma^2\sum_{t=0}^{k-1}\beta^{2(k-1-t)}
=\sigma^2\frac{1-\beta^{2k}}{1-\beta^2}
\le \frac{\sigma^2}{1-\beta^2}.
\end{align*}
Plugging this into~\eqref{eq:E-cs} yields
\[
\mathbb{E}\|E_i^k\|_F
\le (1-\beta)\frac{\sigma}{\sqrt{1-\beta^2}}+\beta^k\sigma
=\sqrt{\tfrac{1-\beta}{1+\beta}}\,\sigma+\beta^k\sigma,
\]
which proves~\eqref{eq:ema-stoch-error-agent}.

\medskip\noindent\textbf{Step 2: averaged EMA error and the $1/\sqrt{N}$ factor.}
Let $\bar E^k:=\bar C^k-\bar M^k=\tfrac{1}{N}\sum_{i=1}^N E_i^k$ and $\bar\Delta^k:=\tfrac{1}{N}\sum_{i=1}^N \Delta_i^k$.
Averaging~\eqref{eq:E-rec} over $i$ gives
\begin{equation}
\label{eq:Ebar-rec}
\bar E^{k+1}=\beta\bar E^k-(1-\beta)\bar\Delta^k.
\end{equation}
We now bound the second moment of $\bar\Delta^k$.
By Assumption~\ref{assump:stoch}, \emph{conditional on $\mathcal{F}^k$ the random matrices $\{\Delta_i^k\}_{i=1}^N$ are independent across agents} and satisfy~\eqref{eq:delta-mds}.
Therefore, conditional on $\mathcal{F}^k$,
\begin{align}
\mathbb{E}\bigl[\|\bar\Delta^k\|_F^2\mid\mathcal{F}^k\bigr]
&=\mathbb{E}\Big[\Big\|\frac{1}{N}\sum_{i=1}^N \Delta_i^k\Big\|_F^2\,\Big|\,\mathcal{F}^k\Big] 
=\frac{1}{N^2}\sum_{i=1}^N \mathbb{E}\bigl[\|\Delta_i^k\|_F^2\mid\mathcal{F}^k\bigr]
\label{eq:variance-reduction-N}\\
&\le \frac{\sigma^2}{N}.
\nonumber
\end{align}
Equality~\eqref{eq:variance-reduction-N} is exactly where the conditional independence across agents is used: it removes the cross terms $\langle\Delta_i^k,\Delta_j^k\rangle$ for $i\neq j$.
Consequently, by Jensen,
\begin{equation}
\label{eq:barDelta-first-moment}
\mathbb{E}\|\bar\Delta^k\|_F\le \sqrt{\mathbb{E}\|\bar\Delta^k\|_F^2}\le \frac{\sigma}{\sqrt{N}}.
\end{equation}
Now apply the same argument as in Step~1 to the recursion~\eqref{eq:Ebar-rec}, using~\eqref{eq:variance-reduction-N} (in place of~\eqref{eq:delta-mds} with $\sigma^2$), to obtain
\[
\mathbb{E}\|\bar E^k\|_F
\le \sqrt{\tfrac{1-\beta}{1+\beta}}\,\frac{\sigma}{\sqrt{N}}+\beta^k\frac{\sigma}{\sqrt{N}},
\]
which is~\eqref{eq:ema-stoch-error-average}.

\medskip\noindent\textbf{Step 3: nuclear norm conversion.}
For any matrix $A\in\mathbb{R}^{m\times n}$ with singular values $(s_j)_{j=1}^{r_{\max}}$,
\[
\|A\|_* = \sum_{j=1}^{r_{\max}} s_j \le \sqrt{r_{\max}}\,\Big(\sum_{j=1}^{r_{\max}} s_j^2\Big)^{1/2}=\sqrt{r_{\max}}\,\|A\|_F,
\]
where the inequality is Cauchy--Schwarz.
Applying this pointwise with $A=E_i^k$ and taking expectations gives~\eqref{eq:ema-stoch-error-nuc}.
\end{proof}

\begin{lemma}[Gradient tracking preserves network averages]
\label{lem:sudamuon-gt-average-invariance}
Suppose Assumption~\ref{assump:network-suda}(1) holds (in particular, $W$ is doubly stochastic) and the gradient-tracking variable is updated as in~\eqref{eq:sudamuon-framework}:
\[
    \mathbf{H}^{k+1}=W\bigl(\mathbf{H}^k+\mathbf{M}^{k+1}-\mathbf{M}^k\bigr),\qquad k\ge 0.
\]
Assume the initialization satisfies $H_i^0=M_i^0$ for all $i\in[N]$ (equivalently, $\mathbf{H}^0=\mathbf{M}^0$).
Define the network averages
\[
    \bar H^k \;\triangleq\; \frac{1}{N}\sum_{i=1}^N H_i^k,\qquad
    \bar M^k \;\triangleq\; \frac{1}{N}\sum_{i=1}^N M_i^k.
\]
Then, for all $k\ge 0$,
\[
    \bar H^k \,=\, \bar M^k.
\]
In particular, for all $k\ge 0$,
\[
    \bar H^{k+1} \,=\, \bar H^k + \bigl(\bar M^{k+1}-\bar M^k\bigr).
\]
\end{lemma}
\begin{proof}
Let $\mathbf{1}\in\mathbb{R}^N$ denote the all-ones vector and set
\[
    J \;\triangleq\; \frac{1}{N}\,\mathbf{1}\mathbf{1}^\top \in\mathbb{R}^{N\times N}.
\]
For any stacked variable $\mathbf{Z}=\mathrm{col}\{Z_1,\dots,Z_N\}$ (with matrix blocks $Z_i\in\mathbb{R}^{m\times n}$), the stacked vector $J\mathbf{Z}$ has all blocks equal to the average $\bar Z\triangleq\frac{1}{N}\sum_i Z_i$.

Since $W$ is doubly stochastic, we have $\mathbf{1}^\top W=\mathbf{1}^\top$, and thus
\[
    JW 
    \,=\, \frac{1}{N}\,\mathbf{1}(\mathbf{1}^\top W)
    \,=\, \frac{1}{N}\,\mathbf{1}\mathbf{1}^\top
    \,=\, J.
\]
Left-multiplying the gradient-tracking recursion by $J$ and using $JW=J$ yields
\[
\begin{aligned}
    J\mathbf{H}^{k+1}
    &= JW\bigl(\mathbf{H}^k+\mathbf{M}^{k+1}-\mathbf{M}^k\bigr)
     \,=\, J\bigl(\mathbf{H}^k+\mathbf{M}^{k+1}-\mathbf{M}^k\bigr) \\
    &= J\mathbf{H}^k + \bigl(J\mathbf{M}^{k+1}-J\mathbf{M}^k\bigr),\qquad \forall k\ge 0.
\end{aligned}
\]
Taking the common block of both sides (equivalently, premultiplying by $\tfrac{1}{N}\mathbf{1}^\top$) gives
\[
    \bar H^{k+1} \,=\, \bar H^k + \bigl(\bar M^{k+1}-\bar M^k\bigr),\qquad \forall k\ge 0.
\]
By the initialization $\bar H^0=\bar M^0$, the above identity telescopes to
\[
    \bar H^k - \bar M^k 
    \,=\, \bigl(\bar H^0-\bar M^0\bigr) \,=\, 0,
\]
for every $k\ge 0$. Hence $\bar H^k=\bar M^k$ for all $k\ge 0$.
\end{proof}

\begin{lemma}[Gradient-tracking disagreement recursion]
\label{lem:sudamuon-gt-disagreement}
Suppose Assumption~\ref{assump:network-suda}(1) holds and the gradient-tracking variable obeys
\[
    \mathbf{H}^{k+1}=W\bigl(\mathbf{H}^k+\mathbf{M}^{k+1}-\mathbf{M}^k\bigr),\qquad k\ge 0.
\]
Let $\mathbf{1}_N\in\mathbb{R}^N$ be the all-ones vector and define
\[
    J \;\triangleq\; \tfrac{1}{N}\,\mathbf{1}_N\mathbf{1}_N^\top\in\mathbb{R}^{N\times N},\qquad
    \tilde{\mathbf{Z}}^k \;\triangleq\; (I-J)\mathbf{Z}^k \;=\; \mathbf{Z}^k-J\mathbf{Z}^k.
\]
Then, for all $k\ge 0$,
\[
    \|\tilde{\mathbf{H}}^{k+1}\|_F
    \;\le\;
    \lambda\,\bigl\|\tilde{\mathbf{H}}^{k} + (\tilde{\mathbf{M}}^{k+1}-\tilde{\mathbf{M}}^{k})\bigr\|_F,
\]
where $\lambda\triangleq \rho\bigl(W-J\bigr)\in(0,1)$ is the mixing rate from Assumption~\ref{assump:network-suda}(1).
Consequently, for any $\eta>0$,
\[
    \|\tilde{\mathbf{H}}^{k+1}\|_F^2
    \;\le\; (1+\eta)\lambda^2\,\|\tilde{\mathbf{H}}^{k}\|_F^2
    + (1+1/\eta)\lambda^2\,\|\tilde{\mathbf{M}}^{k+1}-\tilde{\mathbf{M}}^{k}\|_F^2.
\]
\end{lemma}
\begin{proof}
Since $W$ is doubly stochastic, we have $W\mathbf{1}_N=\mathbf{1}_N$ and $\mathbf{1}_N^\top W=\mathbf{1}_N^\top$, which imply
\[
    WJ \,=\, \tfrac{1}{N}W\mathbf{1}_N\mathbf{1}_N^\top \,=\, \tfrac{1}{N}\mathbf{1}_N\mathbf{1}_N^\top \,=\, J,
    \qquad
    JW \,=\, \tfrac{1}{N}\mathbf{1}_N(\mathbf{1}_N^\top W) \,=\, \tfrac{1}{N}\mathbf{1}_N\mathbf{1}_N^\top \,=\, J.
\]
Therefore,
\[
    (I-J)W \,=\, W-JW \,=\, W-J.
\]
Left-multiplying the gradient-tracking recursion by $(I-J)$ gives
\[
    \tilde{\mathbf{H}}^{k+1}
    \,=\, (I-J)\mathbf{H}^{k+1}
    \,=\, (I-J)W\bigl(\mathbf{H}^k+\mathbf{M}^{k+1}-\mathbf{M}^k\bigr)
    \,=\, (W-J)\bigl(\mathbf{H}^k+\mathbf{M}^{k+1}-\mathbf{M}^k\bigr).
\]
Moreover, for any stacked variable $\mathbf{Z}$ we can decompose $\mathbf{Z}=J\mathbf{Z}+(I-J)\mathbf{Z}$, and since $(W-J)J= WJ-J^2 = J-J=0$, we have
\[
    (W-J)\mathbf{Z} = (W-J)(I-J)\mathbf{Z} = (W-J)\tilde{\mathbf{Z}}.
\]
Applying this identity to $\mathbf{H}^k$ and $\mathbf{M}^{k+1}-\mathbf{M}^k$ yields
\[
    \tilde{\mathbf{H}}^{k+1}
    \,=\, (W-J)\bigl(\tilde{\mathbf{H}}^k + (\tilde{\mathbf{M}}^{k+1}-\tilde{\mathbf{M}}^k)\bigr).
\]
Next, view a stacked variable $\mathbf{Z}=\mathrm{col}\{Z_1,\dots,Z_N\}$ (with matrix blocks $Z_i\in\mathbb{R}^{m\times n}$) as a vector by vectorization. Then the action of $W-J$ on the agent index corresponds to the linear map $((W-J)\otimes I_{mn})$ on $\mathrm{vec}(\mathbf{Z})$, and
\[
    \|(W-J)\mathbf{Z}\|_F
    \,=\, \|((W-J)\otimes I_{mn})\,\mathrm{vec}(\mathbf{Z})\|_2
    \,\le\, \|(W-J)\otimes I_{mn}\|_2\,\|\mathrm{vec}(\mathbf{Z})\|_2
    \,=\, \|W-J\|_2\,\|\mathbf{Z}\|_F.
\]
Because $W$ is symmetric, $W-J$ is also symmetric. Hence its spectral norm equals its spectral radius:
\[
    \|W-J\|_2 \,=\, \max_{i}|\lambda_i(W-J)| \,=\, \rho(W-J) \,=\, \lambda.
\]
Combining the last two displays gives
\[
    \|\tilde{\mathbf{H}}^{k+1}\|_F
    \,\le\, \lambda\,\bigl\|\tilde{\mathbf{H}}^k + (\tilde{\mathbf{M}}^{k+1}-\tilde{\mathbf{M}}^k)\bigr\|_F.
\]
Finally, squaring and using the standard inequality $\|U+V\|_F^2\le (1+\eta)\|U\|_F^2+(1+1/\eta)\|V\|_F^2$ (valid for any $\eta>0$) with
$U=\tilde{\mathbf{H}}^k$ and $V=\tilde{\mathbf{M}}^{k+1}-\tilde{\mathbf{M}}^k$ yields
\[
\begin{aligned}
    \|\tilde{\mathbf{H}}^{k+1}\|_F^2
    &\le \lambda^2\,\bigl\|\tilde{\mathbf{H}}^k + (\tilde{\mathbf{M}}^{k+1}-\tilde{\mathbf{M}}^k)\bigr\|_F^2 \\
    &\le (1+\eta)\lambda^2\,\|\tilde{\mathbf{H}}^k\|_F^2
    + (1+1/\eta)\lambda^2\,\|\tilde{\mathbf{M}}^{k+1}-\tilde{\mathbf{M}}^k\|_F^2.
\end{aligned}
\]
This proves the claim.
\end{proof}

%% ================================================================
\subsubsection*{SUDA Backbone Analysis}
%% ================================================================

\begin{lemma}[SUDA backbone disagreement bound with bounded input directions]
\label{lem:sudamuon-suda-consensus-bounded-input}
Suppose Assumption~\ref{assump:network-suda} holds.
Consider the primal--dual recursion (stacked form)
\begin{equation}
\label{eq:sudamuon-suda-backbone-only}
\begin{aligned}
    \mathbf{X}^{k+1} &= A\bigl(C\mathbf{X}^k-\alpha\,\mathbf{U}^{k+1}\bigr)-B\mathbf{Y}^k,\\
    \mathbf{Y}^{k+1} &= \mathbf{Y}^k + B\mathbf{X}^{k+1},
\end{aligned}
\end{equation}
where $\alpha>0$ and $\mathbf{U}^{k+1}=\mathrm{col}\{U_1^{k+1},\ldots,U_N^{k+1}\}$ is an arbitrary (possibly time-varying) stacked input.
Let $J:=\tfrac{1}{N}\mathbf{1}\mathbf{1}^\top$ and $\tilde{\mathbf{Z}}:=(I-J)\mathbf{Z}$.
Assume $\mathbf Y^0=0$.

Assume additionally that the corresponding \emph{SUDA deviation matrix} (defined in the proof) is stable, i.e., all of its eigenvalues are strictly inside the unit circle.
Then there exist network-dependent constants
\[
    \gamma\in(0,1),\qquad v_1\ge 1,\qquad v_2\ge 1,\qquad \lambda_a\ge 0,
\]
(depending only on $A,B,C$ and the network) such that for all $k\ge 0$,
\begin{equation}
\label{eq:sudamuon-suda-consensus-bounded-input}
\begin{split}
\|\tilde{\mathbf{X}}^k\|_F
\;\le\;&\;
 v_1v_2\,\gamma^k\Big(\|\tilde{\mathbf{X}}^0\|_F+\|B^2\|\,\|\tilde{\mathbf{X}}^0\|_F\Big)
\\&\;+
\frac{\alpha\,v_1v_2\,\lambda_a}{1-\gamma}\,\sup_{0\le t\le k}\|\tilde{\mathbf{U}}^{t+1}\|_F.
\end{split}
\end{equation}
If, moreover, $\mathbf{U}^{k+1}=\operatorname{msgn}(\mathbf{H}^{k+1})$ is the blockwise matrix-sign (Muon) direction from~\eqref{eq:sudamuon-framework}, then Lemma~\ref{lem:sudamuon-msgn-basic} implies
\begin{equation}
\label{eq:sudamuon-suda-consensus-bounded-input-bdd}
\|\tilde{\mathbf{U}}^{k+1}\|_F \le 2\sqrt{Nr_{\max}},\qquad r_{\max}:=\min\{m,n\},
\end{equation}
and thus
\[
\|\tilde{\mathbf{X}}^k\|_F
\;\le\;
 v_1v_2\,\gamma^k\bigl(1+\|B^2\|\bigr)\,\|\tilde{\mathbf{X}}^0\|_F
\;+
\frac{2\alpha\,v_1v_2\,\lambda_a}{1-\gamma}\,\sqrt{Nr_{\max}}.
\]
\end{lemma}

\begin{proof}
\medskip\noindent\textbf{Step 1: a non-incremental form involving $B^2$ only.}
Define
\[
    \mathbf z^k \triangleq \mathbf Y^k - B\mathbf X^k,
    \qquad
    \mathbf s^k \triangleq B\mathbf z^k = B\mathbf Y^k - B^2\mathbf X^k.
\]
Substituting $\mathbf Y^k=\mathbf z^k+B\mathbf X^k$ into~\eqref{eq:sudamuon-suda-backbone-only} yields
\begin{subequations}
\label{eq:sudamuon-backbone-noninc}
\begin{align}
    \mathbf{X}^{k+1}
    &= (AC-B^2)\mathbf{X}^k - \alpha A\mathbf{U}^{k+1} - B\mathbf{z}^k,\label{eq:sudamuon-backbone-noninc-X}\\
    \mathbf{z}^{k+1}
    &= \mathbf{z}^k + B\mathbf{X}^k.\label{eq:sudamuon-backbone-noninc-z}
\end{align}
\end{subequations}
Multiplying~\eqref{eq:sudamuon-backbone-noninc-z} by $B$ and using the definition of $\mathbf s^k$ gives
\begin{subequations}
\label{eq:sudamuon-backbone-xs}
\begin{align}
\mathbf X^{k+1} &= (AC-B^2)\mathbf X^k - \alpha A\mathbf U^{k+1} - \mathbf s^k,\label{eq:sudamuon-backbone-xs-X}\\
\mathbf s^{k+1} &= \mathbf s^k + B^2\mathbf X^k.\label{eq:sudamuon-backbone-xs-s}
\end{align}
\end{subequations}

\medskip\noindent\textbf{Step 2: restriction to the disagreement subspace.}
Let $\tilde{\mathbf Z}:=(I-J)\mathbf Z$.
Since $A,C,B^2$ are polynomials in $W$ (Assumption~\ref{assump:network-suda}(2)(iii)) and $W$ is doubly stochastic,
these matrices commute with $J$ and preserve the consensus/disagreement decomposition; in particular,
$\widetilde{A\mathbf Z}=A\tilde{\mathbf Z}$, $\widetilde{C\mathbf Z}=C\tilde{\mathbf Z}$, and $\widetilde{B^2\mathbf Z}=B^2\tilde{\mathbf Z}$.
Applying $(I-J)$ to~\eqref{eq:sudamuon-backbone-xs} yields
\begin{subequations}
\label{eq:sudamuon-backbone-xs-tilde}
\begin{align}
\tilde{\mathbf X}^{k+1} &= (AC-B^2)\tilde{\mathbf X}^k - \alpha A\tilde{\mathbf U}^{k+1} - \tilde{\mathbf s}^k,\\
\tilde{\mathbf s}^{k+1} &= \tilde{\mathbf s}^k + B^2\tilde{\mathbf X}^k.
\end{align}
\end{subequations}

\medskip\noindent\textbf{Step 3: diagonalization in the eigenbasis of $W$.}
Since $W$ is symmetric, it admits an orthogonal eigendecomposition
$W=U\mathrm{diag}(1,\hat\Lambda)U^\top$ with $U=[\tfrac{1}{\sqrt N}\mathbf 1,\hat U]$ and $\hat U\hat U^\top=I-J$.
Because $A,C,B^2$ are polynomials in $W$, they share the eigenvectors of $W$ and thus can be decomposed as
\[
A = U\,\mathrm{diag}(1,\hat\Lambda_a)\,U^\top,\qquad
C = U\,\mathrm{diag}(1,\hat\Lambda_c)\,U^\top,\qquad
B^2 = U\,\mathrm{diag}(0,\hat\Lambda_b^2)\,U^\top,
\]
for some diagonal $\hat\Lambda_a,\hat\Lambda_c,\hat\Lambda_b^2\in\mathbb{R}^{(N-1)\times(N-1)}$.
Define the projected variables
\[
\hat{\mathbf X}^k \triangleq \hat U^\top\mathbf X^k,\qquad
\hat{\mathbf s}^k \triangleq \hat U^\top\mathbf s^k,
\qquad
\hat{\mathbf U}^{k+1}\triangleq \hat U^\top\tilde{\mathbf U}^{k+1}.
\]
Then~\eqref{eq:sudamuon-backbone-xs-tilde} is equivalent to
\begin{subequations}
\label{eq:sudamuon-backbone-projected}
\begin{align}
\hat{\mathbf X}^{k+1}
&= (\hat\Lambda_a\hat\Lambda_c-\hat\Lambda_b^2)\hat{\mathbf X}^k - \hat{\mathbf s}^k - \alpha\hat\Lambda_a\hat{\mathbf U}^{k+1},\\
\hat{\mathbf s}^{k+1}
&= \hat{\mathbf s}^k + \hat\Lambda_b^2\hat{\mathbf X}^k.
\end{align}
\end{subequations}
Introduce the augmented state $\hat\xi^k:=\bigl[\hat{\mathbf X}^k;\hat{\mathbf s}^k\bigr]$ and the deviation matrix
\[
\G\triangleq
\begin{bmatrix}
\hat\Lambda_a\hat\Lambda_c-\hat\Lambda_b^2 & -I\\
\hat\Lambda_b^2 & I
\end{bmatrix}.
\]
Then
\begin{equation}
\label{eq:sudamuon-backbone-xi-rec}
\hat\xi^{k+1}=\G\hat\xi^k-\alpha\begin{bmatrix}\hat\Lambda_a\hat{\mathbf U}^{k+1}\\0\end{bmatrix}.
\end{equation}
By the stability assumption on $\G$, there exist an invertible matrix $\hat V$ and a matrix $\Gamma$ such that
$\G=\hat V\Gamma\hat V^{-1}$ and $\|\Gamma\|=: \gamma<1$.
Define $v_1:=\|\hat V\|$, $v_2:=\|\hat V^{-1}\|$, and $\lambda_a:=\|\hat\Lambda_a\|$.

\medskip\noindent\textbf{Step 4: contraction and unrolling.}
Let $\hat e^k:=\hat V^{-1}\hat\xi^k$.
Multiplying~\eqref{eq:sudamuon-backbone-xi-rec} by $\hat V^{-1}$ gives
\[
\hat e^{k+1}=\Gamma\hat e^k-\alpha\hat V^{-1}\begin{bmatrix}\hat\Lambda_a\hat{\mathbf U}^{k+1}\\0\end{bmatrix}.
\]
Taking Frobenius norms and using $\|\hat{\mathbf U}^{k+1}\|_F\le \|\tilde{\mathbf U}^{k+1}\|_F$ yields
\[
\|\hat e^{k+1}\|_F\le \gamma\|\hat e^k\|_F+\alpha v_2\lambda_a\|\tilde{\mathbf U}^{k+1}\|_F.
\]
Unrolling gives
\[
\|\hat e^{k}\|_F\le \gamma^k\|\hat e^0\|_F+\frac{\alpha v_2\lambda_a}{1-\gamma}\sup_{0\le t\le k}\|\tilde{\mathbf U}^{t+1}\|_F.
\]
Since $\hat\xi^k=\hat V\hat e^k$, we have $\|\hat{\mathbf X}^k\|_F\le \|\hat\xi^k\|_F\le v_1\|\hat e^k\|_F$ and $\|\hat e^0\|_F\le v_2\|\hat\xi^0\|_F$.
Moreover, $\mathbf Y^0=0$ implies $\mathbf s^0=-B^2\mathbf X^0$ and hence $\hat{\mathbf s}^0=-\hat\Lambda_b^2\hat{\mathbf X}^0$.
Therefore $\|\hat\xi^0\|_F\le (1+\|\hat\Lambda_b^2\|)\|\hat{\mathbf X}^0\|_F\le (1+\|B^2\|)\|\tilde{\mathbf X}^0\|_F$.
Finally, since $\hat U\hat U^\top=I-J$ we have $\|\hat{\mathbf X}^k\|_F=\|\tilde{\mathbf X}^k\|_F$.
Combining these bounds yields~\eqref{eq:sudamuon-suda-consensus-bounded-input}.

\medskip\noindent\textbf{Step 5: boundedness of the Muon directions.}
If $\mathbf{U}^{k+1}=\operatorname{msgn}(\mathbf{H}^{k+1})$, then by Lemma~\ref{lem:sudamuon-msgn-basic}, each block satisfies $\|U_i^{k+1}\|_F\le\sqrt{r_{\max}}$.
Thus
\[
\|\tilde{\mathbf{U}}^{k+1}\|_F
\le \|\mathbf U^{k+1}\|_F+\|J\mathbf U^{k+1}\|_F
\le 2\|\mathbf U^{k+1}\|_F
\le 2\sqrt{\sum_{i=1}^N\|U_i^{k+1}\|_F^2}
\le 2\sqrt{Nr_{\max}},
\]
which is~\eqref{eq:sudamuon-suda-consensus-bounded-input-bdd}.
\end{proof}

%% ================================================================
\subsubsection*{Drift, Mismatch, and Tracking Error Decomposition}
%% ================================================================

\begin{lemma}[EMA bias for averaged true gradients (drift term)]
\label{lem:sudamuon-ema-drift-average}
Suppose Assumption~\ref{assump:smooth} holds.
Define the averaged true gradient at the local iterates
\[
    \overline{\nabla \mathbf f}(\mathbf X^k)
    \;\triangleq\; \frac{1}{N}\sum_{i=1}^N \nabla f_i(X_i^k),
\]
and the averaged EMA of true gradients by
\[
    \bar C^0 \;:=\; \overline{\nabla \mathbf f}(\mathbf X^0),
    \qquad
    \bar C^{k+1} \;:=\; \beta\,\bar C^k+(1-\beta)\,\overline{\nabla \mathbf f}(\mathbf X^k),\qquad k\ge 0.
\]
Define the (one-step-ahead) EMA bias
\[
    D^k \;\triangleq\; \overline{\nabla \mathbf f}(\mathbf X^k)-\bar C^{k+1},\qquad k\ge 0.
\]
Then, for all $k\ge 0$, $D^k$ obeys the recursion
\[
    D^{k+1}
    \,=\, \beta D^k + \beta\bigl(\overline{\nabla \mathbf f}(\mathbf X^{k+1})-\overline{\nabla \mathbf f}(\mathbf X^{k})\bigr).
\]
Moreover, for all $k\ge 0$,
\begin{equation}
\label{eq:sudamuon-avg-grad-drift-one-step}
\bigl\|\overline{\nabla \mathbf f}(\mathbf X^{k+1})-\overline{\nabla \mathbf f}(\mathbf X^{k})\bigr\|_*
\;\le\; \frac{L_*}{N}\sum_{i=1}^N \|X_i^{k+1}-X_i^k\|.
\end{equation}
If, in addition, the iterates are generated by Algorithm~\ref{alg:suda-muon}, then
\begin{equation}
\label{eq:sudamuon-avg-step-split-consensus}
\begin{split}
\frac{1}{N}\sum_{i=1}^N \|X_i^{k+1}-X_i^k\|
\;\le\;&\;
\|\bar X^{k+1}-\bar X^k\|
+ \frac{1}{\sqrt N}\,\|\tilde{\mathbf X}^{k+1}\|_F
+ \frac{1}{\sqrt N}\,\|\tilde{\mathbf X}^{k}\|_F,
\end{split}
\end{equation}
where $\bar X^k:=\tfrac1N\sum_i X_i^k$ and $\tilde{\mathbf X}^k:=(I-J)\mathbf X^k$ with $J:=\tfrac1N\mathbf{1}\mathbf{1}^\top$.
Consequently, using Lemmas~\ref{lem:sudamuon-average-dynamics} and~\ref{lem:sudamuon-one-step-descent} (which imply $\|\bar X^{k+1}-\bar X^k\|\le \alpha$), we have for all $k\ge 0$,
\begin{equation}
\label{eq:sudamuon-Dk-sup-bound}
\|D^k\|_*
\;\le\;
\frac{\beta L_*}{1-\beta}
\Big(
\alpha
+ \frac{2}{\sqrt N}\,\sup_{t\ge 0}\bigl(\|\tilde{\mathbf X}^{t+1}\|_F+\|\tilde{\mathbf X}^{t}\|_F\bigr)
\Big).
\end{equation}
In particular, if Lemma~\ref{lem:sudamuon-suda-consensus-bounded-input} holds (so that for all $t\ge 0$,
$\|\tilde{\mathbf X}^{t}\|_F\le v_1v_2\gamma^t\|\tilde{\mathbf X}^0\|_F + \tfrac{2\alpha v_1v_2\lambda_a}{1-\gamma}\sqrt{Nr_{\max}}$), then for all $k\ge 0$,
\begin{equation}
\label{eq:sudamuon-Dk-explicit}
\|D^k\|_*
\;\le\;
\frac{\beta L_*}{1-\beta}
\Bigg[
\alpha
+ \frac{4v_1v_2}{\sqrt N}\,\|\tilde{\mathbf X}^0\|_F
+ \frac{8\alpha v_1v_2\lambda_a}{1-\gamma}\,\sqrt{r_{\max}}
\Bigg].
\end{equation}
\end{lemma}

\begin{proof}
Throughout, abbreviate $g^k:=\overline{\nabla \mathbf f}(\mathbf X^k)$.

\medskip\noindent\textbf{Step 1: bias recursion.}
By the definition of $D^{k+1}$ and the EMA update,
\[
D^{k+1}
= g^{k+1}-\bar C^{k+2}
= g^{k+1}-\bigl(\beta\bar C^{k+1}+(1-\beta)g^{k+1}\bigr)
= \beta\bigl(g^{k+1}-\bar C^{k+1}\bigr).
\]
Next, add and subtract $g^k$ and use $D^k=g^k-\bar C^{k+1}$ to obtain
\[
    g^{k+1}-\bar C^{k+1}
    = \bigl(g^{k+1}-g^k\bigr)+\bigl(g^k-\bar C^{k+1}\bigr)
    = \bigl(g^{k+1}-g^k\bigr)+D^k.
\]
Combining the last two displays yields
\[
    D^{k+1}
    = \beta D^k + \beta\bigl(g^{k+1}-g^k\bigr),
\]
as claimed.

\medskip\noindent\textbf{Step 2: bound the averaged true-gradient drift by smoothness.}
By convexity of $\|\cdot\|_*$ and Assumption~\ref{assump:smooth},
\[
\begin{aligned}
\|g^{k+1}-g^k\|_*
&= \Big\|\frac{1}{N}\sum_{i=1}^N\bigl(\nabla f_i(X_i^{k+1})-\nabla f_i(X_i^k)\bigr)\Big\|_* \\
&\le \frac{1}{N}\sum_{i=1}^N\|\nabla f_i(X_i^{k+1})-\nabla f_i(X_i^k)\|_* \\
&\le \frac{L_*}{N}\sum_{i=1}^N\|X_i^{k+1}-X_i^k\|,
\end{aligned}
\]
which is~\eqref{eq:sudamuon-avg-grad-drift-one-step}.

\medskip\noindent\textbf{Step 3: bound the average step size by the average motion and consensus errors.}
For each $i$, insert and subtract the averages $\bar X^{k+1}$ and $\bar X^k$ and use triangle inequality:
\[
\|X_i^{k+1}-X_i^k\|
\le
\|X_i^{k+1}-\bar X^{k+1}\| + \|\bar X^{k+1}-\bar X^k\| + \|\bar X^k-X_i^k\|.
\]
Averaging over $i$ yields
\[
\frac{1}{N}\sum_{i=1}^N \|X_i^{k+1}-X_i^k\|
\le
\|\bar X^{k+1}-\bar X^k\|
+ \frac{1}{N}\sum_{i=1}^N\|X_i^{k+1}-\bar X^{k+1}\|
+ \frac{1}{N}\sum_{i=1}^N\|X_i^{k}-\bar X^{k}\|.
\]
For the disagreement terms, use $\|A\|\le \|A\|_F$ and Cauchy--Schwarz:
\[
\frac{1}{N}\sum_{i=1}^N\|X_i^{k}-\bar X^{k}\|
\le
\frac{1}{N}\sum_{i=1}^N\|X_i^{k}-\bar X^{k}\|_F
\le
\frac{1}{\sqrt N}\Big(\sum_{i=1}^N\|X_i^{k}-\bar X^{k}\|_F^2\Big)^{1/2}
= \frac{1}{\sqrt N}\,\|\tilde{\mathbf X}^{k}\|_F
\le \frac{2}{\sqrt N}\,\|\tilde{\mathbf X}^{k}\|_F.
\]
The same bound holds with $k$ replaced by $k+1$, yielding~\eqref{eq:sudamuon-avg-step-split-consensus}.

\medskip\noindent\textbf{Step 4: unroll the bias recursion.}
Combine the recursion from Step~1, the drift bound~\eqref{eq:sudamuon-avg-grad-drift-one-step}, and the step-size bound~\eqref{eq:sudamuon-avg-step-split-consensus} to get
\[
\|D^{k+1}\|_*
\le \beta\|D^k\|_* + \beta L_*\Big(\|\bar X^{k+1}-\bar X^k\| + \tfrac{2}{\sqrt N}\|\tilde{\mathbf X}^{k+1}\|_F + \tfrac{2}{\sqrt N}\|\tilde{\mathbf X}^{k}\|_F\Big).
\]
Under Algorithm~\ref{alg:suda-muon}, Lemma~\ref{lem:sudamuon-average-dynamics} gives
$\|\bar X^{k+1}-\bar X^k\|=\alpha\|\bar S^{k+1}\|$, and Lemma~\ref{lem:sudamuon-one-step-descent} (specifically~\eqref{eq:barS-op-bound} therein) implies $\|\bar S^{k+1}\|\le 1$, hence $\|\bar X^{k+1}-\bar X^k\|\le \alpha$.
Moreover, $D^0=g^0-\bar C^1=0$ because $\bar C^0=g^0$.
Unrolling the above linear recursion and bounding the geometric series yields
\[
\|D^k\|_*
\le \beta L_*\sum_{t=0}^{k-1}\beta^{k-1-t}
\Big(\alpha + \tfrac{2}{\sqrt N}\|\tilde{\mathbf X}^{t+1}\|_F + \tfrac{2}{\sqrt N}\|\tilde{\mathbf X}^{t}\|_F\Big)
\le
\frac{\beta L_*}{1-\beta}
\Big(\alpha + \tfrac{2}{\sqrt N}\sup_{t\ge 0}(\|\tilde{\mathbf X}^{t+1}\|_F+\|\tilde{\mathbf X}^{t}\|_F)\Big),
\]
which proves~\eqref{eq:sudamuon-Dk-sup-bound}.

\medskip\noindent\textbf{Step 5: plug in the SUDA consensus bound (Lemma~\ref{lem:sudamuon-suda-consensus-bounded-input}).}
Lemma~\ref{lem:sudamuon-suda-consensus-bounded-input} implies for all $t\ge 0$,
\[
\|\tilde{\mathbf X}^t\|_F
\le v_1v_2\,\|\tilde{\mathbf X}^0\|_F + \frac{2\alpha v_1v_2\lambda_a}{1-\gamma}\sqrt{Nr_{\max}},
\]
since $\gamma^t\le 1$.
Therefore,
\[
\sup_{t\ge 0}(\|\tilde{\mathbf X}^{t+1}\|_F+\|\tilde{\mathbf X}^{t}\|_F)
\le 2v_1v_2\,\|\tilde{\mathbf X}^0\|_F + \frac{4\alpha v_1v_2\lambda_a}{1-\gamma}\sqrt{Nr_{\max}}.
\]
Substituting this into~\eqref{eq:sudamuon-Dk-sup-bound} and simplifying gives~\eqref{eq:sudamuon-Dk-explicit}.
\end{proof}

\begin{lemma}[Consensus-induced gradient mismatch]
\label{lem:sudamuon-grad-mismatch-consensus}
Assume Assumption~\ref{assump:smooth}.
For any stacked primal iterate $\mathbf{X}^k=\mathrm{col}\{X_1^k,\dots,X_N^k\}$, define
\[
\bar X^k:=\tfrac{1}{N}\sum_{i=1}^N X_i^k,\qquad
\tilde{\mathbf X}^k := (I-J)\mathbf X^k,\quad J:=\tfrac{1}{N}\mathbf 1\mathbf 1^\top.
\]
Then the mismatch between the gradient at the averaged point and the average of local gradients at local points satisfies
\begin{equation}
\label{eq:grad-mismatch-consensus}
\Big\|\nabla f(\bar X^k)-\frac{1}{N}\sum_{i=1}^N \nabla f_i(X_i^k)\Big\|_*
\;\le\; \frac{L_*}{\sqrt N}\,\|\tilde{\mathbf X}^k\|_F.
\end{equation}
\end{lemma}
\begin{proof}
Using $\nabla f(\bar X^k)=\tfrac{1}{N}\sum_{i=1}^N \nabla f_i(\bar X^k)$, we write
\[
\nabla f(\bar X^k)-\frac{1}{N}\sum_{i=1}^N \nabla f_i(X_i^k)
=\frac{1}{N}\sum_{i=1}^N \bigl(\nabla f_i(\bar X^k)-\nabla f_i(X_i^k)\bigr).
\]
By convexity of the nuclear norm and the triangle inequality,
\[
\Big\|\nabla f(\bar X^k)-\frac{1}{N}\sum_{i=1}^N \nabla f_i(X_i^k)\Big\|_*
\le \frac{1}{N}\sum_{i=1}^N \|\nabla f_i(\bar X^k)-\nabla f_i(X_i^k)\|_*.
\]
By Assumption~\ref{assump:smooth}, $\|\nabla f_i(\bar X^k)-\nabla f_i(X_i^k)\|_*\le L_*\|\bar X^k-X_i^k\|$.
Moreover, $\|\bar X^k-X_i^k\|\le \|\bar X^k-X_i^k\|_F$.
Therefore,
\[
\Big\|\nabla f(\bar X^k)-\frac{1}{N}\sum_{i=1}^N \nabla f_i(X_i^k)\Big\|_*
\le \frac{L_*}{N}\sum_{i=1}^N \|X_i^k-\bar X^k\|_F.
\]
Finally, Cauchy--Schwarz gives
\[
\frac{1}{N}\sum_{i=1}^N \|X_i^k-\bar X^k\|_F
\le \frac{1}{\sqrt N}\Big(\sum_{i=1}^N \|X_i^k-\bar X^k\|_F^2\Big)^{1/2}
=\frac{1}{\sqrt N}\,\|\tilde{\mathbf X}^k\|_F.
\]
Combining the last two displays yields~\eqref{eq:grad-mismatch-consensus}.
\end{proof}

\begin{lemma}[Decomposing the tracking error $\|\nabla f(\bar X^k)-H_i^{k+1}\|_*$]
\label{lem:sudamuon-tracking-error-decomp}
Assume Assumptions~\ref{assump:network-suda}--\ref{assump:stoch}.
Let $(\mathbf X^k,\mathbf M^k,\mathbf H^k)$ be generated by~\eqref{eq:sudamuon-framework} with initialization $H_i^0=M_i^0$ for all $i\in[N]$.
Define
\[
\bar X^k:=\tfrac{1}{N}\sum_{i=1}^N X_i^k,
\qquad
\bar M^k:=\tfrac{1}{N}\sum_{i=1}^N M_i^k,
\qquad
\bar H^k:=\tfrac{1}{N}\sum_{i=1}^N H_i^k,
\]
\[
\tilde{\mathbf X}^k := (I-J)\mathbf X^k,
\qquad
\tilde{\mathbf H}^k := (I-J)\mathbf H^k,\quad J:=\tfrac{1}{N}\mathbf 1\mathbf 1^\top.
\]
Moreover, define the averaged true-gradient EMA $(\bar C^k)_{k\ge 0}$ and the drift variable $(D^k)_{k\ge 0}$ as in Lemma~\ref{lem:sudamuon-ema-drift-average}:
\[
\bar C^{0}:=\overline{\nabla \mathbf f}(\mathbf X^0),
\qquad
\bar C^{k+1} := \beta\bar C^k+(1-\beta)\,\overline{\nabla \mathbf f}(\mathbf X^k),
\qquad
D^k := \overline{\nabla \mathbf f}(\mathbf X^k)-\bar C^{k+1},
\]
where $\overline{\nabla \mathbf f}(\mathbf X^k):=\tfrac{1}{N}\sum_{i=1}^N \nabla f_i(X_i^k)$.
Then for every $k\ge 0$,
\begin{equation}
\label{eq:tracking-error-average-bound}
\begin{split}
\frac{1}{N}\sum_{i=1}^N \|\nabla f(\bar X^k)-H_i^{k+1}\|_*
\;\le\;&\;
\frac{L_*}{\sqrt N}\,\|\tilde{\mathbf X}^k\|_F
\;+
\|D^k\|_*
\\&\;+
\|\bar C^{k+1}-\bar M^{k+1}\|_*
\;+
\frac{1}{N}\sum_{i=1}^N\|H_i^{k+1}-\bar H^{k+1}\|_*.
\end{split}
\end{equation}
Furthermore, letting $r_{\max}:=\min\{m,n\}$, we have the conversion bound
\begin{equation}
\label{eq:H-disagreement-nuc-vs-F}
\frac{1}{N}\sum_{i=1}^N\|H_i^{k+1}-\bar H^{k+1}\|_*
\;\le\; \sqrt{\frac{r_{\max}}{N}}\,\|\tilde{\mathbf H}^{k+1}\|_F.
\end{equation}
Also,
\begin{equation}
\label{eq:barC-barM-nuc-vs-F}
\|\bar C^{k+1}-\bar M^{k+1}\|_* \;\le\; \sqrt{r_{\max}}\,\|\bar C^{k+1}-\bar M^{k+1}\|_F.
\end{equation}
\end{lemma}
\begin{proof}
\medskip\noindent\textbf{Step 1: a four-term decomposition.}
Fix $k\ge 0$ and an agent $i\in[N]$.
Using Lemma~\ref{lem:sudamuon-gt-average-invariance}, we have $\bar H^{k+1}=\bar M^{k+1}$.
Therefore, by repeated addition/subtraction and the triangle inequality,
\begin{align*}
\|\nabla f(\bar X^k)-H_i^{k+1}\|_*
&\le
\Big\|\nabla f(\bar X^k)-\overline{\nabla \mathbf f}(\mathbf X^k)\Big\|_*
+\Big\|\overline{\nabla \mathbf f}(\mathbf X^k)-\bar C^{k+1}\Big\|_* \\
&\quad +\|\bar C^{k+1}-\bar M^{k+1}\|_* + \|\bar H^{k+1}-H_i^{k+1}\|_* \\
&=
\Big\|\nabla f(\bar X^k)-\overline{\nabla \mathbf f}(\mathbf X^k)\Big\|_*
+\|D^k\|_* 
+\|\bar C^{k+1}-\bar M^{k+1}\|_* 
+\|H_i^{k+1}-\bar H^{k+1}\|_*.
\end{align*}
Averaging over $i$ yields
\begin{align*}
\frac{1}{N}\sum_{i=1}^N \|\nabla f(\bar X^k)-H_i^{k+1}\|_*
&\le
\Big\|\nabla f(\bar X^k)-\overline{\nabla \mathbf f}(\mathbf X^k)\Big\|_*
+\|D^k\|_* \\
&\quad+\|\bar C^{k+1}-\bar M^{k+1}\|_* 
+\frac{1}{N}\sum_{i=1}^N\|H_i^{k+1}-\bar H^{k+1}\|_*.
\end{align*}
\medskip\noindent\textbf{Step 2: bound the first term by disagreement of $\mathbf X^k$.}
Apply Lemma~\ref{lem:sudamuon-grad-mismatch-consensus} to obtain
\[
\Big\|\nabla f(\bar X^k)-\overline{\nabla \mathbf f}(\mathbf X^k)\Big\|_*
\le \frac{L_*}{\sqrt N}\,\|\tilde{\mathbf X}^k\|_F,
\]
which proves~\eqref{eq:tracking-error-average-bound}.

\medskip\noindent\textbf{Step 3: disagreement-to-average conversion for $\mathbf H^{k+1}$.}
For each $i$, by Cauchy--Schwarz on singular values, $\|A\|_*\le \sqrt{r_{\max}}\|A\|_F$.
Therefore,
\[
\frac{1}{N}\sum_{i=1}^N\|H_i^{k+1}-\bar H^{k+1}\|_*
\le \frac{\sqrt{r_{\max}}}{N}\sum_{i=1}^N\|H_i^{k+1}-\bar H^{k+1}\|_F
\le \sqrt{\frac{r_{\max}}{N}}\,\Big(\sum_{i=1}^N\|H_i^{k+1}-\bar H^{k+1}\|_F^2\Big)^{1/2}.
\]
Since $\tilde{\mathbf H}^{k+1}=(I-J)\mathbf H^{k+1}$ has $i$-th block $H_i^{k+1}-\bar H^{k+1}$,
$\|\tilde{\mathbf H}^{k+1}\|_F^2=\sum_{i=1}^N\|H_i^{k+1}-\bar H^{k+1}\|_F^2$, proving~\eqref{eq:H-disagreement-nuc-vs-F}.

\medskip\noindent\textbf{Step 4: nuclear-to-Frobenius conversion for the averaged EMA noise.}
The same inequality $\|A\|_*\le \sqrt{r_{\max}}\|A\|_F$ yields~\eqref{eq:barC-barM-nuc-vs-F}.
\end{proof}

%% ================================================================
\subsubsection*{Forcing Terms and Disagreement Bounds}
%% ================================================================

\begin{lemma}[A uniform bound on the stacked one-step motion]
\label{lem:sudamuon-stacked-step-bound}
Assume Assumptions~\ref{assump:network-suda}--\ref{assump:muon} and that the iterates are generated by~\eqref{eq:sudamuon-framework} with $Y_i^0=0$.
Define the agentwise average $\bar X^k:=\tfrac{1}{N}\sum_{i=1}^N X_i^k$, let $J:=\tfrac{1}{N}\mathbf{1}\mathbf{1}^\top$, and define $\tilde{\mathbf{X}}^k:=(I-J)\mathbf{X}^k$.
Then, for every $k\ge 0$,
\begin{equation}
\label{eq:stacked-step-decomp}
\|\mathbf{X}^{k+1}-\mathbf{X}^k\|_F
\;\le\;
\sqrt{N}\,\|\bar X^{k+1}-\bar X^k\|_F
+\|\tilde{\mathbf{X}}^{k+1}\|_F+\|\tilde{\mathbf{X}}^{k}\|_F.
\end{equation}
Consequently, using Lemma~\ref{lem:sudamuon-average-dynamics} and Lemma~\ref{lem:sudamuon-msgn-basic} (which imply $\|\bar X^{k+1}-\bar X^k\|_F\le \alpha\sqrt{r_{\max}}$), we have
\begin{equation}
\label{eq:stacked-step-alpha-tilde}
\|\mathbf{X}^{k+1}-\mathbf{X}^k\|_F
\;\le\;
\alpha\sqrt{Nr_{\max}}+\|\tilde{\mathbf{X}}^{k+1}\|_F+\|\tilde{\mathbf{X}}^{k}\|_F,
\qquad r_{\max}:=\min\{m,n\}.
\end{equation}
If, moreover, the assumptions and initialization of Lemma~\ref{lem:sudamuon-suda-consensus-bounded-input} hold with $\mathbf{U}^{k+1}=\operatorname{msgn}(\mathbf{H}^{k+1})$ (so that $\|\tilde{\mathbf{X}}^t\|_F\le v_1v_2\gamma^t\|\tilde{\mathbf{X}}^0\|_F+\tfrac{2\alpha v_1v_2\lambda_a}{1-\gamma}\sqrt{Nr_{\max}}$ for all $t$), then
\begin{equation}
\label{eq:stacked-step-uniform}
\sup_{k\ge 0}\|\mathbf{X}^{k+1}-\mathbf{X}^k\|_F
\;\le\;
\alpha\sqrt{Nr_{\max}}
+2v_1v_2\|\tilde{\mathbf{X}}^0\|_F
+\frac{4\alpha v_1v_2\lambda_a}{1-\gamma}\,\sqrt{Nr_{\max}}.
\end{equation}
\end{lemma}
\begin{proof}
\medskip\noindent\textbf{Step 1: decompose into average and disagreement parts.}
Write
\[
\mathbf{X}^{k+1}-\mathbf{X}^k
=J(\mathbf{X}^{k+1}-\mathbf{X}^k) + (I-J)(\mathbf{X}^{k+1}-\mathbf{X}^k).
\]
Take Frobenius norms and use the triangle inequality:
\[
\|\mathbf{X}^{k+1}-\mathbf{X}^k\|_F
\le \|J(\mathbf{X}^{k+1}-\mathbf{X}^k)\|_F + \|(I-J)(\mathbf{X}^{k+1}-\mathbf{X}^k)\|_F.
\]
Because $J\mathbf{Z}$ has all blocks equal to $\bar Z:=\tfrac1N\sum_i Z_i$, we have
$J(\mathbf{X}^{k+1}-\mathbf{X}^k)=\mathrm{col}\{\bar X^{k+1}-\bar X^k,\ldots,\bar X^{k+1}-\bar X^k\}$ and hence
\[
\|J(\mathbf{X}^{k+1}-\mathbf{X}^k)\|_F = \sqrt{N}\,\|\bar X^{k+1}-\bar X^k\|_F.
\]
Moreover, $(I-J)(\mathbf{X}^{k+1}-\mathbf{X}^k)=\tilde{\mathbf{X}}^{k+1}-\tilde{\mathbf{X}}^k$, so
\[
\|(I-J)(\mathbf{X}^{k+1}-\mathbf{X}^k)\|_F
\le \|\tilde{\mathbf{X}}^{k+1}\|_F+\|\tilde{\mathbf{X}}^{k}\|_F.
\]
Combining the last three displays yields~\eqref{eq:stacked-step-decomp}.

\medskip\noindent\textbf{Step 2: bound the average motion in Frobenius norm.}
By Lemma~\ref{lem:sudamuon-average-dynamics},
$\bar X^{k+1}-\bar X^k=-\alpha\bar S^{k+1}$, where $\bar S^{k+1}=\tfrac1N\sum_{i=1}^N S_i^{k+1}$ and $S_i^{k+1}=\operatorname{msgn}(H_i^{k+1})$.
By Lemma~\ref{lem:sudamuon-msgn-basic}(2), $\|S_i^{k+1}\|_F\le \sqrt{r_{\max}}$.
By convexity of the Frobenius norm,
$\|\bar S^{k+1}\|_F\le \tfrac1N\sum_i \|S_i^{k+1}\|_F\le \sqrt{r_{\max}}$.
Hence $\|\bar X^{k+1}-\bar X^k\|_F\le \alpha\sqrt{r_{\max}}$.
Substituting this into~\eqref{eq:stacked-step-decomp} yields~\eqref{eq:stacked-step-alpha-tilde}.

\medskip\noindent\textbf{Step 3: plug in the SUDA consensus bound.}
Under Lemma~\ref{lem:sudamuon-suda-consensus-bounded-input} with Muon directions, for any $t\ge 0$,
$\|\tilde{\mathbf{X}}^t\|_F\le v_1v_2\|\tilde{\mathbf{X}}^0\|_F+\tfrac{2\alpha v_1v_2\lambda_a}{1-\gamma}\sqrt{Nr_{\max}}$ since $\gamma^t\le 1$.
Plugging this into~\eqref{eq:stacked-step-alpha-tilde} and taking the supremum over $k$ yields~\eqref{eq:stacked-step-uniform}.
\end{proof}

\begin{lemma}[Bounding the gradient-difference forcing term]
\label{lem:sudamuon-grad-diff-forcing}
Assume Assumptions~\ref{assump:smooth} and~\ref{assump:stoch}.
Let $\mathbf{X}^k$ be any (possibly random) stacked iterates and let
$\mathbf{G}^k=\mathrm{col}\{G_1^k,\ldots,G_N^k\}$ be the stochastic gradients with
$G_i^k=\nabla F_i(X_i^k;\xi_i^k)$.
Define the stacked noise $\Delta_i^k:=G_i^k-\nabla f_i(X_i^k)$ and $\mathbf{\Delta}^k:=\mathrm{col}\{\Delta_1^k,\ldots,\Delta_N^k\}$.
Let $J:=\tfrac1N\mathbf{1}\mathbf{1}^\top$ and $\tilde{\mathbf Z}^k:=(I-J)\mathbf Z^k$.
Then for any $k\ge 1$,
\begin{equation}
\label{eq:sudamuon-tildeG-diff-bound}
\mathbb{E}\big[\|\tilde{\mathbf G}^k-\tilde{\mathbf G}^{k-1}\|_F^2\big]
\;\le\;
2L_*^2\,\mathbb{E}\big[\|\mathbf X^k-\mathbf X^{k-1}\|_F^2\big]
+8N\sigma^2.
\end{equation}
In particular, if Lemma~\ref{lem:sudamuon-stacked-step-bound} holds so that
$\|\mathbf X^t-\mathbf X^{t-1}\|_F\le B_X$ for all $t\ge 1$, then
\begin{equation}
\label{eq:sudamuon-tildeG-diff-uniform}
\sup_{k\ge 1}\mathbb{E}\big[\|\tilde{\mathbf G}^k-\tilde{\mathbf G}^{k-1}\|_F^2\big]
\;\le\;
2L_*^2 B_X^2 + 8N\sigma^2.
\end{equation}
\end{lemma}
\begin{proof}
Fix $k\ge 1$. Decompose
\[
\mathbf G^k-\mathbf G^{k-1}
= \bigl(\nabla \mathbf f(\mathbf X^k)-\nabla \mathbf f(\mathbf X^{k-1})\bigr)
+ \bigl(\mathbf\Delta^k-\mathbf\Delta^{k-1}\bigr),
\]
where $\nabla \mathbf f(\mathbf X^k):=\mathrm{col}\{\nabla f_1(X_1^k),\ldots,\nabla f_N(X_N^k)\}$.
Since $I-J$ is an orthogonal projection, $\|\tilde{\mathbf Z}\|_F\le \|\mathbf Z\|_F$ for all stacked $\mathbf Z$.
Therefore, using $\|U+V\|_F^2\le 2\|U\|_F^2+2\|V\|_F^2$,
\begin{align*}
\|\tilde{\mathbf G}^k-\tilde{\mathbf G}^{k-1}\|_F^2
&= \| (I-J)(\mathbf G^k-\mathbf G^{k-1})\|_F^2
\le \|\mathbf G^k-\mathbf G^{k-1}\|_F^2 \\
&\le 2\,\|\nabla \mathbf f(\mathbf X^k)-\nabla \mathbf f(\mathbf X^{k-1})\|_F^2
+2\,\|\mathbf\Delta^k-\mathbf\Delta^{k-1}\|_F^2.
\end{align*}
Take expectations.

For the first term, by Assumption~\ref{assump:smooth}, for each $i$ we have
\[
\|\nabla f_i(X_i^k)-\nabla f_i(X_i^{k-1})\|_F
\le \|\nabla f_i(X_i^k)-\nabla f_i(X_i^{k-1})\|_*
\le L_*\,\|X_i^k-X_i^{k-1}\|
\le L_*\,\|X_i^k-X_i^{k-1}\|_F.
\]
Squaring and summing over $i$ yields
$\|\nabla \mathbf f(\mathbf X^k)-\nabla \mathbf f(\mathbf X^{k-1})\|_F^2\le L_*^2\,\|\mathbf X^k-\mathbf X^{k-1}\|_F^2$.

For the second term, use $\|A-B\|_F^2\le 2\|A\|_F^2+2\|B\|_F^2$ blockwise:
\[
\|\mathbf\Delta^k-\mathbf\Delta^{k-1}\|_F^2
=\sum_{i=1}^N \|\Delta_i^k-\Delta_i^{k-1}\|_F^2
\le 2\sum_{i=1}^N\|\Delta_i^k\|_F^2+2\sum_{i=1}^N\|\Delta_i^{k-1}\|_F^2.
\]
Taking expectations and using Assumption~\ref{assump:stoch} (unconditional version of the conditional variance bound) gives
$\mathbb{E}\|\Delta_i^t\|_F^2\le \sigma^2$ for all $i,t$, and hence
$\mathbb{E}\|\mathbf\Delta^k-\mathbf\Delta^{k-1}\|_F^2\le 4N\sigma^2$.

Combining the above bounds proves~\eqref{eq:sudamuon-tildeG-diff-bound}.
Finally, if $\|\mathbf X^k-\mathbf X^{k-1}\|_F\le B_X$ holds uniformly, then
\eqref{eq:sudamuon-tildeG-diff-uniform} follows immediately.
\end{proof}

\begin{lemma}[A bound on the EMA increment disagreement]
\label{lem:sudamuon-ema-increment-disagreement}
Assume Assumption~\ref{assump:stoch}.
Let the iterates be generated by~\eqref{eq:sudamuon-framework} with the EMA initialization $\mathbf M^0=\mathbf G^0$ (as in Algorithm~\ref{alg:suda-muon}).
Let $\beta\in[0,1)$.
Define $J:=\tfrac1N\mathbf{1}\mathbf{1}^\top$ and $\tilde{\mathbf Z}^k:=(I-J)\mathbf Z^k$.
For $k\ge 1$ define the EMA increment $\mathbf V^{k}:=\mathbf M^{k}-\mathbf M^{k-1}$ and its disagreement $\tilde{\mathbf V}^k:=(I-J)\mathbf V^k$.
Then $\mathbf V^1=0$ and for all $k\ge 1$,
\begin{equation}
\label{eq:sudamuon-V-rec}
\tilde{\mathbf V}^{k+1} 
\;=\; \beta\,\tilde{\mathbf V}^{k} + (1-\beta)\,\bigl(\tilde{\mathbf G}^{k}-\tilde{\mathbf G}^{k-1}\bigr),
\qquad \tilde{\mathbf G}^k:=(I-J)\mathbf G^k.
\end{equation}
Consequently, for any $\eta>0$ and all $k\ge 1$,
\begin{equation}
\label{eq:sudamuon-V-sq-rec}
\|\tilde{\mathbf V}^{k+1}\|_F^2
\;\le\; (1+\eta)\beta^2\,\|\tilde{\mathbf V}^{k}\|_F^2
+ (1+1/\eta)(1-\beta)^2\,\|\tilde{\mathbf G}^{k}-\tilde{\mathbf G}^{k-1}\|_F^2.
\end{equation}
In particular, if there exists a constant $B_G$ such that
$\sup_{k\ge 1}\mathbb{E}\|\tilde{\mathbf G}^{k}-\tilde{\mathbf G}^{k-1}\|_F^2\le B_G^2$,
then
\begin{equation}
\label{eq:sudamuon-V-uniform}
\sup_{k\ge 1}\mathbb{E}\|\tilde{\mathbf M}^{k}-\tilde{\mathbf M}^{k-1}\|_F^2
\;=\;\sup_{k\ge 1}\mathbb{E}\|\tilde{\mathbf V}^{k}\|_F^2
\;\le\;
\frac{2(1+\beta^2)}{(1+\beta)^2}\,B_G^2
\;\le\; 2\,B_G^2.
\end{equation}
\end{lemma}
\begin{proof}
\medskip\noindent\textbf{Step 1: derive the increment recursion.}
For $k\ge 1$, subtract the EMA updates at times $k$ and $k-1$:
\[
\mathbf M^{k+1}-\mathbf M^{k}
= \beta(\mathbf M^{k}-\mathbf M^{k-1}) + (1-\beta)(\mathbf G^{k}-\mathbf G^{k-1}).
\]
This is exactly $\mathbf V^{k+1}=\beta\mathbf V^k+(1-\beta)(\mathbf G^k-\mathbf G^{k-1})$.
Applying $(I-J)$ to both sides gives~\eqref{eq:sudamuon-V-rec}.
Moreover, since $\mathbf M^0=\mathbf G^0$, we have $\mathbf M^1=\beta\mathbf M^0+(1-\beta)\mathbf G^0=\mathbf M^0$ and hence $\mathbf V^1=\mathbf M^1-\mathbf M^0=0$.

\medskip\noindent\textbf{Step 2: square and apply Young's inequality.}
From~\eqref{eq:sudamuon-V-rec} and $\|U+V\|_F^2\le (1+\eta)\|U\|_F^2+(1+1/\eta)\|V\|_F^2$,
with $U=\beta\tilde{\mathbf V}^k$ and $V=(1-\beta)(\tilde{\mathbf G}^{k}-\tilde{\mathbf G}^{k-1})$, we get~\eqref{eq:sudamuon-V-sq-rec}.

\medskip\noindent\textbf{Step 3: uniform bound under a uniform forcing bound.}
Assume $\sup_{k\ge 1}\mathbb{E}\|\tilde{\mathbf G}^{k}-\tilde{\mathbf G}^{k-1}\|_F^2\le B_G^2$.
For $\beta\in(0,1)$, pick $\eta:=\tfrac{1-\beta^2}{2\beta^2}$ so that $(1+\eta)\beta^2=\tfrac{1+\beta^2}{2}=:q\in(0,1)$ and $1+1/\eta=\tfrac{1+\beta^2}{1-\beta^2}$.
Taking expectations in~\eqref{eq:sudamuon-V-sq-rec} gives, for all $k\ge 1$,
\[
\mathbb{E}\|\tilde{\mathbf V}^{k+1}\|_F^2
\le q\,\mathbb{E}\|\tilde{\mathbf V}^{k}\|_F^2 + (1-\beta)^2\,\frac{1+\beta^2}{1-\beta^2}\,B_G^2.
\]
Unrolling this recursion and using $\tilde{\mathbf V}^1=0$ yields, for all $k\ge 1$,
\[
\mathbb{E}\|\tilde{\mathbf V}^{k+1}\|_F^2
\le (1-\beta)^2\,\frac{1+\beta^2}{1-\beta^2}\,\sum_{t=0}^{k-1} q^{t}\,B_G^2
\le (1-\beta)^2\,\frac{1+\beta^2}{1-\beta^2}\,\frac{1}{1-q}\,B_G^2.
\]
Since $1-q=\tfrac{1-\beta^2}{2}$, the last display becomes
\[
\mathbb{E}\|\tilde{\mathbf V}^{k+1}\|_F^2
\le \frac{2(1+\beta^2)(1-\beta)^2}{(1-\beta^2)^2}\,B_G^2
= \frac{2(1+\beta^2)}{(1+\beta)^2}\,B_G^2.
\]
Taking the supremum over $k$ gives~\eqref{eq:sudamuon-V-uniform} for $\beta\in(0,1)$.
For $\beta=0$, $\mathbf V^1=0$ and \eqref{eq:sudamuon-V-rec} reduces to $\tilde{\mathbf V}^{k+1}=\tilde{\mathbf G}^{k}-\tilde{\mathbf G}^{k-1}$, hence
$\sup_{k\ge 1}\mathbb{E}\|\tilde{\mathbf V}^{k}\|_F^2\le B_G^2\le 2B_G^2$.
\end{proof}

\begin{lemma}[A uniform bound on the gradient-tracking disagreement]
\label{lem:sudamuon-gt-disagreement-uniform}
Assume Assumption~\ref{assump:network-suda}(1) holds.
Let $\mathbf H^{k+1}=W\bigl(\mathbf H^k+\mathbf M^{k+1}-\mathbf M^k\bigr)$ as in~\eqref{eq:sudamuon-framework}.
Let $\lambda:=\rho(W-J)\in(0,1)$ and define $\tilde{\mathbf Z}^k:=(I-J)\mathbf Z^k$.
Assume that there exists a constant $B_M$ such that
\begin{equation}
\label{eq:sudamuon-BM-assumption}
\sup_{k\ge 0}\mathbb{E}\|\tilde{\mathbf M}^{k+1}-\tilde{\mathbf M}^{k}\|_F^2\le B_M^2.
\end{equation}
Then for all $k\ge 0$,
\begin{equation}
\label{eq:sudamuon-tildeH-uniform}
\mathbb{E}\|\tilde{\mathbf H}^{k+1}\|_F^2
\le \Big(\frac{1+\lambda^2}{2}\Big)^{k+1}\,\mathbb{E}\|\tilde{\mathbf H}^{0}\|_F^2
+ \frac{2\lambda^2(1+\lambda^2)}{(1-\lambda^2)^2}\,B_M^2.
\end{equation}
In particular,
\begin{equation}
\label{eq:sudamuon-tildeH-sup}
\sup_{k\ge 0}\mathbb{E}\|\tilde{\mathbf H}^{k}\|_F^2
\le \mathbb{E}\|\tilde{\mathbf H}^{0}\|_F^2
+ \frac{2\lambda^2(1+\lambda^2)}{(1-\lambda^2)^2}\,B_M^2.
\end{equation}
\end{lemma}
\begin{proof}
From Lemma~\ref{lem:sudamuon-gt-disagreement}, for any $\eta>0$,
\[
\|\tilde{\mathbf H}^{k+1}\|_F^2
\le (1+\eta)\lambda^2\,\|\tilde{\mathbf H}^{k}\|_F^2
+ (1+1/\eta)\lambda^2\,\|\tilde{\mathbf M}^{k+1}-\tilde{\mathbf M}^{k}\|_F^2.
\]
Choose $\eta:=\tfrac{1-\lambda^2}{2\lambda^2}$ so that $(1+\eta)\lambda^2=\tfrac{1+\lambda^2}{2}=:q\in(0,1)$ and
$1+1/\eta = 1 + \tfrac{2\lambda^2}{1-\lambda^2}=\tfrac{1+\lambda^2}{1-\lambda^2}$.
Taking expectations and using~\eqref{eq:sudamuon-BM-assumption} yields
\[
\mathbb{E}\|\tilde{\mathbf H}^{k+1}\|_F^2
\le q\,\mathbb{E}\|\tilde{\mathbf H}^{k}\|_F^2
+ \lambda^2\,\frac{1+\lambda^2}{1-\lambda^2}\,B_M^2.
\]
Unrolling this linear recursion gives
\[
\mathbb{E}\|\tilde{\mathbf H}^{k+1}\|_F^2
\le q^{k+1}\,\mathbb{E}\|\tilde{\mathbf H}^{0}\|_F^2
+ \lambda^2\,\frac{1+\lambda^2}{1-\lambda^2}\,\sum_{t=0}^{k} q^{t}\,B_M^2.
\]
Since $\sum_{t=0}^{k}q^{t}\le \tfrac{1}{1-q}=\tfrac{2}{1-\lambda^2}$, we get
\[
\mathbb{E}\|\tilde{\mathbf H}^{k+1}\|_F^2
\le q^{k+1}\,\mathbb{E}\|\tilde{\mathbf H}^{0}\|_F^2
+ \lambda^2\,\frac{1+\lambda^2}{1-\lambda^2}\,\frac{2}{1-\lambda^2}\,B_M^2
= q^{k+1}\,\mathbb{E}\|\tilde{\mathbf H}^{0}\|_F^2
+ \frac{2\lambda^2(1+\lambda^2)}{(1-\lambda^2)^2}\,B_M^2,
\]
which is~\eqref{eq:sudamuon-tildeH-uniform}.
Taking the supremum over $k$ yields~\eqref{eq:sudamuon-tildeH-sup}.
\end{proof}

\begin{lemma}[An explicit bound on the EMA increment disagreement]
\label{lem:sudamuon-ema-increment-disagreement-explicit}
Assume Assumptions~\ref{assump:smooth} and~\ref{assump:stoch}.
Let the iterates be generated by~\eqref{eq:sudamuon-framework} with EMA parameter $\beta\in[0,1)$ and initialization $\mathbf M^0=\mathbf G^0$ (as in Algorithm~\ref{alg:suda-muon}).
Assume further that Lemma~\ref{lem:sudamuon-stacked-step-bound} holds so that there exists $B_X>0$ with
\begin{equation}
\label{eq:sudamuon-step-bound-BX}
\|\mathbf X^{k}-\mathbf X^{k-1}\|_F\le B_X,\qquad \forall k\ge 1.
\end{equation}
Define
\begin{equation}
\label{eq:sudamuon-BG-def}
B_G^2 \;:=\; 2L_*^2 B_X^2 + 8N\sigma^2.
\end{equation}
Then the forcing term in Lemma~\ref{lem:sudamuon-gt-disagreement} admits the uniform bound
\begin{equation}
\label{eq:sudamuon-BM-explicit}
\sup_{k\ge 0}\mathbb{E}\|\tilde{\mathbf M}^{k+1}-\tilde{\mathbf M}^{k}\|_F^2
\;\le\; 2\,B_G^2.
\end{equation}
\end{lemma}
\begin{proof}
Lemma~\ref{lem:sudamuon-grad-diff-forcing} and~\eqref{eq:sudamuon-step-bound-BX} imply
$\sup_{k\ge 1}\mathbb{E}\|\tilde{\mathbf G}^{k}-\tilde{\mathbf G}^{k-1}\|_F^2\le B_G^2$.
Then Lemma~\ref{lem:sudamuon-ema-increment-disagreement} yields
$\sup_{k\ge 1}\mathbb{E}\|\tilde{\mathbf M}^{k}-\tilde{\mathbf M}^{k-1}\|_F^2\le 2B_G^2$.
Renaming the index ($k\mapsto k+1$) gives~\eqref{eq:sudamuon-BM-explicit}.
\end{proof}

%% ================================================================
\subsubsection*{Proofs of the Main Convergence Results}
%% ================================================================

\begin{proposition}[A closed stationarity bound]
\label{prop:sudamuon-closed-stationarity}
Assume Assumptions~\ref{assump:network-suda}--\ref{assump:muon}.
Let the iterates be generated by~\eqref{eq:sudamuon-framework} with initialization $H_i^0=M_i^0=G_i^0$ and $Y_i^0=0$ for all $i\in[N]$.
Assume additionally that $\mathbf{Y}^0=0$ and the stability condition in Lemma~\ref{lem:sudamuon-suda-consensus-bounded-input} hold (so that the constants $\gamma\in(0,1)$, $v_1,v_2\ge 1$, and $\lambda_a\ge 0$ in Lemma~\ref{lem:sudamuon-suda-consensus-bounded-input} exist).
Let $K\ge 1$.
Define $\Delta_0:=\mathbb{E}[f(\bar X^0)]-f_{\inf}$ and $r_{\max}:=\min\{m,n\}$.
Then
\begin{equation}
\label{eq:sudamuon-closed-stationarity}
\begin{split}
\frac{1}{K}\sum_{k=0}^{K-1}\mathbb{E}\big[\|\nabla f(\bar X^k)\|_*\big]
\;\le\;&\;
\frac{\Delta_0}{\alpha K}
+\frac{L_*\alpha}{2}
+2\,\mathcal T_{\rm cons}(K)
+2\,\mathcal T_{\rm drift}
+2\,\mathcal T_{\rm noise}(K)
+2\,\mathcal T_{\rm gt},
\end{split}
\end{equation}
where
\begin{align}
\mathcal T_{\rm cons}(K)
&:= \frac{L_*}{K\sqrt N}\sum_{k=0}^{K-1}\mathbb{E}\|\tilde{\mathbf X}^k\|_F
\le \frac{L_*v_1v_2}{K\sqrt N}\,\frac{\|\tilde{\mathbf X}^0\|_F}{1-\gamma}
+ \frac{2L_*\alpha v_1v_2\lambda_a}{1-\gamma}\,\sqrt{r_{\max}},
\label{eq:sudamuon-Tcons}
\\
\mathcal T_{\rm drift}
&:= \sup_{k\ge 0}\mathbb{E}\|D^k\|_*
\le \frac{\beta L_*}{1-\beta}
\Bigg[
\alpha
+\frac{4v_1v_2}{\sqrt N}\,\|\tilde{\mathbf X}^0\|_F
+\frac{8\alpha v_1v_2\lambda_a}{1-\gamma}\,\sqrt{r_{\max}}
\Bigg],
\label{eq:sudamuon-Tdrift}
\\
\mathcal T_{\rm noise}(K)
&:= \frac{1}{K}\sum_{k=0}^{K-1}\mathbb{E}\|\bar C^{k+1}-\bar M^{k+1}\|_*
\notag\\
&\quad\le \sqrt{r_{\max}}\Bigg[
\sqrt{\frac{1-\beta}{1+\beta}}\,\frac{\sigma}{\sqrt N}
+\frac{\sigma}{K\sqrt N}\,\frac{1}{1-\beta}
\Bigg],
\label{eq:sudamuon-Tnoise}
\\
\mathcal T_{\rm gt}
&:= \sqrt{\frac{r_{\max}}{N}}\,\sup_{k\ge 0}\mathbb{E}\|\tilde{\mathbf H}^{k+1}\|_F
\le \sqrt{\frac{r_{\max}}{N}}
\sqrt{\mathbb{E}\|\tilde{\mathbf H}^{0}\|_F^2
+ \frac{4\lambda^2(1+\lambda^2)}{(1-\lambda^2)^2}\,B_G^2},
\label{eq:sudamuon-Tgt}
\end{align}
with $\lambda:=\rho(W-J)\in(0,1)$ and
\begin{equation}
\label{eq:sudamuon-BG-in-prop}
B_G^2:=2L_*^2 B_X^2+8N\sigma^2,
\qquad
B_X:=\alpha\sqrt{Nr_{\max}}+2v_1v_2\|\tilde{\mathbf X}^0\|_F+\frac{4\alpha v_1v_2\lambda_a}{1-\gamma}\,\sqrt{Nr_{\max}}.
\end{equation}
\end{proposition}
\begin{proof}
Start from Proposition~\ref{prop:telescoping-descent}:
\[
\frac{1}{K}\sum_{k=0}^{K-1}\mathbb{E}\|\nabla f(\bar X^k)\|_*
\le \frac{\Delta_0}{\alpha K}+\frac{L_*\alpha}{2}
+\frac{2}{K}\sum_{k=0}^{K-1}\mathbb{E}\Big[\frac{1}{N}\sum_{i=1}^N\|\nabla f(\bar X^k)-H_i^{k+1}\|_*\Big].
\]
Apply Lemma~\ref{lem:sudamuon-tracking-error-decomp} to the averaged tracking error inside the expectation and sum over $k$.
This yields the decomposition into four contributions.

For the consensus term, apply Lemma~\ref{lem:sudamuon-suda-consensus-bounded-input} (with Muon directions) and sum the geometric series to obtain~\eqref{eq:sudamuon-Tcons}.

For the drift term, use the uniform bound from Lemma~\ref{lem:sudamuon-ema-drift-average}, giving~\eqref{eq:sudamuon-Tdrift}.

For the stochastic EMA term, use Lemma~\ref{lem:sudamuon-ema-stoch-error} and the nuclear-to-Frobenius conversion \eqref{eq:barC-barM-nuc-vs-F} from Lemma~\ref{lem:sudamuon-tracking-error-decomp}; averaging $\beta^{k+1}$ over $k=0,\dots,K-1$ yields~\eqref{eq:sudamuon-Tnoise}.

For the gradient-tracking disagreement term, use \eqref{eq:H-disagreement-nuc-vs-F} from Lemma~\ref{lem:sudamuon-tracking-error-decomp} and Jensen/Cauchy--Schwarz:
$\mathbb{E}\|\tilde{\mathbf H}^{k+1}\|_F\le \sqrt{\mathbb{E}\|\tilde{\mathbf H}^{k+1}\|_F^2}$.
Then apply Lemma~\ref{lem:sudamuon-gt-disagreement-uniform} together with Lemma~\ref{lem:sudamuon-ema-increment-disagreement-explicit}.
Finally, invoke Lemma~\ref{lem:sudamuon-stacked-step-bound} to bound $B_X$ as stated in~\eqref{eq:sudamuon-BG-in-prop}.
Collecting terms yields~\eqref{eq:sudamuon-closed-stationarity}.
\end{proof}

\begin{lemma}[Averaged tracking error bound with $\sqrt{1-\beta}$ noise scaling]
\label{lem:sudamuon-avg-gt-error-decomposition}
Assume Assumptions~\ref{assump:network-suda}--\ref{assump:stoch}.
Let the iterates be generated by~\eqref{eq:sudamuon-framework} with initialization $H_i^0=M_i^0$ for all $i\in[N]$.
Define the network averages
\[
\bar X^k:=\frac{1}{N}\sum_{i=1}^N X_i^k,\qquad
\bar H^k:=\frac{1}{N}\sum_{i=1}^N H_i^k,\qquad
\bar M^k:=\frac{1}{N}\sum_{i=1}^N M_i^k.
\]
Define the (average) tracking error
\[
E_{\mathrm{gt}}^k\;:=\;\bar H^k-\nabla f(\bar X^k)\in\mathbb{R}^{m\times n}.
\]
Let
\[
\bar G^k:=\frac{1}{N}\sum_{i=1}^N G_i^k,\qquad
\bar g^k:=\frac{1}{N}\sum_{i=1}^N \nabla f_i(X_i^k),\qquad
\bar\Delta^k:=\bar G^k-\bar g^k.
\]
Then for every integer $K\ge 1$ we have
\begin{equation}
\label{eq:sudamuon-avg-gt-error-decomp-main}
\begin{split}
\frac{1}{K}\sum_{k=0}^{K-1}\mathbb{E}\,\|E_{\mathrm{gt}}^{k+1}\|_F
\;\le\;&\;
\underbrace{\sqrt{\tfrac{1-\beta}{1+\beta}}\,\tfrac{\sigma}{\sqrt{N}}}_{\text{stochastic noise }\sim \frac{\sigma}{\sqrt N}\sqrt{1-\beta}}
\;+
\underbrace{\frac{L_*}{1-\beta}\cdot \frac{1}{K}\sum_{k=0}^{K-1}\mathbb{E}\,\|\bar X^{k+1}-\bar X^k\|_F}_{\text{drift }\sim \frac{L_*\alpha}{1-\beta}}
\\&\;+
\underbrace{\frac{1}{K}\sum_{k=0}^{K-1}\mathbb{E}\,\|\bar g^k-\nabla f(\bar X^k)\|_F}_{\text{consensus bias}}
\;+
\frac{1}{K}\sum_{k=0}^{K-1}\beta^{k+1}\,\mathbb{E}\,\|E_{\mathrm{gt}}^{0}\|_F.
\end{split}
\end{equation}
Moreover, under Assumption~\ref{assump:smooth}, the consensus bias can be bounded by
\begin{equation}
\label{eq:sudamuon-avg-gt-error-consensus-bias}
\|\bar g^k-\nabla f(\bar X^k)\|_F
\;\le\;
\frac{L_*}{\sqrt{N}}\,\|\tilde{\mathbf X}^k\|_F,
\qquad
\tilde{\mathbf X}^k:=(I-J)\mathbf X^k,\; J:=\tfrac{1}{N}\mathbf 1\mathbf 1^\top.
\end{equation}
\end{lemma}

\begin{proof}
By Lemma~\ref{lem:sudamuon-gt-average-invariance}, $\bar H^k=\bar M^k$ for all $k\ge 0$.
Averaging the EMA update gives
\begin{equation}
\label{eq:sudamuon-barM-rec}
\bar M^{k+1}=\beta\bar M^k+(1-\beta)\bar G^k.
\end{equation}
Fix $k\ge 0$. Using $E_{\mathrm{gt}}^k=\bar M^k-\nabla f(\bar X^k)$ and~\eqref{eq:sudamuon-barM-rec},
\begin{align}
E_{\mathrm{gt}}^{k+1}
&=\bar M^{k+1}-\nabla f(\bar X^{k+1})
=\beta\bar M^k+(1-\beta)\bar G^k-\nabla f(\bar X^{k+1})\nonumber\\
&=\beta(\bar M^k-\nabla f(\bar X^k))
+\bigl(\nabla f(\bar X^k)-\nabla f(\bar X^{k+1})\bigr)
+(1-\beta)\bigl(\bar G^k-\nabla f(\bar X^k)\bigr)\nonumber\\
&=\beta E_{\mathrm{gt}}^{k}
+\underbrace{(1-\beta)\bar\Delta^k}_{\text{noise}}
+\underbrace{(1-\beta)(\bar g^k-\nabla f(\bar X^k))}_{\text{consensus bias}}
+\underbrace{\bigl(\nabla f(\bar X^k)-\nabla f(\bar X^{k+1})\bigr)}_{\text{drift}}.
\label{eq:sudamuon-Egt-rec}
\end{align}

\medskip\noindent\textbf{Step 1: decompose $E_{\mathrm{gt}}^k=A^k+B^k$.}
Define $A^0:=0$ and $B^0:=E_{\mathrm{gt}}^0$, and for $k\ge 0$ define
\[
A^{k+1}:=\beta A^k+(1-\beta)\bigl(\bar\Delta^k+(\bar g^k-\nabla f(\bar X^k))\bigr),
\qquad
B^{k+1}:=\beta B^k+\bigl(\nabla f(\bar X^k)-\nabla f(\bar X^{k+1})\bigr).
\]
Then~\eqref{eq:sudamuon-Egt-rec} implies $E_{\mathrm{gt}}^{k}=A^{k}+B^{k}$ for all $k$.

\medskip\noindent\textbf{Step 2: bound the stochastic part (noise).}
Unrolling the $A$-recursion and keeping only the noise contribution yields
\[
A_{\rm noise}^{k+1}=(1-\beta)\sum_{t=0}^{k}\beta^{k-t}\bar\Delta^t.
\]
By Assumption~\ref{assump:stoch}, conditional on $\mathcal F^t$ the noises are independent across agents and satisfy
$\mathbb{E}[\bar\Delta^t\mid \mathcal F^t]=0$ and
$\mathbb{E}[\|\bar\Delta^t\|_F^2\mid \mathcal F^t]\le \sigma^2/N$.
Therefore cross terms vanish, and
\[
\mathbb{E}\|A_{\rm noise}^{k+1}\|_F
\le \sqrt{\mathbb{E}\|A_{\rm noise}^{k+1}\|_F^2}
\le (1-\beta)\sqrt{\sum_{t=0}^{k}\beta^{2(k-t)}\cdot \frac{\sigma^2}{N}}
\le (1-\beta)\sqrt{\frac{1}{1-\beta^2}}\cdot \frac{\sigma}{\sqrt N}
=\sqrt{\tfrac{1-\beta}{1+\beta}}\,\frac{\sigma}{\sqrt N}.
\]

\medskip\noindent\textbf{Step 3: bound the drift part.}
Unrolling $B^{k+1}=\beta B^k+d^k$ with $d^k:=\nabla f(\bar X^k)-\nabla f(\bar X^{k+1})$ gives
\[
B^{k+1}=\beta^{k+1}B^0+\sum_{t=0}^{k}\beta^{k-t}d^t.
\]
Thus, by triangle inequality,
\[
\mathbb{E}\|B^{k+1}\|_F
\le \beta^{k+1}\,\mathbb{E}\|B^0\|_F
+\sum_{t=0}^{k}\beta^{k-t}\,\mathbb{E}\|d^t\|_F.
\]
Now average over $k=0,\dots,K-1$ and exchange the order of summation:
\[
\frac{1}{K}\sum_{k=0}^{K-1}\sum_{t=0}^{k}\beta^{k-t}\,\mathbb{E}\|d^t\|_F
=\frac{1}{K}\sum_{t=0}^{K-1}\mathbb{E}\|d^t\|_F\sum_{k=t}^{K-1}\beta^{k-t}
\le \frac{1}{1-\beta}\cdot \frac{1}{K}\sum_{t=0}^{K-1}\mathbb{E}\|d^t\|_F.
\]
Therefore,
\[
\frac{1}{K}\sum_{k=0}^{K-1}\mathbb{E}\|B^{k+1}\|_F
\le \frac{1}{K}\sum_{k=0}^{K-1}\beta^{k+1}\,\mathbb{E}\|B^0\|_F
+\frac{1}{1-\beta}\cdot \frac{1}{K}\sum_{t=0}^{K-1}\mathbb{E}\|d^t\|_F.
\]
Under Assumption~\ref{assump:smooth},
$\|d^t\|_F\le \|d^t\|_*\le L_*\|\bar X^{t+1}-\bar X^t\|$, and $\|\cdot\|\le \|\cdot\|_F$.
This yields the drift term in~\eqref{eq:sudamuon-avg-gt-error-decomp-main}.

\medskip\noindent\textbf{Step 4: bound the consensus bias.}
The consensus-bias contribution in $A^{k+1}$ is
\[
A_{\rm bias}^{k+1}=(1-\beta)\sum_{t=0}^{k}\beta^{k-t}\bigl(\bar g^t-\nabla f(\bar X^t)\bigr).
\]
By triangle inequality,
\[
\mathbb{E}\|A_{\rm bias}^{k+1}\|_F
\le (1-\beta)\sum_{t=0}^{k}\beta^{k-t}\,\mathbb{E}\|\bar g^t-\nabla f(\bar X^t)\|_F.
\]
Now average over $k=0,\dots,K-1$ and exchange summations:
\[
\frac{1}{K}\sum_{k=0}^{K-1}\mathbb{E}\|A_{\rm bias}^{k+1}\|_F
\le \frac{1}{K}\sum_{t=0}^{K-1}\mathbb{E}\|\bar g^t-\nabla f(\bar X^t)\|_F\cdot (1-\beta)\sum_{k=t}^{K-1}\beta^{k-t}
\le \frac{1}{K}\sum_{t=0}^{K-1}\mathbb{E}\|\bar g^t-\nabla f(\bar X^t)\|_F,
\]
where we used $(1-\beta)\sum_{s=0}^{\infty}\beta^s=1$.
This yields the bias term in~\eqref{eq:sudamuon-avg-gt-error-decomp-main}.

Combining Steps 2--4 and using $E_{\mathrm{gt}}^{k+1}=A^{k+1}+B^{k+1}$ proves~\eqref{eq:sudamuon-avg-gt-error-decomp-main}.
Finally,~\eqref{eq:sudamuon-avg-gt-error-consensus-bias} follows from the same argument as Lemma~\ref{lem:sudamuon-grad-mismatch-consensus}:
\[
\|\bar g^k-\nabla f(\bar X^k)\|_F
=\Big\|\frac{1}{N}\sum_{i=1}^N\bigl(\nabla f_i(X_i^k)-\nabla f_i(\bar X^k)\bigr)\Big\|_F
\le \frac{1}{N}\sum_{i=1}^N L_*\|X_i^k-\bar X^k\|_F
\le \frac{L_*}{\sqrt N}\|\tilde{\mathbf X}^k\|_F.
\]
\end{proof}

\begin{lemma}[Gradient-tracking disagreement bound (no heterogeneity term)]
\label{lem:sudamuon-gt-disagreement-O1mb}
Assume Assumptions~\ref{assump:network-suda}(1), \ref{assump:smooth}, and~\ref{assump:stoch}.
Let $(\mathbf X^k,\mathbf M^k,\mathbf H^k)$ be generated by~\eqref{eq:sudamuon-framework}.
Define the disagreement operator $\tilde{\mathbf Z}:=(I-J)\mathbf Z$ with $J:=\tfrac1N\mathbf 1\mathbf 1^\top$.
Let $\lambda:=\rho(W-J)\in(0,1)$.
Then, for every integer $K\ge 1$,
\begin{equation}
\label{eq:lem18-avg-tildeH-basic}
\frac{1}{K}\sum_{k=0}^{K-1}\mathbb{E}\,\|\tilde{\mathbf H}^{k+1}\|_F
\;\le\;
\frac{\lambda}{K(1-\lambda)}\,\mathbb{E}\,\|\tilde{\mathbf H}^{0}\|_F
+\frac{\lambda}{1-\lambda}\cdot \frac{1}{K}\sum_{k=0}^{K-1}\mathbb{E}\,\|\tilde{\mathbf M}^{k+1}-\tilde{\mathbf M}^{k}\|_F.
\end{equation}
Moreover, with $r_{\max}:=\min\{m,n\}$,
\begin{equation}
\label{eq:lem18-avg-tildeM-incr-refined}
\begin{split}
\frac{1}{K}\sum_{k=0}^{K-1}\mathbb{E}\,\|\tilde{\mathbf M}^{k+1}-\tilde{\mathbf M}^{k}\|_F
\;\le\;&\;
\frac{\mathbb{E}\,\|\widetilde{\nabla \mathbf f}(\mathbf X^0)-\tilde{\mathbf M}^0\|_F}{K}
+2(1-\beta)\,\sigma\sqrt N
\\&\;+L_*\Bigg(\alpha\,\sqrt{Nr_{\max}}+\frac{2}{K}\sum_{k=0}^{K-1}\mathbb{E}\,\|\tilde{\mathbf X}^{k}\|_F+\frac{1}{K}\,\mathbb{E}\,\|\tilde{\mathbf X}^{K}\|_F\Bigg).
\end{split}
\end{equation}
Combining~\eqref{eq:lem18-avg-tildeH-basic}--\eqref{eq:lem18-avg-tildeM-incr-refined} yields
\begin{equation}
\label{eq:lem18-avg-tildeH-refined}
\begin{split}
\frac{1}{K}&\sum_{k=0}^{K-1}\mathbb{E}\,\|\tilde{\mathbf H}^{k+1}\|_F
\;\le\;
\frac{\lambda}{K(1-\lambda)}\,\mathbb{E}\,\|\tilde{\mathbf H}^{0}\|_F
+\frac{\lambda}{1-\lambda}\Bigg[
\frac{\mathbb{E}\,\|\widetilde{\nabla \mathbf f}(\mathbf X^0)-\tilde{\mathbf M}^0\|_F}{K}
\\&\quad+L_*\Bigg(\alpha\,\sqrt{Nr_{\max}}+\frac{2}{K}\sum_{k=0}^{K-1}\mathbb{E}\,\|\tilde{\mathbf X}^{k}\|_F+\frac{1}{K}\,\mathbb{E}\,\|\tilde{\mathbf X}^{K}\|_F\Bigg)
+2(1-\beta)\,\sigma\sqrt N\Bigg].
\end{split}
\end{equation}
In particular, if Lemma~\ref{lem:sudamuon-suda-consensus-bounded-input} holds with Muon directions, then
\begin{equation}
\label{eq:lem18-avg-tildeH-explicit}
\begin{split}
\frac{1}{K}&\sum_{k=0}^{K-1}\mathbb{E}\,\|\tilde{\mathbf H}^{k+1}\|_F
\;\le\;
\frac{\lambda}{K(1-\lambda)}\,\mathbb{E}\,\|\tilde{\mathbf H}^{0}\|_F
+\frac{\lambda}{1-\lambda}\Bigg[
\frac{\mathbb{E}\,\|\widetilde{\nabla \mathbf f}(\mathbf X^0)-\tilde{\mathbf M}^0\|_F}{K}
+L_*\alpha\,\sqrt{Nr_{\max}}
\\
&\quad+\frac{2L_*v_1v_2(1+\|B^2\|)}{K(1-\gamma)}\,\|\tilde{\mathbf X}^{0}\|_F
+\frac{L_*v_1v_2(1+\|B^2\|)}{K}\,\gamma^{K}\,\|\tilde{\mathbf X}^{0}\|_F
+\frac{2L_*\alpha v_1v_2\lambda_a}{(1-\gamma)K}\,\sqrt{Nr_{\max}}
\\
&\quad+\frac{4L_*\alpha v_1v_2\lambda_a}{1-\gamma}\,\sqrt{Nr_{\max}}
+2(1-\beta)\,\sigma\sqrt N\Bigg].
\end{split}
\end{equation}
\end{lemma}

\begin{proof}
\medskip\noindent\textbf{Step 1: unroll the GT disagreement recursion.}
From Lemma~\ref{lem:sudamuon-gt-disagreement} (non-squared form),
\[
\|\tilde{\mathbf H}^{k+1}\|_F
\le \lambda\,\|\tilde{\mathbf H}^{k} + (\tilde{\mathbf M}^{k+1}-\tilde{\mathbf M}^{k})\|_F
\le \lambda\,\|\tilde{\mathbf H}^{k}\|_F+\lambda\,\|\tilde{\mathbf M}^{k+1}-\tilde{\mathbf M}^{k}\|_F.
\]
Unrolling this scalar recursion and summing the resulting geometric series yields~\eqref{eq:lem18-avg-tildeH-basic}.

\medskip\noindent\textbf{Step 2: bound the average EMA increment without a heterogeneity term.}
Define the local EMA tracking error
\[
\mathbf E^k := \nabla \mathbf f(\mathbf X^k)-\mathbf M^k,
\qquad \tilde{\mathbf E}^k := (I-J)\mathbf E^k=\widetilde{\nabla \mathbf f}(\mathbf X^k)-\tilde{\mathbf M}^k.
\]
From the EMA update $\mathbf M^{k+1}=\beta\mathbf M^k+(1-\beta)\mathbf G^k$ and the decomposition
$\mathbf G^k=\nabla \mathbf f(\mathbf X^k)+\mathbf\Delta^k$, we have
\[
\mathbf E^{k+1}
=\nabla \mathbf f(\mathbf X^{k+1})-\beta\mathbf M^k-(1-\beta)\nabla \mathbf f(\mathbf X^k)-(1-\beta)\mathbf\Delta^k
=\beta\mathbf E^k+\big(\nabla \mathbf f(\mathbf X^{k+1})-\nabla \mathbf f(\mathbf X^k)\big)-(1-\beta)\mathbf\Delta^k.
\]
Applying $(I-J)$ and using its non-expansiveness yields the recursion
\begin{equation}
\label{eq:lem18-E-tilde-rec}
\|\tilde{\mathbf E}^{k+1}\|_F
\le \beta\,\|\tilde{\mathbf E}^{k}\|_F
+\|\widetilde{\nabla \mathbf f}(\mathbf X^{k+1})-\widetilde{\nabla \mathbf f}(\mathbf X^k)\|_F
+(1-\beta)\,\|\tilde{\mathbf\Delta}^k\|_F.
\end{equation}
Unrolling~\eqref{eq:lem18-E-tilde-rec} and averaging over $k=0,\dots,K-1$ gives
\[
\frac{1}{K}\sum_{k=0}^{K-1}\mathbb{E}\,\|\tilde{\mathbf E}^{k}\|_F
\le \frac{\mathbb{E}\,\|\tilde{\mathbf E}^{0}\|_F}{K(1-\beta)}
+\frac{1}{K(1-\beta)}\sum_{k=0}^{K-1}\mathbb{E}\,\|\widetilde{\nabla \mathbf f}(\mathbf X^{k+1})-\widetilde{\nabla \mathbf f}(\mathbf X^k)\|_F
+\frac{1}{K}\sum_{k=0}^{K-1}\mathbb{E}\,\|\tilde{\mathbf\Delta}^k\|_F,
\]
where we used the weight identity
$(1-\beta)\sum_{s=0}^{\infty}\beta^s=1$.

Next, since each $f_i$ is $L_*$-smooth in the sense of Assumption~\ref{assump:smooth}, we have for each $i$ and any $X,Y$ that
$\|\nabla f_i(X)-\nabla f_i(Y)\|_F\le \|\nabla f_i(X)-\nabla f_i(Y)\|_*\le L_*\|X-Y\|\le L_*\|X-Y\|_F$.
Stacking and using non-expansiveness of $(I-J)$ yields
\[
\|\widetilde{\nabla \mathbf f}(\mathbf X^{k+1})-\widetilde{\nabla \mathbf f}(\mathbf X^k)\|_F
\le L_*\,\|\mathbf X^{k+1}-\mathbf X^k\|_F.
\]
Moreover, by Assumption~\ref{assump:stoch} and Jensen,
$\mathbb{E}\,\|\tilde{\mathbf\Delta}^k\|_F\le \mathbb{E}\,\|\mathbf\Delta^k\|_F\le \sigma\sqrt N$.
Therefore,
\begin{equation}
\label{eq:lem18-avg-Etilde}
\frac{1}{K}\sum_{k=0}^{K-1}\mathbb{E}\,\|\tilde{\mathbf E}^{k}\|_F
\le \frac{\mathbb{E}\,\|\tilde{\mathbf E}^{0}\|_F}{K(1-\beta)}
+\frac{L_*}{1-\beta}\cdot\frac{1}{K}\sum_{k=0}^{K-1}\mathbb{E}\,\|\mathbf X^{k+1}-\mathbf X^k\|_F
+\sigma\sqrt N.
\end{equation}

Finally, note that
$\tilde{\mathbf M}^{k+1}-\tilde{\mathbf M}^{k}=(1-\beta)\big(\tilde{\mathbf G}^k-\tilde{\mathbf M}^k\big)=(1-\beta)\big(\tilde{\mathbf\Delta}^k+\tilde{\mathbf E}^k\big)$.
Thus, using triangle inequality and~\eqref{eq:lem18-avg-Etilde},
\[
\frac{1}{K}\sum_{k=0}^{K-1}\mathbb{E}\,\|\tilde{\mathbf M}^{k+1}-\tilde{\mathbf M}^{k}\|_F
\le (1-\beta)\cdot\frac{1}{K}\sum_{k=0}^{K-1}\mathbb{E}\,\|\tilde{\mathbf\Delta}^k\|_F
+(1-\beta)\cdot\frac{1}{K}\sum_{k=0}^{K-1}\mathbb{E}\,\|\tilde{\mathbf E}^k\|_F
\]
\[
\le (1-\beta)\sigma\sqrt N
+\frac{\mathbb{E}\,\|\tilde{\mathbf E}^{0}\|_F}{K}
+L_*\cdot\frac{1}{K}\sum_{k=0}^{K-1}\mathbb{E}\,\|\mathbf X^{k+1}-\mathbf X^k\|_F
+(1-\beta)\sigma\sqrt N.
\]
It remains to bound the average primal increment.
Write $\mathbf X^k=\bar{\mathbf X}^k+\tilde{\mathbf X}^k$ with $\bar{\mathbf X}^k:=J\mathbf X^k$.
Then
\[
\|\mathbf X^{k+1}-\mathbf X^{k}\|_F
\le \|\bar{\mathbf X}^{k+1}-\bar{\mathbf X}^{k}\|_F+\|\tilde{\mathbf X}^{k+1}-\tilde{\mathbf X}^{k}\|_F.
\]
Since $\bar{\mathbf X}^k$ is a stacked consensus variable with all blocks equal to $\bar X^k$, we have
$\|\bar{\mathbf X}^{k+1}-\bar{\mathbf X}^{k}\|_F=\sqrt N\,\|\bar X^{k+1}-\bar X^{k}\|_F$.
By Lemma~\ref{lem:sudamuon-average-dynamics}, $\|\bar X^{k+1}-\bar X^{k}\|_F\le \alpha\sqrt{r_{\max}}$.
Moreover, $\|\tilde{\mathbf X}^{k+1}-\tilde{\mathbf X}^{k}\|_F\le \|\tilde{\mathbf X}^{k+1}\|_F+\|\tilde{\mathbf X}^{k}\|_F$.
Averaging over $k=0,\dots,K-1$ yields
\[
\frac{1}{K}\sum_{k=0}^{K-1}\mathbb{E}\,\|\mathbf X^{k+1}-\mathbf X^{k}\|_F
\le \alpha\sqrt{Nr_{\max}}+\frac{2}{K}\sum_{k=0}^{K-1}\mathbb{E}\,\|\tilde{\mathbf X}^{k}\|_F+\frac{1}{K}\,\mathbb{E}\,\|\tilde{\mathbf X}^{K}\|_F.
\]
Substituting into the previous display gives~\eqref{eq:lem18-avg-tildeM-incr-refined}.

\medskip\noindent\textbf{Step 3: plug into the GT bound and (optionally) use Lemma~\ref{lem:sudamuon-suda-consensus-bounded-input}.}
Combining Step 1 and Step 2 yields~\eqref{eq:lem18-avg-tildeH-refined}.
If Lemma~\ref{lem:sudamuon-suda-consensus-bounded-input} holds with Muon directions, then
$\|\tilde{\mathbf X}^k\|_F\le v_1v_2\gamma^k(1+\|B^2\|)\|\tilde{\mathbf X}^0\|_F+\frac{2\alpha v_1v_2\lambda_a}{1-\gamma}\sqrt{Nr_{\max}}$.
Averaging and using $\frac1K\sum_{k=0}^{K-1}\gamma^k\le \frac{1}{K(1-\gamma)}$ gives
\[
\frac{1}{K}\sum_{k=0}^{K-1}\mathbb{E}\,\|\tilde{\mathbf X}^{k}\|_F
\le \frac{v_1v_2(1+\|B^2\|)}{K(1-\gamma)}\,\|\tilde{\mathbf X}^{0}\|_F
+\frac{2\alpha v_1v_2\lambda_a}{1-\gamma}\,\sqrt{Nr_{\max}},
\]
which implies~\eqref{eq:lem18-avg-tildeH-explicit}.
\end{proof}

\begin{proposition}[Transient complexity with explicit (fast) network-dependent decay]
\label{prop:sudamuon-final-convergence-K-1-4}
Assume Assumptions~\ref{assump:network-suda}--\ref{assump:muon}.
Let the iterates be generated by~\eqref{eq:sudamuon-framework} with initialization
$\mathbf Y^0=0$, $\mathbf M^0=\mathbf G^0$, and $\mathbf H^0=\mathbf M^0$ as in Algorithm~\ref{alg:suda-muon}.
Assume also the stability condition in Lemma~\ref{lem:sudamuon-suda-consensus-bounded-input} so that its constants $(\gamma,v_1,v_2,\lambda_a)$ exist.
Fix a horizon $K\ge 1$ and choose
\[
\alpha:=\frac{\alpha_0}{K^{3/4}},\qquad 1-\beta:=\frac{b_0}{\sqrt K},
\qquad \text{with fixed constants }\alpha_0>0,\; b_0\in(0,1].
\]
Let $r_{\max}:=\min\{m,n\}$ and $\Delta_0:=\mathbb{E}[f(\bar X^0)]-f_{\inf}$.
Let $\lambda:=\rho(W-J)\in(0,1)$.
Define the average tracking error
\[
E_{\rm gt}^k:=\bar H^k-\nabla f(\bar X^k),
\qquad \bar H^k:=\frac1N\sum_{i=1}^N H_i^k,\quad \bar X^k:=\frac1N\sum_{i=1}^N X_i^k.
\]
Then
{\small
\begin{equation}
\label{eq:sudamuon-final-rate-K-1-4-refined}
\begin{aligned}
&\frac{1}{K}\sum_{k=0}^{K-1}\mathbb{E}\,\|\nabla f(\bar X^k)\|_*
\;\le\;
\underbrace{\Bigg(
\frac{\Delta_0}{\alpha_0}
+\frac{2\sigma\sqrt{r_{\max}}}{\sqrt N}\,\sqrt{b_0}
+\frac{2L_*\alpha_0 r_{\max}}{b_0}
\Bigg)K^{-1/4}}_{\text{topology-free leading term}}
\\[0.25em]
&\;+\underbrace{\Bigg(
\frac{5}{2}L_*\alpha_0
+\frac{4L_*\alpha_0 v_1v_2\lambda_a}{1-\gamma}\,r_{\max}
+\frac{2\lambda}{1-\lambda}\,L_*\alpha_0\,r_{\max}
+\frac{8\lambda}{1-\lambda}\,\frac{L_*\alpha_0 v_1v_2\lambda_a}{1-\gamma}\,r_{\max}
\Bigg)K^{-3/4}}_{\text{network-dependent drift terms (no $\sigma$)}}
\\[0.25em]
&\;+\underbrace{\Bigg(
\frac{2\sqrt{r_{\max}}}{b_0}\,\mathbb{E}\|E_{\rm gt}^0\|_F
+\frac{4\lambda}{1-\lambda}\,b_0\,\sigma\sqrt{r_{\max}}
\Bigg)K^{-1/2}}_{\text{higher-order stochastic terms}}
\\[0.25em]
&\;+\Bigg(
\frac{2L_*v_1v_2(1+\|B^2\|)}{(1-\gamma)\sqrt N}\,\sqrt{r_{\max}}\,\|\tilde{\mathbf X}^0\|_F
+\frac{4\lambda}{1-\lambda}\,\frac{L_*v_1v_2(1+\|B^2\|)}{(1-\gamma)\sqrt N}\,\sqrt{r_{\max}}\,\|\tilde{\mathbf X}^0\|_F
\\
&\qquad+\frac{2\lambda}{1-\lambda}\,\frac{L_*v_1v_2(1+\|B^2\|)}{\sqrt N}\,\sqrt{r_{\max}}\,\|\tilde{\mathbf X}^0\|_F
+\frac{2\lambda}{1-\lambda}\,\sqrt{\frac{r_{\max}}{N}}\,\mathbb{E}\|\widetilde{\nabla \mathbf f}(\mathbf X^0)-\tilde{\mathbf M}^0\|_F
\\
&\qquad+\frac{2\lambda}{1-\lambda}\,\sqrt{\frac{r_{\max}}{N}}\,\mathbb{E}\|\tilde{\mathbf H}^0\|_F
\Bigg)\underbrace{K^{-1}}_{\text{init.\ transients}}
\\[0.25em]
&\;+\underbrace{\frac{4\lambda}{1-\lambda}\,\frac{L_*\alpha_0 v_1v_2\lambda_a}{1-\gamma}\,r_{\max}\,K^{-7/4}}_{\text{fast transient from }\alpha/K}.
\end{aligned}
\end{equation}
}
In particular,
\[
\lim_{K\to\infty}\frac{1}{K}\sum_{k=0}^{K-1}\mathbb{E}\,\|\nabla f(\bar X^k)\|_*\;=\;0.
\]
\end{proposition}

\begin{proof}
\medskip\noindent\textbf{Step 1: start from the telescoping descent inequality.}
Proposition~\ref{prop:telescoping-descent} gives
\begin{equation}
\label{eq:pf-final-start}
\frac{1}{K}\sum_{k=0}^{K-1}\mathbb{E}\,\|\nabla f(\bar X^k)\|_*
\le
\frac{\Delta_0}{\alpha K}
+\frac{L_*\alpha}{2}
+\frac{2}{K}\sum_{k=0}^{K-1}\mathbb{E}\Big[\frac{1}{N}\sum_{i=1}^N\|\nabla f(\bar X^k)-H_i^{k+1}\|_*\Big].
\end{equation}

\medskip\noindent\textbf{Step 2: reduce the tracking mismatch to averaged GT error and GT disagreement.}
For each $k$,
\[
\mathbb{E}\Big[\frac{1}{N}\sum_{i=1}^N\|\nabla f(\bar X^k)-H_i^{k+1}\|_*\Big]
\le L_*\alpha
+\sqrt{r_{\max}}\,\mathbb{E}\|E_{\rm gt}^{k+1}\|_F
+\sqrt{\frac{r_{\max}}{N}}\,\mathbb{E}\|\tilde{\mathbf H}^{k+1}\|_F.
\]
Averaging over $k=0,\dots,K-1$ and substituting into~\eqref{eq:pf-final-start} yields
\begin{equation}
\label{eq:pf-final-master}
\begin{split}
\frac{1}{K}\sum_{k=0}^{K-1}\mathbb{E}\,\|\nabla f(\bar X^k)\|_*
\le&\;
\frac{\Delta_0}{\alpha K}
+\frac{5}{2}L_*\alpha
+2\sqrt{r_{\max}}\cdot \frac{1}{K}\sum_{k=0}^{K-1}\mathbb{E}\|E_{\rm gt}^{k+1}\|_F
\\&\;+2\sqrt{\frac{r_{\max}}{N}}\cdot \frac{1}{K}\sum_{k=0}^{K-1}\mathbb{E}\|\tilde{\mathbf H}^{k+1}\|_F.
\end{split}
\end{equation}

\medskip\noindent\textbf{Step 3: bound the averaged GT error and the averaged GT disagreement.}
Lemma~\ref{lem:sudamuon-avg-gt-error-decomposition} gives
\begin{align}
\label{eq:pf-final-Egt-bound}
\frac{1}{K}\sum_{k=0}^{K-1}\mathbb{E}\|E_{\rm gt}^{k+1}\|_F
&\le \sqrt{\tfrac{1-\beta}{1+\beta}}\,\tfrac{\sigma}{\sqrt N}
+\frac{L_*}{1-\beta}\cdot \frac{1}{K}\sum_{k=0}^{K-1}\mathbb{E}\|\bar X^{k+1}-\bar X^k\|_F
\notag\\
&\quad +\frac{1}{K}\sum_{k=0}^{K-1}\mathbb{E}\|\bar g^k-\nabla f(\bar X^k)\|_F
+\frac{1}{K}\sum_{k=0}^{K-1}\beta^{k+1}\,\mathbb{E}\|E_{\rm gt}^{0}\|_F.
\end{align}
By Lemma~\ref{lem:sudamuon-average-dynamics}, $\|\bar X^{k+1}-\bar X^k\|_F\le \alpha\sqrt{r_{\max}}$.
Moreover, by~\eqref{eq:sudamuon-avg-gt-error-consensus-bias} and Lemma~\ref{lem:sudamuon-suda-consensus-bounded-input} (Muon case),
\[
\frac{1}{K}\sum_{k=0}^{K-1}\mathbb{E}\|\bar g^k-\nabla f(\bar X^k)\|_F
\le \frac{L_*}{\sqrt N}\cdot \frac{1}{K}\sum_{k=0}^{K-1}\mathbb{E}\|\tilde{\mathbf X}^k\|_F.
\]
By Lemma~\ref{lem:sudamuon-suda-consensus-bounded-input} with Muon directions (using~\eqref{eq:sudamuon-suda-consensus-bounded-input-bdd}), for every $k\ge 0$,
\[
\|\tilde{\mathbf X}^k\|_F
\le v_1v_2\gamma^k(1+\|B^2\|)\|\tilde{\mathbf X}^0\|_F
+\frac{2\alpha v_1v_2\lambda_a}{1-\gamma}\,\sqrt{Nr_{\max}},
\]
and hence, averaging over $k=0,\dots,K-1$ and using $\frac1K\sum_{k=0}^{K-1}\gamma^k\le \frac{1}{K(1-\gamma)}$,
\[
\frac{1}{K}\sum_{k=0}^{K-1}\mathbb{E}\|\tilde{\mathbf X}^k\|_F
\le \frac{v_1v_2(1+\|B^2\|)}{(1-\gamma)K}\,\|\tilde{\mathbf X}^0\|_F
+\frac{2\alpha v_1v_2\lambda_a}{1-\gamma}\,\sqrt{Nr_{\max}}.
\]
In particular, the primal disagreement is of order $O(\alpha)$ up to an exponentially decaying/initial transient, and this bound is \emph{independent of the gradient noise level $\sigma$} thanks to the uniform boundedness of Muon directions.

For the GT disagreement, Lemma~\ref{lem:sudamuon-gt-disagreement-O1mb} yields~\eqref{eq:lem18-avg-tildeH-refined}.
In this bound, the stochastic-noise contribution enters multiplied by $(1-\beta)$ (via the EMA increment), giving a term of the form
\[
\frac{\lambda}{1-\lambda}\,(1-\beta)\,\sigma\sqrt N,
\]
which becomes a higher-order term of order $(1-\beta)\sigma=O(K^{-1/2})$ under $1-\beta=b_0K^{-1/2}$.
The remaining deterministic contribution in~\eqref{eq:lem18-avg-tildeH-refined} is controlled by smoothness through the primal increments and the primal disagreement; under Muon, these increments are $O(\alpha)$.
Moreover, Lemma~\ref{lem:sudamuon-gt-disagreement-O1mb} contains an additional boundary term of the form $\tfrac{L_*}{K}\,\mathbb{E}\|\tilde{\mathbf X}^K\|_F$.
Using Lemma~\ref{lem:sudamuon-suda-consensus-bounded-input} (Muon case) at time $k=K$ and $\gamma^K\le 1$ yields
\[
\frac{1}{K}\,\mathbb{E}\|\tilde{\mathbf X}^K\|_F
\le \frac{v_1v_2(1+\|B^2\|)}{K}\,\|\tilde{\mathbf X}^0\|_F+\frac{2\alpha v_1v_2\lambda_a}{(1-\gamma)K}\,\sqrt{Nr_{\max}}.
\]
The first term contributes an additional $K^{-1}$ initial transient, while the second contributes a faster term of order $\alpha/K$.
Substituting the GT disagreement bound~\eqref{eq:lem18-avg-tildeH-refined} (together with the averaged primal-disagreement bound above) into~\eqref{eq:pf-final-master} yields the claimed $K$-dependent decomposition.

\medskip\noindent\textbf{Step 4: substitute $\alpha=\alpha_0K^{-3/4}$ and $1-\beta=b_0K^{-1/2}$.}
Collect powers of $K$ using
\[
\frac{\Delta_0}{\alpha K}=\frac{\Delta_0}{\alpha_0}K^{-1/4},\qquad
\sqrt{1-\beta}=\sqrt{b_0}\,K^{-1/4},\qquad
\frac{\alpha}{1-\beta}=\frac{\alpha_0}{b_0}K^{-1/4},
\]
\[
\frac{1}{K}\sum_{k=0}^{K-1}\beta^{k+1}\le \frac{1}{K(1-\beta)}=\frac{1}{b_0}K^{-1/2},\qquad
\alpha=\alpha_0K^{-3/4},\qquad
\frac{\alpha}{K}=\alpha_0K^{-7/4},
\]
\[
1-\beta=b_0K^{-1/2},\qquad
\frac{1}{K}=K^{-1}.
\]
This yields~\eqref{eq:sudamuon-final-rate-K-1-4-refined}.
\end{proof}

\begin{corollary}[Big-O form of the SUDA--Muon convergence rate]
\label{cor:sudamuon-final-convergence-bigO}
Under the assumptions of Proposition~\ref{prop:sudamuon-final-convergence-K-1-4}
and with the same stepsize and EMA choices
$\alpha = \alpha_0 K^{-3/4}$ and $1-\beta = b_0 K^{-1/2}$,
there exists a constant $C>0$ (independent of $K$, $N$, $\lambda$, $\gamma$, and $\sigma$)
such that for all integers $K\ge 1$,
\begin{equation}
\label{eq:sudamuon-final-rate-bigO}
\begin{split}
\frac{1}{K}\sum_{k=0}^{K-1} \mathbb{E}\big[\|\nabla f(\bar X^k)\|_*\big]
\;\le\; C\Bigg(&
\underbrace{\Bigl(1+\tfrac{\sigma}{\sqrt{N}}\Bigr)K^{-1/4}}_{\text{dominant topology-free term}}
+\underbrace{\frac{K^{-3/4}}{(1-\gamma)(1-\lambda)}}_{\text{network-dependent drift terms}}
\\
&+\underbrace{\Bigl(1+\frac{\sigma}{1-\lambda}\Bigr)K^{-1/2}}_{\text{higher-order stochastic network term}}
+\underbrace{\frac{K^{-1}}{\sqrt{N}(1-\gamma)(1-\lambda)}}_{\text{network-dependent initial transient}}
\\
&+\underbrace{\frac{K^{-7/4}}{(1-\gamma)(1-\lambda)}}_{\text{fast transient}}
\Bigg).
\end{split}
\end{equation}
Equivalently,
\begin{multline*}
\frac{1}{K}\sum_{k=0}^{K-1} \mathbb{E}\big[\|\nabla f(\bar X^k)\|_*\big]
= \mathcal{O}\Bigl(
\bigl(1+\tfrac{\sigma}{\sqrt{N}}\bigr)K^{-1/4}
+ \tfrac{K^{-3/4}}{(1-\gamma)(1-\lambda)}\\
+ \Bigl(1+\tfrac{\sigma}{1-\lambda}\Bigr)K^{-1/2}
+ \tfrac{K^{-1}}{\sqrt{N}(1-\gamma)(1-\lambda)}
+ \tfrac{K^{-7/4}}{(1-\gamma)(1-\lambda)}
\Bigr).
\end{multline*}
\end{corollary}

\begin{proof}
Starting from the refined bound~\eqref{eq:sudamuon-final-rate-K-1-4-refined} in
Proposition~\ref{prop:sudamuon-final-convergence-K-1-4}, note that it decomposes into
five groups of terms with powers $K^{-1/4}$, $K^{-3/4}$, $K^{-1/2}$, $K^{-1}$, and $K^{-7/4}$.
We upper-bound each coefficient by a universal constant times the desired
$(K,N,1-\lambda,1-\gamma,\sigma)$-dependent factor.

\medskip\noindent\textbf{1.\ $K^{-1/4}$ term.}
The coefficient of $K^{-1/4}$ in~\eqref{eq:sudamuon-final-rate-K-1-4-refined} is
\[
\frac{\Delta_0}{\alpha_0}
+ \frac{2\sigma\sqrt{r_{\max}}}{\sqrt N}\,\sqrt{b_0}
+ \frac{2L_*\alpha_0 r_{\max}}{b_0},
\]
which is bounded by a constant depending only on
$(\Delta_0, L_*, r_{\max}, \alpha_0, b_0)$ times $1+\sigma/\sqrt{N}$.
This yields the first term in~\eqref{eq:sudamuon-final-rate-bigO}.

\medskip\noindent\textbf{2.\ $K^{-3/4}$ term.}
The $K^{-3/4}$ coefficient in~\eqref{eq:sudamuon-final-rate-K-1-4-refined}
contains only deterministic drift terms and depends on $\gamma$ and $\lambda$ via
expressions of the form $1/(1-\gamma)$ and $\lambda/(1-\lambda)$.
Using $0<\gamma<1$ and $0<\lambda<1$, we have $1\le 1/(1-\gamma)$ and
$\lambda/(1-\lambda)\le 1/(1-\lambda)$, so the entire coefficient is bounded by
a constant multiple of $1/\bigl((1-\gamma)(1-\lambda)\bigr)$.
This gives the second term in~\eqref{eq:sudamuon-final-rate-bigO}.

\medskip\noindent\textbf{3.\ $K^{-1/2}$ term.}
The $K^{-1/2}$ group in~\eqref{eq:sudamuon-final-rate-K-1-4-refined} consists of
\[
\frac{2\sqrt{r_{\max}}}{b_0}\,\mathbb{E}\|E_{\rm gt}^0\|_F
+ \frac{4\lambda}{1-\lambda}\,b_0\sigma\sqrt{r_{\max}},
\]
which we bound by a constant (independent of $K$, $N$, $\lambda$, $\gamma$, $\sigma$)
plus a constant multiple of $\sigma/(1-\lambda)$.
Absorbing the purely deterministic part into $C$ yields the third term
in~\eqref{eq:sudamuon-final-rate-bigO}.

\medskip\noindent\textbf{4.\ $K^{-1}$ term.}
The $K^{-1}$ group in~\eqref{eq:sudamuon-final-rate-K-1-4-refined} collects
initial transients. Factoring out the explicit $\sqrt{r_{\max}}/\sqrt N$ and
using again $1\le 1/(1-\gamma)$ and $\lambda/(1-\lambda)\le 1/(1-\lambda)$, the
whole coefficient is bounded by a constant multiple of
$1/\bigl(\sqrt{N}(1-\gamma)(1-\lambda)\bigr)$, giving the fourth term
in~\eqref{eq:sudamuon-final-rate-bigO}.

\medskip\noindent\textbf{5.\ $K^{-7/4}$ term.}
The last group in~\eqref{eq:sudamuon-final-rate-K-1-4-refined} is already of the form
$\text{const}\cdot \lambda/\bigl((1-\lambda)(1-\gamma)\bigr)\,K^{-7/4}$.
Using $\lambda\le 1$, we upper-bound this by a constant times
$K^{-7/4}/\bigl((1-\gamma)(1-\lambda)\bigr)$, which yields the fifth term
in~\eqref{eq:sudamuon-final-rate-bigO}.

Combining these five estimates and enlarging $C$ if necessary, we obtain
\eqref{eq:sudamuon-final-rate-bigO}. The equivalent big-$\mathcal{O}$ expression
follows immediately.
\end{proof}

With the master bound established, we now draw two structural consequences. We first instantiate it for concrete SUDA communication backbones in Subsection~\ref{sec:discussion} and then show in Subsection~\ref{sec:necessity} that tracking before polarization is not a modular detail but a structural requirement. Section~\ref{sec:linear-speedup} then turns to a complementary counterexample on federated versus fully decentralized linear speedup.

\subsection{Modular Axis: Backbone Choices}\label{sec:discussion}

The convergence theory implies the first structural conclusion of the paper:
within our decomposition, the choice of decentralized communication backbone is
the modular axis.
To compare concrete backbone choices, it suffices to inspect how
Corollary~\ref{cor:sudamuon-final-convergence-bigO} depends on the SUDA
contraction rate $\gamma$. For convenience, we recall the big-$\mathcal{O}$
master bound:
\begin{equation}
\label{eq:discussion-master}
\begin{split}
\frac{1}{K}\sum_{k=0}^{K-1} \mathbb{E}\big[\|\nabla f(\bar X^k)\|_*\big]
\;\le\; C\Bigg(&
\Bigl(1+\tfrac{\sigma}{\sqrt{N}}\Bigr)K^{-1/4}
+\frac{K^{-3/4}}{(1-\gamma)(1-\lambda)}
\\
&+\Bigl(1+\frac{\sigma}{1-\lambda}\Bigr)K^{-1/2}
+\frac{K^{-1}}{\sqrt{N}\,(1-\gamma)(1-\lambda)}
\\
&+\frac{K^{-7/4}}{(1-\gamma)(1-\lambda)}
\Bigg),
\end{split}
\end{equation}

The dominant term is the same for every SUDA instantiation and is completely
independent of the graph. Different communication backbones appear only through
the factor $(1-\gamma)^{-1}$ in the lower-order network corrections, where
$\gamma$ is the SUDA contraction rate.

For the SUDA instantiations of theoretical interest here, the substitution is simple:
\[
\gamma=\lambda^2 \quad \text{for ED/D$^2$ and EXTRA}, \qquad
\gamma=\lambda \quad \text{for ATC-GT}.
\]

Therefore
\[
\frac{1}{(1-\gamma)(1-\lambda)}
= \mathcal{O}\!\left(\frac{1}{(1-\lambda)^2}\right)
\]
in all three cases. At the big-$\mathcal{O}$ level, ED/D$^2$, EXTRA, and
ATC-GT therefore all scale as $\mathcal{O}((1-\lambda)^{-2})$ in the
$K^{-3/4}$, $K^{-1}$, and $K^{-7/4}$ network-dependent corrections.

The difference lies in constants rather than order. ED/D$^2$ and EXTRA use
$1-\gamma=(1-\lambda)(1+\lambda)$, whereas ATC-GT uses $1-\gamma=1-\lambda$.
When $\lambda$ is close to one, ATC-GT therefore carries a larger prefactor
in the subdominant network terms, even though the asymptotic order is the
same. This is precisely the modularity provided by the SUDA formulation:
among convergent SUDA instantiations, the leading stochastic term is graph-free,
and the communication backbone changes only the transient network corrections
through the substitution for $\gamma$.

In short, the theory does not predict different dominant rates for ED/D$^2$,
EXTRA, and ATC-GT. It predicts the same leading order and different transient
behavior. This completes the modular part of the structural decomposition. The
next two subsections turn to the complementary non-modular part: first an
internal boundary concerning the order between tracking and polarization, and
then an external boundary comparing decentralized and federated Muon.

\subsection{Internal Boundary: Tracking Before Polarization}\label{sec:necessity}

The previous subsection identified the communication backbone as the modular
axis of decentralized Muon design. We now turn to the first non-modular
boundary, which is internal to the decentralized architecture: whether Muon
polarizes a tracked quantity or a purely local one. A key design choice in \sudamuon{} is that the Muon direction is computed from a
\emph{network-aggregated / tracked} quantity $H_i^{k+1}$ rather than directly from a
local stochastic gradient (or its EMA). The following proposition shows that if we
remove this pre-polarization tracking step while keeping the rest of the algorithm
unchanged, then one cannot, in general, guarantee \emph{exact} convergence to
stationarity for heterogeneous objectives. For clarity we present this internal obstruction in
the deterministic/noiseless subcase $\sigma=0$; this is \emph{not} an additional
assumption, since the stochastic-oracle model in Assumption~\ref{assump:stoch}
explicitly allows $\sigma=0$ (deterministic gradients), and any convergence guarantee
under that model must hold in particular for this simplest admissible case. In this
setting, we construct smooth and lower-bounded problems on which the averaged iterate
$\bar X^k$ gets stuck at a \emph{non-stationary} point.

\begin{proposition}[A non-stationary fixed point for the no-tracking variant (general matrix case)]
\label{prop:no-tracking-nonstationary-fixed-point}
For notational simplicity, consider the deterministic/noiseless setting $\sigma=0$ and
the general matrix case $m,n\ge 1$.
Define the \emph{no-tracking variant} of \sudamuon{} by keeping
Algorithm~\ref{alg:suda-muon} unchanged except that the gradient-tracking update is
removed and replaced by
\begin{equation}
\label{eq:no-tracking-Hi}
H_i^{k+1} \gets M_i^{k+1}\qquad \text{for all }i\in[N],\ k\ge 0.
\end{equation}
Assume the communication/SUDA matrices satisfy Assumption~\ref{assump:network-suda},
and initialize as in Algorithm~\ref{alg:suda-muon} with $Y_i^0=0$ and $M_i^0=G_i^0$.

Then there exist $N=2$ and local objectives $f_1,f_2:\mathbb{R}^{m\times n}\to\mathbb{R}$
that are $L_*$-smooth in the sense of Assumption~\ref{assump:smooth} and are lower
bounded, such that for the consensus initialization $X_1^0=X_2^0=X^\dagger$, the
averaged iterate satisfies
\[\bar X^k \equiv X^\dagger\quad \text{for all }k\ge 0,\]
while the global objective $f(X)\triangleq\tfrac12\bigl(f_1(X)+f_2(X)\bigr)$ has
\[\nabla f(X^\dagger)\neq 0.
\]
Consequently,
\[
\frac{1}{K}\sum_{k=0}^{K-1}\bigl\|\nabla f(\bar X^k)\bigr\|_* \equiv \bigl\|\nabla f(X^\dagger)\bigr\|_*>0,
\]
and hence the no-tracking variant \eqref{eq:no-tracking-Hi} cannot guarantee exact
stationarity convergence (vanishing averaged gradient nuclear norm at $\bar X^k$)
over the class of smooth heterogeneous objectives.
\end{proposition}

\begin{proof}
We first note that the averaging argument in the proof of
Lemma~\ref{lem:sudamuon-average-dynamics} uses only the $\mathbf X$- and $\mathbf Y$-updates
(and the doubly-stochasticity/polynomial properties of $A,B,C$), and does not rely on
the specific recursion defining $\mathbf H^{k+1}$. Hence the same argument applies to
the no-tracking variant as well and yields
\begin{equation}
\label{eq:pf-no-tracking-avg}
\bar X^{k+1}=\bar X^k-\alpha\,\bar S^{k+1},
\qquad
\bar S^{k+1}\triangleq \frac{1}{2}\sum_{i=1}^2 S_i^{k+1},\quad
S_i^{k+1}=\operatorname{msgn}(H_i^{k+1}).
\end{equation}
Under the no-tracking rule \eqref{eq:no-tracking-Hi}, we have $H_i^{k+1}=M_i^{k+1}$.
It remains to exhibit a smooth lower-bounded instance for which
$\bar S^{k+1}\equiv 0$ while $\nabla f(\bar X^0)\neq 0$.

\medskip\noindent\textbf{Construction of the counterexample.}
Choose $f_1,f_2$ as in Example~\ref{ex:matrix-logistic-counterexample} below and
initialize at consensus $\bar X^0=X^\dagger$.
For these functions, for every $X\in\mathbb{R}^{m\times n}$, the gradients satisfy
\begin{equation}
\label{eq:grad-aligned}
\nabla f_1(X)=c_1(X)\,U,\qquad \nabla f_2(X)=c_2(X)\,U,
\end{equation}
where $U\in\mathbb{R}^{m\times n}$ is a fixed rank-one matrix and the scalar
coefficients obey $c_1(X)>0$ and $c_2(X)<0$.
In particular, writing $U=uv^\top$ with $\|u\|_2=\|v\|_2=1$, for any scalar $\mu\neq 0$ we have
$\operatorname{msgn}(\mu U)=\mathrm{sign}(\mu)\,U$ because a reduced SVD of $\mu U$ is
$\mu U=u(|\mu|)(\mathrm{sign}(\mu)v)^\top$.
Therefore, $\operatorname{msgn}(\nabla f_1(X))=U$ and $\operatorname{msgn}(\nabla f_2(X))=-U$ for all $X$.

\medskip\noindent\textbf{Inductive cancellation argument.}
In the noiseless setting, $G_i^k=\nabla f_i(X_i^k)$.
Since $M_i^0=G_i^0$ and $M_i^{k+1}=\beta M_i^k+(1-\beta)G_i^k$, the alignment
\eqref{eq:grad-aligned} implies inductively that $M_1^{k}=\mu_1^{k}U$ and
$M_2^{k}=\mu_2^{k}U$ for scalars $\mu_1^{k}>0$ and $\mu_2^{k}<0$ for all $k$.
Therefore, under \eqref{eq:no-tracking-Hi}, we have
$S_1^{k+1}=\operatorname{msgn}(M_1^{k+1})=U$ and
$S_2^{k+1}=\operatorname{msgn}(M_2^{k+1})=-U$, and hence
$\bar S^{k+1}=(U-U)/2=0$ for all $k\ge 0$.
Plugging $\bar S^{k+1}=0$ into \eqref{eq:pf-no-tracking-avg} yields
$\bar X^{k+1}=\bar X^k$ for all $k$, so $\bar X^k\equiv \bar X^0=X^\dagger$.
Finally, Example~\ref{ex:matrix-logistic-counterexample} satisfies
$\nabla f(X^\dagger)\neq 0$, so the averaged stationarity measure stays bounded away
from zero.
\end{proof}

\begin{example}[A concrete smooth counterexample in the general matrix setting]
\label{ex:matrix-logistic-counterexample}
Let $N=2$.
Pick any unit vectors $u\in\mathbb{R}^m$ and $v\in\mathbb{R}^n$ with
$\|u\|_2=\|v\|_2=1$ and define the rank-one matrix
\[U\triangleq u v^\top.\]
Let $a>b>0$ be constants and define, for $X\in\mathbb{R}^{m\times n}$,
\[
 f_1(X)\triangleq a\,\log\bigl(1+e^{\langle U,X\rangle}\bigr),
 \qquad
 f_2(X)\triangleq b\,\log\bigl(1+e^{-\langle U,X\rangle}\bigr),
\]
where $\langle U,X\rangle\triangleq \mathrm{trace}(U^\top X)$.
Then $f_1$ and $f_2$ are lower bounded (both are nonnegative) and differentiable with
\[
\nabla f_1(X)= a\,\frac{1}{1+e^{-\langle U,X\rangle}}\,U,
\qquad
\nabla f_2(X)= -b\,\frac{1}{1+e^{\langle U,X\rangle}}\,U.
\]
In particular, $\nabla f_1(X)$ and $\nabla f_2(X)$ are always aligned with $U$ but have
opposite signs in their scalar coefficients; hence
$\operatorname{msgn}(\nabla f_1(X))=uv^\top$ and $\operatorname{msgn}(\nabla f_2(X))=-uv^\top$
for all $X$.

Moreover, the gradients are Lipschitz in the sense of Assumption~\ref{assump:smooth}.
Indeed, let $\varphi(t)\triangleq (1+e^{-t})^{-1}$ so that $\varphi'(t)=\varphi(t)(1-\varphi(t))\le 1/4$.
Since $\|U\|_* = 1$ (rank-one with unit singular value) and
$|\langle U, X-Y\rangle|\le \|U\|_*\,\|X-Y\|=\|X-Y\|$, we have
\[
\|\nabla f_1(X)-\nabla f_1(Y)\|_*
= a\,|\varphi(\langle U,X\rangle)-\varphi(\langle U,Y\rangle)|\,\|U\|_*
\le \frac{a}{4}\,|\langle U,X-Y\rangle|
\le \frac{a}{4}\,\|X-Y\|,
\]
and similarly $\|\nabla f_2(X)-\nabla f_2(Y)\|_*\le \frac{b}{4}\,\|X-Y\|$.
Thus Assumption~\ref{assump:smooth} holds with $L_*\triangleq \max\{a,b\}/4$.

Finally, with $X^\dagger\triangleq 0$, for the global objective
$f\triangleq \tfrac12(f_1+f_2)$ we have
\[
\nabla f(X^\dagger)=\tfrac12\bigl(\nabla f_1(0)+\nabla f_2(0)\bigr)
=\tfrac12\Bigl(\frac{a}{2}-\frac{b}{2}\Bigr)U
=\frac{a-b}{4}\,U\neq 0.
\]
Therefore, the hypotheses and conclusion of Proposition~\ref{prop:no-tracking-nonstationary-fixed-point}
are satisfied in the general matrix setting.
\end{example}

This structural result underscores the role of the gradient-tracking step in
Algorithm~\ref{alg:suda-muon}: because $\operatorname{msgn}$ is nonlinear,
replacing a tracked quantity by a purely local one can break exact convergence
under heterogeneous objectives. The next section asks a complementary question:
even when averaging is exact and topology is no obstacle, does decentralized
Muon retain the familiar federated linear-speedup effect?

\subsection{External Boundary: Federated vs.\ Decentralized Linear Speedup}
\label{sec:linear-speedup}

The previous subsection identified the first non-modular boundary, internal to
decentralized Muon itself: tracking-before-polarization is necessary for this natural no-tracking variant. We now turn
to a second, external structural boundary. Even if topology is made maximally
favorable, the fully decentralized update need not inherit the usual
linear-speedup intuition from server-based Muon. The reason is again the
ordering of operations. In a federated or server-based architecture, gradients
can be averaged first and polarized afterward. In a fully decentralized
architecture, each node polarizes a local noisy gradient and only then
communicates. The following toy example isolates this architectural difference
in the simplest possible setting.

We consider the vector case as a special instance of our matrix formulation,
namely $X^k=(x_1^k,x_2^k)^\top\in\mathbb{R}^{2\times 1}$. Let
\begin{equation}
\label{eq:speedup-f}
f(X)\triangleq \frac{1}{2}x_1^2,
\end{equation}

so that $\nabla f(X)=(x_1,0)^\top$. At iteration $k$, node $i\in[N]$ observes
the stochastic gradient
\begin{equation}
\label{eq:speedup-stochgrad}
G_i^k
= \nabla f(X^k) + \begin{bmatrix}0\\ \xi_i^k\end{bmatrix}
= \begin{bmatrix}x_1^k\\ \xi_i^k\end{bmatrix},
\qquad
\xi_i^k\in\{+\sigma,-\sigma\},
\end{equation}

where $\{\xi_i^k\}$ are i.i.d.\ across $i$ and $k$ with
$\E[\xi_i^k]=0$ and $\E[(\xi_i^k)^2]=\sigma^2$. The noise is transverse to the
signal: it acts only in the second coordinate.

Since $X^k\in\mathbb{R}^{2\times 1}$, the matrix-sign operator reduces to
Euclidean normalization,
\begin{equation}
\label{eq:speedup-msgn}
\operatorname{msgn}(V)=\frac{V}{\norm{V}_2},
\qquad
\operatorname{msgn}(0)=0.
\end{equation}

To give the decentralized method its best possible chance, we assume a complete
graph with exact averaging and synchronized iterates. The only distinction is
the order of averaging and polarization.
\begin{equation}
\label{eq:speedup-dec}
\textbf{DSGD-Muon:}\qquad
X^{k+1}=X^k-\alpha\cdot\frac{1}{N}\sum_{i=1}^N \operatorname{msgn}(G_i^k),
\end{equation}
\begin{equation}
\label{eq:speedup-fed}
\textbf{Centralized-Muon:}\qquad
X^{k+1}=X^k-\alpha\cdot\operatorname{msgn}\!\left(\frac{1}{N}\sum_{i=1}^N G_i^k\right).
\end{equation}

\begin{proposition}[Federated Muon can speed up, decentralized Muon cannot]
\label{prop:speedup-counterexample}

Consider \eqref{eq:speedup-f}--\eqref{eq:speedup-msgn} with any $\sigma>0$ and
stepsize $\alpha\in(0,\sigma)$.

\begin{enumerate}
\item 
\textbf{DSGD-Muon is insensitive to $N$.}
The first coordinate of \eqref{eq:speedup-dec} obeys the deterministic
recursion
\begin{equation}
\label{eq:speedup-recursion}
x_1^{k+1}
= x_1^k\left(1-\frac{\alpha}{\sqrt{(x_1^k)^2+\sigma^2}}\right).
\end{equation}

Hence
\[
|x_1^{k+1}| \ge \left(1-\frac{\alpha}{\sigma}\right)|x_1^k|,
\]
and reaching $|x_1^T|\le \varepsilon |x_1^0|$ requires
\[
T=\Omega\!\left(\frac{\sigma}{\alpha}\log\frac{1}{\varepsilon}\right),
\]
which does not improve with $N$.

\item 
\textbf{Centralized-Muon enjoys variance reduction.}
Let $\bar \xi^k\triangleq \frac{1}{N}\sum_{i=1}^N \xi_i^k$, so that
$\E[\bar \xi^k]=0$ and $\E[(\bar \xi^k)^2]=\sigma^2/N$.
Then the federated update \eqref{eq:speedup-fed} is driven by noise at scale
$\sigma/\sqrt{N}$.
In the noise-dominated regime $|x_1^k|\lesssim \sigma/\sqrt{N}$, the
corresponding iteration complexity scales as
\[
T=O\!\left(\frac{\sigma}{\alpha\sqrt{N}}\log\frac{1}{\varepsilon}\right),
\]
which is the usual linear-speedup behavior.

\end{enumerate}
\end{proposition}

\begin{proof}[Proof sketch]

For decentralized Muon, note that
\[
\norm{G_i^k}_2=\sqrt{(x_1^k)^2+(\xi_i^k)^2}=\sqrt{(x_1^k)^2+\sigma^2}.
\]
Therefore the first coordinate of the local Muon direction equals
$x_1^k/\sqrt{(x_1^k)^2+\sigma^2}$, which is independent of both the node index
$i$ and the sign of the noise. Averaging over $i$ yields
\eqref{eq:speedup-recursion}, and the lower bound follows from
$\sqrt{(x_1^k)^2+\sigma^2}\ge \sigma$.

For federated Muon, averaging occurs before normalization, so the second
coordinate becomes $\bar \xi^k$ with variance $\sigma^2/N$. The effective noise
level is therefore reduced from $\sigma$ to $\sigma/\sqrt{N}$, which gives the
stated complexity in the noise-dominated regime.

\end{proof}

This counterexample is useful precisely because topology plays no role in it:
the graph is complete and averaging is exact. The failure of linear speedup
comes solely from local polarization before averaging. In that sense, the
example isolates the external structural boundary of the paper: fully
decentralized Muon is not merely a communication-degraded version of federated
Muon, because local-polarize-then-average is structurally different from
average-then-polarize.

Figure~\ref{fig:speedup-counterexample} visualizes the same phenomenon for
$\sigma=50$ and $N\in\{1,8,64,256\}$. The dashed federated curves accelerate as
$N$ grows, while the solid decentralized curves nearly overlap across all
choices of $N$.

\begin{figure}[t]
\centering
\includegraphics[width=0.72\textwidth]{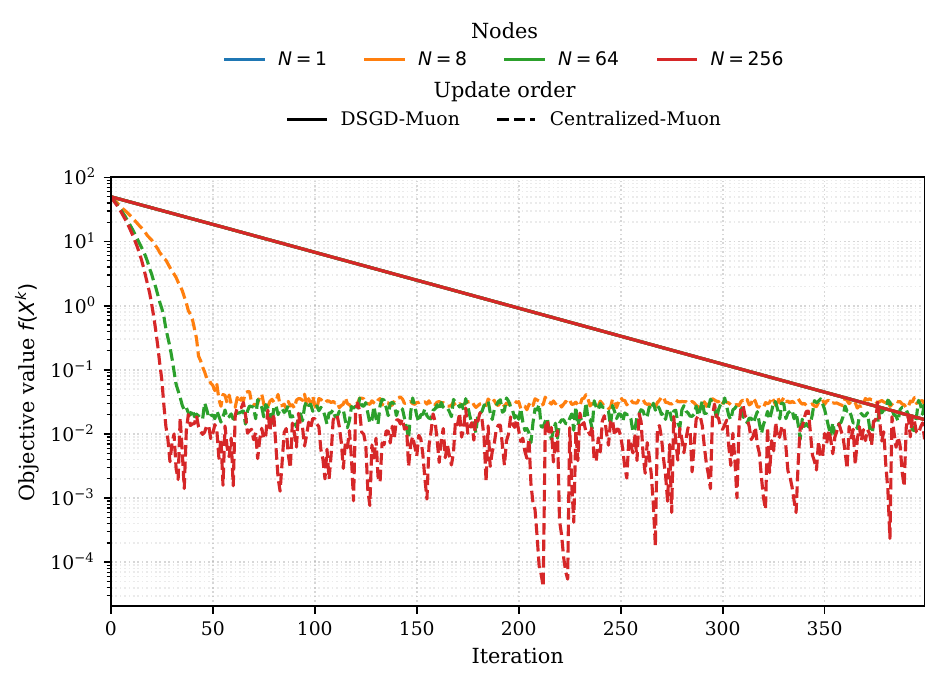}
\caption{Counterexample for linear speedup under transverse noise
($\sigma=50$). Dashed curves: federated Muon
\eqref{eq:speedup-fed} with server-side averaging before polarization. Solid
curves: decentralized Muon \eqref{eq:speedup-dec} with local polarization
before averaging. As $N$ increases, the federated curves accelerate clearly,
whereas the decentralized curves essentially overlap, indicating the absence of
linear speedup.}
\label{fig:speedup-counterexample}
\end{figure}

\section{Numerical Experiments}
\label{sec:numerical-experiments}

\subsection{Experimental Setup}
\label{subsec:exp-setup}

The experiments are designed to test the paper's structural claims rather than
only to rank methods. Their primary purpose is to probe the \emph{modular}
axis of our decomposition: once the Muon geometry is fixed, do different SUDA
backbones and topologies mainly affect transient behavior, and in which
regimes do those differences become visible? The experiments are not meant to
serve as direct ablations of the two non-modular boundaries established in
Subsections~\ref{sec:necessity} and~\ref{sec:linear-speedup}; those statements
are proved there by structural counterexamples. Instead, the empirical role of
this section is to show that decentralized Muon exhibits architecture- and
regime-dependent behavior that is consistent with that broader structural
picture. Throughout, we compare \textsc{SUDA--Muon-ED},
\textsc{SUDA--Muon-ATC-GT}, and \textsc{DeMuon}, and all quantitative
summaries are computed from pre-aggregated CSV logs produced by the training
scripts.

We use two benchmarks that stress different parts of the decentralized Muon
story: CIFAR-100 with strong cross-node heterogeneity, and GPT-2 fine-tuning
on Wikitext-2 with a much milder, near-IID partition.

For CIFAR-100, we use the standard $50{,}000/10{,}000$ train/test split, a
ResNet-18-style model with $100$ output classes, and a Dirichlet partition
with concentration $\alpha = 0.05$, which produces severe label skew across
workers. The main comparison uses $N=10$ nodes on a ring with mixing parameter
$\rho = 0.25$, batch size $512$ per node, constant stepsize $\eta = 0.03$,
EMA parameter $\beta = 0.9$, weight decay $5\times 10^{-4}$, and $25$
training epochs. Unless otherwise stated, each configuration is run with $3$
random seeds, and we report the performance of the averaged model.

We also include two CIFAR-100 extensions aimed at the communication story.
First, fixing \textsc{SUDA--Muon-ED}, we vary the topology and node count over
$N \in \{4,8,12\}$ to isolate how system scale changes the preferred graph.
Second, we run a larger and longer experiment with $N=20$ ring-connected nodes
for $100$ epochs using a slightly smaller constant stepsize $\eta = 0.01$.
This long-horizon setting is where steady-state communication effects should be
easiest to see.

For GPT-2/Wikitext-2, each node hosts the pretrained \texttt{gpt2} model. We
tokenize the corpus into length-$512$ sequences, use local batch size $16$ with base stepsize $\eta = 0.03$, and apply a
warmup-plus-cosine schedule. The main comparison uses a $15$-node ring, while
the ablation again varies topology and node count for
\textsc{SUDA--Muon-ED}. Each GPT configuration is run with $2$ seeds, and we
report validation perplexity at epoch $4$.

\subsection{Results and Analysis}
\label{subsec:results-analysis}

\textbf{CIFAR-100 main comparison (10-node ring).}
The short-horizon $10$-node ring benchmark is relatively balanced. After $25$ epochs, all three decentralized Muon variants finish within about $0.6$ accuracy points of one another, as shown in Table~\ref{tab:cifar-main}. \textsc{SUDA--Muon-ATC-GT} achieves the highest
mean final accuracy, while \textsc{DeMuon} is statistically very close and
\textsc{SUDA--Muon-ED} remains competitive despite the strongly non-IID split.
So the right reading of this table is not that one method wins decisively at a short budget; it is that the SUDA-based instantiations already match \textsc{DeMuon} in the default heterogeneous ring regime.

\begin{table}[t!]
    \centering
    \caption{CIFAR-100 top-1 test accuracy under a $10$-node ring topology
    with Dirichlet heterogeneity $\alpha = 0.05$. We report mean $\pm$
    population standard deviation over three seeds for the final accuracy at
    epoch $25$. Higher is better.}
    \label{tab:cifar-main}
    \begin{tabular}{lccccc}
        \toprule
        Method & \#Nodes & Topology & $\alpha$ & $\rho$ & Final acc (\%) \\
        \midrule
        \textsc{SUDA--Muon-ED}     & 10 & ring & 0.05 & 0.25 & $44.99 \pm 5.19$ \\
        \textsc{SUDA--Muon-ATC-GT} & 10 & ring & 0.05 & 0.25 & $\mathbf{45.56 \pm 3.18}$ \\
        \textsc{DeMuon}            & 10 & ring & 0.05 & 0.25 & $45.50 \pm 5.57$ \\
        \bottomrule
    \end{tabular}
\end{table}

The variance pattern is also informative. \textsc{SUDA--Muon-ATC-GT} has the
smallest spread across seeds, which is consistent with the intuition that
explicit correction of cross-node drift can stabilize training even when the
final means are close. In other words, the short-horizon experiment does not
yet separate the methods strongly by final accuracy, but it already suggests
that the SUDA coupling is not paying an optimization penalty for using the more
structured primal--dual backbone.

\textbf{CIFAR-100 ablation on topology and node count.}
Fixing the optimizer to \textsc{SUDA--Muon-ED}, Table~\ref{tab:cifar-ablation}
shows that the preferred topology changes with scale. At $N=4$, ring is best.
At $N=8$, star becomes strongest. At $N=12$, fully connected is best and the
line graph degrades the most. This is the clearest finite-time sign that graph
effects are real but problem dependent: the communication backbone is not a
mere implementation detail, yet neither is there a single topology that wins
uniformly.

\begin{table}[t!]
    \centering
    \caption{CIFAR-100 test accuracy for \textsc{SUDA--Muon-ED} under varying
    numbers of nodes and topologies. We report mean $\pm$ population standard
    deviation over three seeds of final and best test accuracy. Bold values
    indicate the best performance within each node-count group.}
    \label{tab:cifar-ablation}
    \begin{tabular}{cccc}
        \toprule
        \#Nodes & Topology & Final acc (\%) & Best acc (\%) \\
        \midrule
        4  & fully connected & $51.71 \pm 0.94$ & $53.71 \pm 0.57$ \\
        4  & line            & $51.76 \pm 1.04$ & $53.22 \pm 0.95$ \\
        4  & ring            & $\mathbf{53.23 \pm 0.30}$ & $\mathbf{55.27 \pm 1.53}$ \\
        4  & star            & $52.93 \pm 0.73$ & $53.21 \pm 0.33$ \\
        \midrule
        8  & fully connected & $47.69 \pm 0.55$ & $47.82 \pm 0.72$ \\
        8  & line            & $45.91 \pm 1.27$ & $46.07 \pm 1.03$ \\
        8  & ring            & $45.90 \pm 0.97$ & $47.04 \pm 0.69$ \\
        8  & star            & $\mathbf{50.19 \pm 1.60}$ & $\mathbf{50.94 \pm 0.98}$ \\
        \midrule
        12 & fully connected & $\mathbf{43.62 \pm 2.86}$ & $\mathbf{45.14 \pm 3.43}$ \\
        12 & line            & $37.26 \pm 5.32$ & $39.40 \pm 4.93$ \\
        12 & ring            & $41.72 \pm 3.73$ & $43.03 \pm 4.47$ \\
        12 & star            & $42.19 \pm 4.76$ & $42.19 \pm 4.76$ \\
        \bottomrule
    \end{tabular}
\end{table}

This ablation is useful for the paper's story because it isolates the modular
part of the design decomposition. Once the local Muon geometry is fixed,
changing only the graph still produces visibly different transient outcomes,
especially as $N$ grows. That is exactly the regime where the topology-dependent
correction terms should matter most in practice, even if the leading stochastic
term in the theory is topology separated.

\textbf{CIFAR-100 extended training on a 20-node ring.}
The strongest empirical evidence for the SUDA+Muon design appears in the larger
and longer non-IID run. Table~\ref{tab:cifar-long} and
Figure~\ref{fig:cifar-long} show that \textsc{SUDA--Muon-ED} clearly separates
from both alternatives after $100$ epochs on a $20$-node ring: it attains the
lowest final training loss and the highest final test accuracy, outperforming
\textsc{SUDA--Muon-ATC-GT} by more than $4$ points and \textsc{DeMuon} by
nearly $9$ points in final accuracy. This is the setting where the advantage
of the SUDA-style correction is no longer subtle.

\begin{table}[t!]
    \centering
    \caption{CIFAR-100 training loss and test accuracy after $100$ epochs on a
    $20$-node ring topology. We report mean $\pm$ population standard deviation
    over three seeds. Bold entries indicate the best values.}
    \label{tab:cifar-long}
    \resizebox{\textwidth}{!}{%
    \begin{tabular}{cccccc}
        \toprule
        Method & \#Nodes & Epochs & Final train loss & Final test acc (\%) & Seeds \\
        \midrule
        \textsc{SUDA--Muon-ED}     & 20 & 100 & $\mathbf{0.281 \pm 0.018}$ & $\mathbf{52.89 \pm 0.82}$ & 3 \\
        \textsc{SUDA--Muon-ATC-GT} & 20 & 100 & $0.329 \pm 0.008$ & $48.58 \pm 0.73$ & 3 \\
        \textsc{DeMuon}            & 20 & 100 & $0.386 \pm 0.009$ & $43.94 \pm 1.68$ & 3 \\
        \bottomrule
    \end{tabular}}
\end{table}

\begin{figure}[t!]
    \centering
    \includegraphics[width=\textwidth]{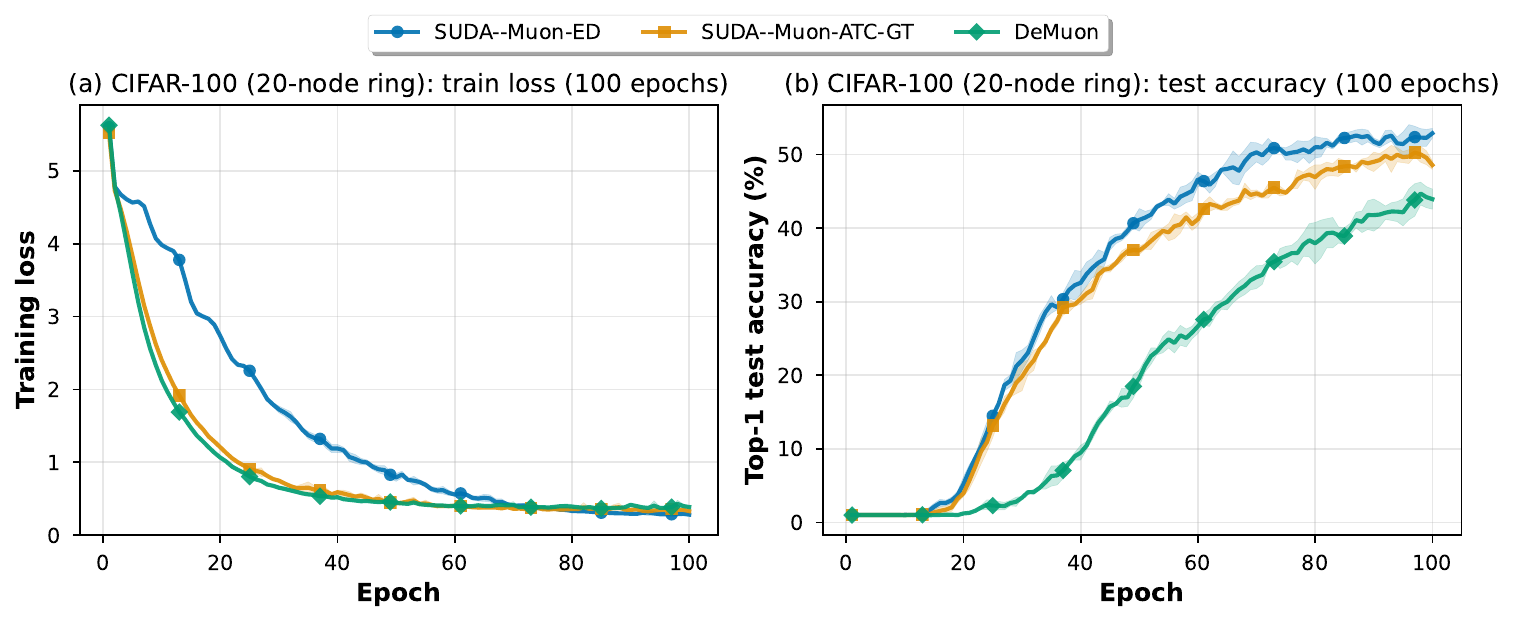}
    \caption{CIFAR-100 performance of the three decentralized Muon variants on
    a $20$-node ring over $100$ epochs. \textsc{SUDA--Muon-ED} converges faster
    and finishes higher than \textsc{SUDA--Muon-ATC-GT} and \textsc{DeMuon}.
    Shaded regions, where visible, denote one standard deviation over three
    seeds.}
    \label{fig:cifar-long}
\end{figure}

The trajectory plot sharpens the comparison. During the first half of training,
both SUDA-based variants already stay ahead of \textsc{DeMuon}; later on,
\textsc{SUDA--Muon-ED} continues improving and cleanly overtakes
\textsc{SUDA--Muon-ATC-GT}. So the most defensible claim is tied to this specific regime: in the demanding $20$-node, $100$-epoch non-IID setting, both SUDA-based variants outperform \textsc{DeMuon}, and among the instantiations tested here \textsc{SUDA--Muon-ED} is the strongest.

\textbf{GPT-2 / Wikitext-2 main comparison (15-node ring).}
The language-modeling picture is intentionally different. Under the $15$-node
ring, the three decentralized Muon variants are almost indistinguishable after
$4$ epochs, with validation perplexities clustered between $18.1396$ and
$18.1406$ in Table~\ref{tab:gpt-main}. We therefore do not read this benchmark
as evidence that one decentralized Muon mechanism dominates the others. The
useful message is instead that the SUDA-based variants remain fully competitive
in the benign near-IID regime; adopting the SUDA backbone does not sacrifice
performance when the communication problem is mild.

\begin{table}[t!]
    \centering
    \caption{Validation perplexity on Wikitext-2 for GPT-2 under a $15$-node
    ring topology after four training epochs. Values are means over two seeds.
    Lower is better; bold indicates the lowest perplexity.}
    \label{tab:gpt-main}
    \begin{tabular}{lcccc}
        \toprule
        Method & \#Nodes & Topology & Seeds & Eval ppl (4 epochs) \\
        \midrule
        \textsc{SUDA--Muon-ED}     & 15 & ring & 2 & $\mathbf{18.1396}$ \\
        \textsc{SUDA--Muon-ATC-GT} & 15 & ring & 2 & $18.1406$ \\
        \textsc{DeMuon}            & 15 & ring & 2 & $18.1401$ \\
        \bottomrule
    \end{tabular}
\end{table}

\begin{figure}[t!]
    \centering
    \begin{subfigure}[t]{0.24\textwidth}
        \centering
        \includegraphics[width=\textwidth]{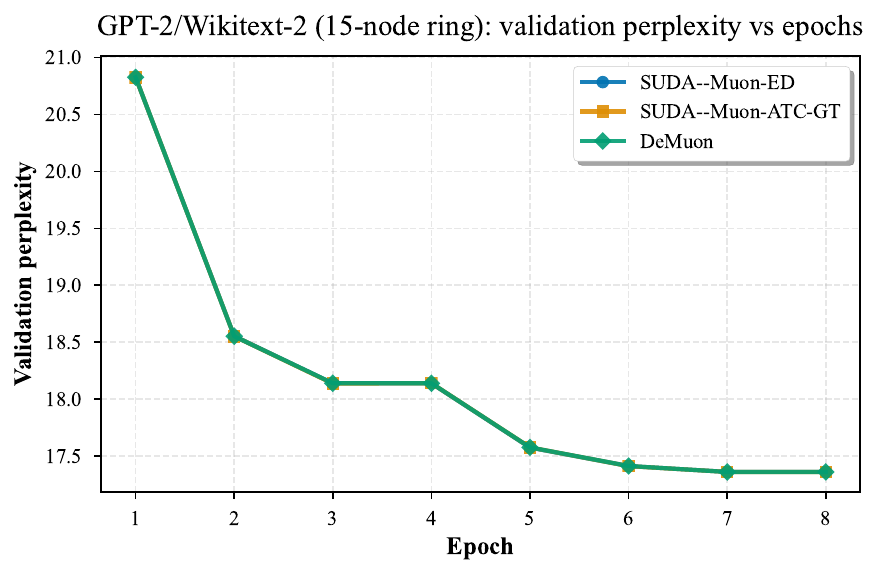}
        \caption{Main $15$-node ring run.}
        \label{fig:gpt-main}
    \end{subfigure}\hfill
    \begin{subfigure}[t]{0.74\textwidth}
        \centering
        \includegraphics[width=\textwidth]{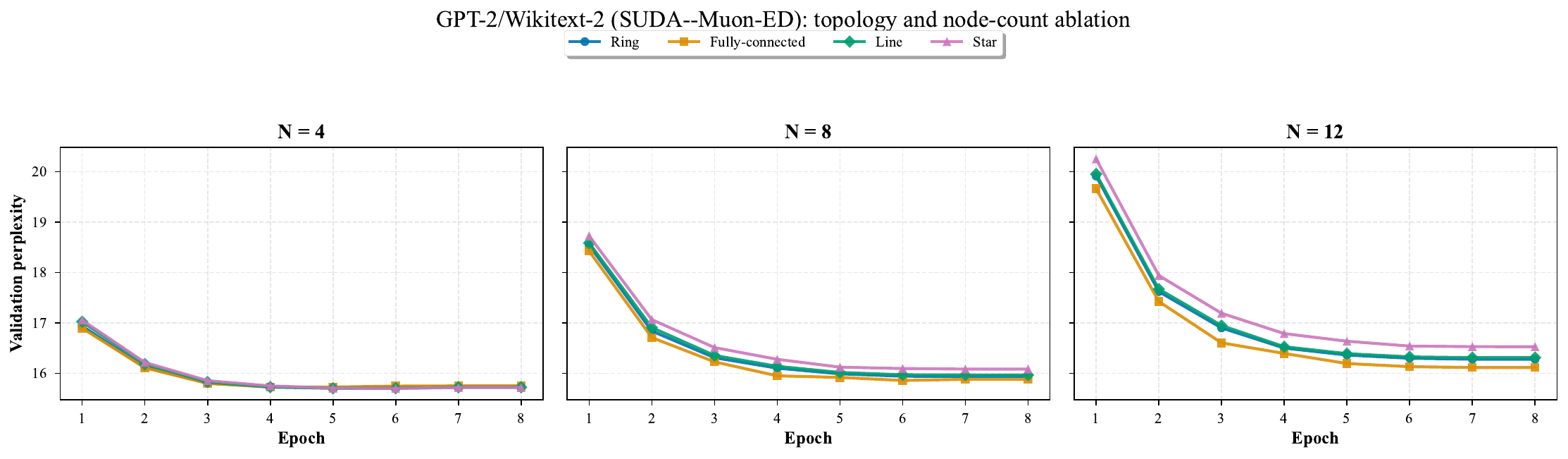}
        \caption{Ablation over node count and topology.}
        \label{fig:gpt-ablation-fig}
    \end{subfigure}
    \caption{Validation perplexity of decentralized Muon variants on GPT-2 /
    Wikitext-2. The left panel compares \textsc{SUDA--Muon-ED},
    \textsc{SUDA--Muon-ATC-GT}, and \textsc{DeMuon} under the main $15$-node
    ring configuration, while the right panel reports the
    \textsc{SUDA--Muon-ED} topology ablation. Shaded regions denote one
    standard deviation over two seeds.}
    \label{fig:gpt}
\end{figure}

\textbf{GPT-2 / Wikitext-2 ablation on topology and node count.}
Fixing the optimizer to \textsc{SUDA--Muon-ED}, Table~\ref{tab:gpt-ablation}
shows that all topologies remain close at $N=4$, while fully connected becomes
best at $N=8$ and $N=12$. The absolute spread is still small, so this is not a
large-effect benchmark, but it reinforces the same practical lesson as on
CIFAR-100: the graph interacts with task statistics and scale, and the
preferred topology can shift as those conditions change.

\begin{table}[t!]
    \centering
    \caption{Validation perplexity of \textsc{SUDA--Muon-ED} on Wikitext-2
    after four training epochs under varying numbers of nodes and topologies.
    Values are means over the available seeds. Lower is better; bold indicates
    the best performance within each node-count group.}
    \label{tab:gpt-ablation}
    \begin{tabular}{ccc}
        \toprule
        \#Nodes & Topology & Eval ppl (4 epochs) \\
        \midrule
        4  & fully connected & $15.75$ \\
        4  & line            & $15.73$ \\
        4  & ring            & $15.74$ \\
        4  & star            & $\mathbf{15.72}$ \\
        \midrule
        8  & fully connected & $\mathbf{15.88}$ \\
        8  & line            & $15.97$ \\
        8  & ring            & $15.94$ \\
        8  & star            & $16.09$ \\
        \midrule
        12 & fully connected & $\mathbf{16.12}$ \\
        12 & line            & $16.31$ \\
        12 & ring            & $16.28$ \\
        12 & star            & $16.53$ \\
        \bottomrule
    \end{tabular}
\end{table}

Taken together, these experiments mainly support the \emph{modular} axis of
the paper's story. When the Muon geometry is held fixed and only the backbone
or topology is varied, the differences are real but strongly regime dependent:
they are small in short-horizon or near-IID settings and much larger in the
long-horizon, strongly heterogeneous CIFAR-100 run. This is consistent with the
theoretical picture that backbone choices alter transient communication effects
and lower-order corrections more than the dominant stochastic term. The
two non-modular boundaries about tracking-before-polarization and
federated-versus-decentralized ordering should therefore be read as structural
limits established theoretically, with the experiments providing practical
context rather than direct empirical proofs.

\section{Conclusion}\label{sec:conclusion}

This paper studies one specific obstruction in fully decentralized Muon: the
polarization step is nonlinear, so it does not interact with gossip averaging
in the same way as an ordinary gradient step. Our main contribution is not
merely a new decentralized Muon variant, but a structural decomposition of
decentralized Muon design. In this decomposition, Muon supplies the matrix
geometry, while a unified primal--dual backbone handles communication
correction. \sudamuon{} realizes this decomposition through a SUDA-style
template that lets several decentralized correction mechanisms be compared
within one architecture.

The main theoretical message is organized around one modular axis and two
non-modular boundaries. The modular axis is the communication backbone: the
dominant stochastic term remains topology-free, while the graph enters only
through faster-decaying corrections, so changing ED/D$^2$, EXTRA, or ATC-GT
affects transient network corrections but not the leading stochastic order. The
first non-modular boundary is internal: tracking before polarization is not
optional in general. Without it, one can construct smooth heterogeneous
problems on which the averaged iterate stays at a non-stationary point.

A second non-modular boundary is external and appears when comparing federated and fully
decentralized Muon. In our linear-speedup counterexample, server-side averaging
before polarization recovers the usual $\sigma/\sqrt{N}$ variance reduction,
while fully decentralized local polarization before averaging does not, even on
a complete graph with exact averaging. Thus fully decentralized Muon is not
merely a communication-degraded version of federated Muon; it has its own
algorithmic structure.

Our experiments on CIFAR-100 and GPT-2/Wikitext-2 primarily support the
modular axis of this picture. Short-horizon and near-IID runs make the decentralized Muon variants look similar, but in the $20$-node, $100$-epoch non-IID CIFAR-100 setting both SUDA-based variants outperform \textsc{DeMuon}, with \textsc{SUDA--Muon-ED} performing best. When the partition is near IID, the
different SUDA--Muon instances behave much more similarly and the best
topology can shift. This regime dependence is consistent with the view that
backbone choices mainly change transient communication behavior, whereas the
two non-modular statements of the paper are established by the counterexamples
rather than by benchmark ranking alone.

Natural next steps include communication-efficient variants with compression,
sharper rates under additional structure such as PL-type conditions, online
tuning of the stepsize and EMA parameters, and larger-scale experiments that
test whether the same topology effects persist for bigger models and graphs.

\section*{Acknowledgments}
We gratefully acknowledge ReasFlow \citep{reasflowteam2026reasflow}, a reasoning-centric scientific discovery assistant, for its substantial contributions to the preparation of this paper. A significant portion of the work, including the literature review, mathematical proofs, numerical experiments, and the initial manuscript draft, was generated automatically with the assistance of ReasFlow. The authors’ contributions lay primarily in identifying the research problem, proposing the high-level algorithmic design, articulating the key ideas underlying the mathematical proofs, specifying the methodology and requirements for the numerical experiments, and polishing the manuscript to meet the standards required for submission. In particular, the authors devoted considerable effort to verifying the correctness of the mathematical proofs and refining the resulting arguments.

\bibliographystyle{plainnat}
\bibliography{reference}

\end{document}